\documentclass{article}
\usepackage[journal=JEMS,lang=british]{ems-journal-jems}

\usepackage{amsfonts,mathrsfs,verbatim,cases,fouridx}
\usepackage{mathtools}
\usepackage{tikz,graphicx}
\usetikzlibrary{hobby,math}
\usepackage{upgreek}

\allowdisplaybreaks[3]

\newtheorem{theorem}{Theorem}[section]
\newtheorem{lemma}[theorem]{Lemma}
\newtheorem{proposition}[theorem]{Proposition}
\newtheorem{corollary}[theorem]{Corollary}
\newtheorem{conjecture}[theorem]{Conjecture}

\theoremstyle{remark}
\newtheorem{remark}[theorem]{Remark}
\theoremstyle{definition}

\numberwithin{equation}{section}

\renewcommand{\leq}{\leqslant}
\renewcommand{\geq}{\geqslant}

\DeclareMathOperator{\Vir}{Vir}
\DeclareMathOperator{\sgn}{sgn}
\DeclareMathOperator{\ch}{ch}
\DeclareMathOperator{\lev}{lev}
\DeclareMathOperator{\eup}{e}

\newcommand{\qbin}[2]{\genfrac{[}{]}{0pt}{}{#1}{#2}}
\newcommand{\abs}[1]{\vert#1\vert}
\newcommand{\la}{\lambda}
\newcommand{\La}{\Lambda}

\newcommand{\Par}{\mathscr{P}}
\newcommand{\qHyp}[5]{\fourIdx{}{#1}{}{#2}\phi
                      \bigg[\genfrac{}{}{0pt}{}{#3}{#4};#5\bigg]}
\newcommand{\CSA}{\mathfrak{h}}

\newcommand{\CSAa}{\hat{\mathfrak{h}}}
\newcommand{\dCSAa}{\hat{\mathfrak{h}}^{\ast}}
\newcommand{\bil}[2]{(#1\vert#2)}
\newcommand{\pairing}[2]{\langle#1,#2\rangle}

\begin{document}

\title{An $\mathrm{A}_2$ Bailey tree and $\mathrm{A}_2^{(1)}$
Rogers--Ramanujan-type identities}
\titlemark{An $\mathrm{A}_2$ Bailey tree}

\emsauthor{1}{
	\givenname{S. Ole}
	\surname{Warnaar}
	\mrid{269674}
	\orcid{0000-0002-9786-0175}}{S.~O.~Warnaar}

\Emsaffil{1}{
	\pretext{}
	\department{School of Mathematics and Physics}
	\organisation{The University of Queensland}
	\city{Brisbane}
	\zip{QLD 4072}
	\country{Australia}
	\affemail{o.warnaar@maths.uq.edu.au}}
	


\begin{abstract}
The $\mathrm{A}_2$ Bailey chain of Andrews, Schilling and the
author is extended to a four-parameter $\mathrm{A}_2$ Bailey tree.
As main application of this tree, we prove the Kanade--Russell
conjecture for a three-parameter family of Rogers--Ramanujan-type
identities related to the principal characters of the affine Lie
algebra $\mathrm{A}_2^{(1)}$.
Combined with known $q$-series results, this further implies an
$\mathrm{A}_2^{(1)}$-analogue of the celebrated Andrews--Gordon
$q$-series identities.  
We also use the $\mathrm{A}_2$ Bailey tree to prove a
Rogers--Selberg-type identity for the characters of the principal
subspaces of $\mathrm{A}_2^{(1)}$ indexed by arbitrary level-$k$
dominant integral weights $\uplambda$.
This generalises a result of Feigin, Feigin, Jimbo, Miwa and
Mukhin for $\uplambda=k\La_0$.
\end{abstract}

\maketitle


\section{Introduction}
Let $(a;q)_{\infty}:=(1-a)(1-aq)\cdots$ and
$(a;q)_n:=(a;q)_{\infty}/(aq^n;q)_{\infty}$ for $n$ an integer.
In particular, $(a;q)_0=1$, $(a;q)_n=(1-a)(1-aq)\cdots(1-aq^{n-1})$
for $n>0$ and $1/(q;q)_n=0$ for $n<0$.
Further let $a,k,\tau$ be integers such that 
$k\geq 1$, $0\leq a\leq k$, $\tau\in\{0,1\}$, and fix $K:=2k+\tau+2$.
Then the modulus-$K$ Andrews--Gordon--Bressoud $q$-series identities
are given by
\begin{align}\label{Eq_AGB}
&\sum_{\la_1\geq\cdots\geq\la_k\geq 0}
\frac{q^{\la_1^2+\dots+\la_k^2+\la_{a+1}+\dots+\la_k}}
{(q;q)_{\la_1-\la_2}\cdots(q;q)_{\la_{k-1}-\la_k}
(q^{2-\tau};q^{2-\tau})_{\la_k}} \\
&\qquad\qquad=\frac{(q^{a+1};q^K)_{\infty}(q^{K-a-1};q^K)_{\infty}
(q^K;q^K)_{\infty}}{(q;q)_{\infty}} \notag,
\end{align}
where $\tau=1$ corresponds to the Andrews--Gordon or odd modulus case
\cite{Andrews74} and $\tau=0$ to the Bressoud or even modulus
case~\cite{Bressoud80}.
The Andrews--Gordon identities for $k=1$ simplify to the famous
Rogers--Ramanujan identities~\cite{Rogers94,Rogers17,RR19}
\begin{subequations}\label{Eq_RR}
\begin{align}\label{Eq_RR1}
\sum_{n=0}^{\infty}
\frac{q^{n^2}}{(1-q)(1-q^2)\cdots(1-q^n)}
=\prod_{n=0}^{\infty}\frac{1}{(1-q^{5n+1})(1-q^{5n+4})},
\shortintertext{and}
\label{Eq_RR2}
\sum_{n=0}^{\infty}
\frac{q^{n^2+n}}{(1-q)(1-q^2)\cdots(1-q^n)}
=\prod_{n=0}^{\infty}\frac{1}{(1-q^{5n+2})(1-q^{5n+3})}.
\end{align}
\end{subequations}
These identities and their generalisations due to Andrews, Gordon
and Bressoud have a rich history.
They are the analytic counterpart of well-known theorems for integer
partitions \cite{Bressoud79,Bressoud80,Gordon61,MacMahon04,Schur17},
have numerous important interpretations in terms of the representation
theory of affine Lie algebras and vertex operator algebras
\cite{DHK21,CLM06,CF13,FF93,GOW16,KP18,Lepowsky82,LW81a,LW81b,LW82,LW84,MP87,MP99,SF94},
and have arisen in a variety of other contexts such as in algebraic geometry
\cite{BMS13,Mourtada25}, combinatorics \cite{Corteel17,FW16},
commutative algebra \cite{ADJM21,BGK20,MM23}, group theory \cite{Cohen85},
knot theory \cite{AD11,Hikami03},
number theory \cite{BCFK14,MMO08}, statistical mechanics
\cite{BA86,BMcC98,W97}, the theory of orthogonal polynomials
\cite{GIS99, IS03}, and symmetric function 
theory~\cite{BW15,IJZ06,Stembridge90,RW21}. 
For a comprehensive introduction to the Rogers--Ramanujan identities
and their generalisations we refer the reader to
\textit{An invitation to the Rogers--Ramanujan identities},
by Sills~\cite{Sills18}.

The representation-theoretic interpretations of the
Andrews--Gordon--Bressoud identities based on the affine Lie
algebra $\mathrm{A}_1^{(1)}$ make it a natural problem to
try to extend \eqref{Eq_AGB} to $\mathrm{A}_{r-1}^{(1)}$.
Despite the long history of the subject, this is very much an
open problem.
In 1999 Andrews, Schilling and the author succeeded in finding (some)
analogues of \eqref{Eq_AGB} for $\mathrm{A}_2^{(1)}$ for all
moduli~\cite{ASW99}.
To succinctly describe these results, we require the modified theta
functions $\theta(z;q):=(z;q)_{\infty}(q/z;q)_{\infty}$ and
$\theta(z_1,\dots,z_n;q):=\theta(z_1;q)\cdots\theta(z_n;q)$, and the
$q$-binomial coefficients
\[
\qbin{n}{m}=\qbin{n}{m}_q:=\frac{(q;q)_n}{(q;q)_m(q;q)_{n-m}}
\]
for integers $n,m$ such that $0\leq m\leq n$ and zero otherwise.
We also need the appropriate $\mathrm{A}_2^{(1)}$-analogue of
$1/(q^{2-\tau};q^{2-\tau})_n$ (which occurs in \eqref{Eq_AGB} with
$n=\la_k$), and for $n,m$ nonnegative integers and
$\tau\in\{-1,0,1\}$, we define
\begin{equation}\label{Eq_g}
g_{n,m;\tau}(q):=\frac{q^{\tau(\tau-1)nm}}
{(q;q)_{n+m}(q^2;q)_{n+m}}\qbin{n+m}{n}_p,
\end{equation}
where $p=q$ if $\tau^2=1$ and $p=q^3$ if $\tau=0$.
Thus, in the simplest and perhaps most important case,
$g_{n,m;1}(q)=1/((q;q)_n(q;q)_m(q^2;q)_{n+m})$.
Then, for $a,k,\tau$ integers such that $k\geq 1$, $0\leq a\leq k$
and $\tau\in\{-1,0,1\}$, it was shown in \cite{ASW99} that
\begin{align}\label{Eq_ASW}
&\sum_{\substack{\la_1\geq\cdots\geq\la_k\geq 0 \\[1pt] 
\mu_1\geq\cdots\geq\mu_k\geq 0}} 
\frac{1-q^{\la_a+\mu_a+1}}{1-q}\,
\frac{q^{\sum_{i=1}^k(\la_i^2-\la_i\mu_i+\mu_i^2)+
\sum_{i=a+1}^k(\la_i+\mu_i)}}
{\prod_{i=1}^{k-1} (q;q)_{\la_i-\la_{i+1}}(q;q)_{\mu_i-\mu_{i+1}}} \,
g_{\la_k,\mu_k;\tau}(q) \\
&\qquad\qquad=\frac{(q^K;q^K)_{\infty}^2}{(q;q)_{\infty}^3}\,
\theta\big(q^{a+1},q^{a+1},q^{2a+2};q^K\big), \notag
\end{align}
where $K:=3k+\tau+3$ and $q^{\la_0}=q^{\mu_0}:=0$.
From a $q$-series as well as combinatorial point of view this is a 
perfectly good analogue of \eqref{Eq_AGB}.
For example, by the Borodin product formula \cite{Borodin07}, the
right-hand side corresponds to the generating function of cylindric
partitions \cite{GK97} of three rows with `profile' given by $(K-2a-3,a,a)$.
If, however, one wishes to interpret \eqref{Eq_ASW} as an
identity for the principal characters of $\mathrm{A}_2^{(1)}$
(characters of the principally graded subspaces of basic
$\mathrm{A}_2^{(1)}$ modules in the sense of \cite{Frenkel82,LW82})
or, for $3\nmid K$, as branching functions of $\mathrm{A}_2^{(1)}$
and characters of the $\mathcal{W}_3(3,K)$ vertex operator algebra
(see \cite[Section 4]{W23}), then one should multiply both sides
of \eqref{Eq_ASW} by $(q;q)_{\infty}$.\footnote{The result
\eqref{Eq_ASW} may be interpreted as an identity for the principally
specialised characters of $\widehat{\mathfrak{gl}(3)}$ indexed by 
$(K-2a-3)\Lambda_0+a(\Lambda_1+\Lambda_2)$ for $0\leq a\leq k$,
see e.g.,~\cite{Frenkel82,Tingley08}.
This, however, does not match the interpretation of the 
Andrews--Gordon--Bressoud identities as character identities for
the principal characters of
$\widehat{\mathfrak{sl}(2)}=\mathrm{A}_1^{(1)}$.}
This would obscure the positivity of the left-hand side,
and for this reason we will not view the above as the ``proper''
$\mathrm{A}_2^{(1)}$-analogues of the Andrews--Gordon--Bressoud
identities.
Instead we follow Kanade and Russell \cite{KR23} and refer to 
\eqref{Eq_ASW} as the Andrews--Schilling--Warnaar identities,
or ASW identities for short.
From both a representation theoretic and cylindric partition point
of view it is clear that the above set of ASW identities is
not complete, and there should be an appropriate multisum expression
for each dominant integral weight $(K-a-b-3)\La_0+a\La_1+b\La_2$ of
$\mathrm{A}_2^{(1)}$ or each cylindric-partition profile $(K-a-b-3,a,b)$,
with corresponding product form as above but with theta function
given by $\theta(q^{a+1},q^{b+1},q^{a+b+2};q^K)$.
Recently Kanade and Russell \cite[Conjecture~5.1]{KR23} (see also 
\cite{Kanade23}) posed the following beautiful conjecture that covers
all cases for which $0\leq a,b\leq k$.

\begin{conjecture}[Kanade--Russell]\label{Con_KR-ASW}
Let $a,b,k$ be integers such that $0\leq a,b\leq k$,
and let $K:=3k+\tau+3$ for $\tau\in\{-1,0,1\}$.
Then
\begin{align}\label{Eq_KR-ASW}
&\sum_{\substack{\la_1\geq\cdots\geq\la_k\geq 0 \\[1pt] 
\mu_1\geq\cdots\geq\mu_k\geq 0}}
\frac{1-q^{\la_a+\mu_b+1}}{1-q}\,
\frac{q^{\sum_{i=1}^k(\la_i^2-\la_i\mu_i+\mu_i^2)+
\sum_{i=a+1}^k\la_i+\sum_{i=b+1}^k\mu_i}}
{\prod_{i=1}^{k-1} (q;q)_{\la_i-\la_{i+1}}(q;q)_{\mu_i-\mu_{i+1}}} \,
g_{\la_k,\mu_k;\tau}(q) \\
&\qquad\qquad=\frac{(q^K;q^K)_{\infty}^2}{(q;q)_{\infty}^3}\,
\theta\big(q^{a+1},q^{b+1},q^{a+b+2};q^K\big), \notag
\end{align}
where $q^{\la_0}=q^{\mu_0}:=0$.
\end{conjecture}

For $b=0$ and $\tau^2=1$ this was previously conjectured
in~\cite[Conjecture 7.4]{W23}.
By symmetry in $a$ and $b$, there are $\binom{k+2}{2}$ distinct
identities for fixed $k$, where it is noted that the right-hand
sides for $(a,b)=(k,k)$ and $(a,b)=(k,k-1)$ are the same if $\tau=-1$
due to the simple relation $\theta(z;q)=\theta(q/z;q)$.
In the following we may thus without loss of generality assume that
$a\geq b$.
For $\tau=-1$ the sum over $\mu_k$ can be carried out by a limiting
case of the $q$-Chu--Vandermonde summation (see e.g., \eqref{Eq_1phi1}
below), resulting in the slightly simpler
\begin{align*}
&\sum_{\substack{\la_1\geq\cdots\geq\la_k\geq 0 \\[1pt] 
\mu_1\geq\cdots\geq\mu_{k-1}\geq 0}}
\frac{1-q^{\la_a+\mu_b+1}}{1-q}\,
\frac{q^{\sum_{i=1}^k(\la_i^2-\la_i\mu_i+\mu_i^2)+
\sum_{i=a+1}^k\la_i+\sum_{i=b+1}^{k-1}\mu_i}}
{(q^2;q)_{\la_k+\mu_{k-1}}\prod_{i=1}^k (q;q)_{\la_i-\la_{i+1}}
\prod_{i=1}^{k-1} (q;q)_{\mu_i-\mu_{i+1}}} \\
&\qquad\qquad
=\frac{(q^{2k+2};q^{2k+2})_{\infty}^2}{(q;q)_{\infty}^3}\,
\theta\big(q^{a+1},q^{b+1},q^{a+b+2};q^{3k+2}\big), \notag
\end{align*}
where $0\leq b\leq a\leq k$ ($b\neq k$), $\la_{k+1}=\mu_k:=0$, and, 
for $k=1$, $\mu_0:=\infty$.

Besides \eqref{Eq_ASW}, also the $(a,b)=(k,0)$ and $(k-1,0)$ instances
of \eqref{Eq_KR-ASW} for $\tau^2=1$ were proved in~\cite{ASW99}.
For the moduli $5$ and $7$ this covers all identities in \eqref{Eq_KR-ASW}.
The identity of smallest modulus missing from \cite{ASW99} corresponds
to $(a,b,k,\tau)=(1,0,1,0)$ which has modulus~$6$.
Kanade and Russell proved this by solving the Corteel--Welsh system
of functional equations \cite{CW19} for cylindric partitions of profile 
$(d-a-b,a,b)$ for $d=3$, see~\cite[Corollary 7.5]{KR23}.
For the moduli $8$ and $10$ they again solved the corresponding
Corteel--Welsh systems (in these cases $d=5$ and $d=7$ respectively)
confirming the conjecture.
Alternatively, the modulus-$8$ case is implied by combining the
recent results of Corteel--Dousse--Uncu \cite{CDU22} and the author
\cite{W23} on modulus-$8$ Rogers--Ramanujan-type identities for
$\mathrm{A}_2^{(1)}$.
Finally, Uncu \cite[Theorems 4.4 \& 5.4]{Uncu23} settled the moduli 
$11$ and $13$ by algorithmically confirming and complementing a 
conjectured partial solution to the Corteel--Welsh system due to Kanade
and Russell.

The first main result of this paper is a case-free proof of the
Kanade--Russell conjecture for arbitrary modulus.

\begin{theorem}\label{Thm_KRistrue}
The Kanade--Russell conjecture holds for all moduli.
\end{theorem}

The three cases of smallest modulus not previously proved in the
literature are $k=2$, $\tau=0$ and $(a,b)\in\{(1,0),(2,0),(2,1)\}$.
For example, for $(a,b)=(2,0)$ the theorem confirms the modulus-$9$
identity
\begin{align*}
&\sum_{\la_1,\la_2,\mu_1,\mu_2=0}^{\infty}
\frac{q^{\la_1^2-\la_1\mu_1+\mu_1^2+\la_2^2-\la_2\mu_2+\mu_2^2+\mu_1+\mu_2}
(q^3;q^3)_{\la_2+\mu_2}}
{(q;q)_{\la_1-\la_2}(q;q)_{\mu_1-\mu_2}(q^3;q^3)_{\la_2}(q^3;q^3)_{\mu_2}
(q;q)_{\la_2+\mu_2}(q;q)_{\la_2+\mu_2+1}} \\
&\qquad\qquad=\prod_{n=1}^{\infty}
\frac{(1-q^{9n})}{(1-q^{n})^2(1-q^{9n-7})(1-q^{9n-2})},
\end{align*}
where we recall that $1/(q;q)_n=0$ if $n$ is a negative integer,
so that the summand vanishes unless $\la_1\geq\la_2$ and $\mu_1\geq\mu_2$.

As mentioned above, from a representation theoretic point of view
the ASW identities should be multiplied by a factor $(q;q)_{\infty}$.
For $\tau^2=1$ this factor can be absorbed in the multisum using a
transformation formula from~\cite{W23}.
This gives what we view as the Andrews--Gordon identities for
$\mathrm{A}_2^{(1)}$. 
In full generality this result is too involved to be stated in the
introduction and below we restrict ourselves to the special case $b=0$.
For the full result the reader is referred to Theorems~\ref{Thm_3k4} 
and \ref{Thm_3k2}.

\begin{theorem}[$\mathrm{A}_2^{(1)}$ Andrews--Gordon identities; $b=0$ case]
Let $a,k$ be integers such that $0\leq a\leq k$.
Then
\begin{subequations}
\begin{align}\label{Eq_b3kplus2}
&\sum_{\substack{\la_1,\dots,\la_k\geq 0 \\[1pt]
\mu_1,\dots,\mu_{k-1}\geq 0}}
\frac{q^{\la_k^2+\sum_{i=a+1}^k \la_i}}{(q;q)_{\la_1}} 
\prod_{i=1}^{k-1} q^{\la_i^2-\la_i\mu_i+\mu_i^2+\mu_i} 
\qbin{\la_i}{\la_{i+1}} 
\qbin{\la_i-\la_{i+1}+\mu_{i+1}}{\mu_i} \\
&\qquad\qquad=\frac{(q^K;q^K)_{\infty}^2}{(q;q)_{\infty}^2}\,
\theta\big(q,q^{a+1},q^{a+2};q^K\big), \notag
\end{align}
where $\mu_k:=2\la_k$ and $K:=3k+2$, and
\begin{align}\label{Eq_b3kplus4}
&\sum_{\substack{\la_1,\dots,\la_k\geq 0 \\[1pt] \mu_1,\dots,\mu_k\geq 0}}
\frac{q^{\sum_{i=a+1}^k \la_i}}{(q;q)_{\la_1}}
\prod_{i=1}^k q^{\la_i^2-\la_i\mu_i+\mu_i^2+\mu_i}
\qbin{\la_i}{\la_{i+1}} \qbin{\la_i-\la_{i+1}+\mu_{i+1}}{\mu_i} \\
&\qquad\qquad=\frac{(q^K;q^K)_{\infty}^2}{(q;q)_{\infty}^2}\,
\theta\big(q,q^{a+1},q^{a+2};q^K\big), \notag
\end{align}
\end{subequations}
where $\la_{k+1}:=0$, $\mu_{k+1}:=\la_k$ and $K:=3k+4$.
\end{theorem}

These results were conditionally proved in \cite{W23} assuming the
truth of \eqref{Eq_KR-ASW} for $b=0$ and $\tau^2=1$.
The $q$-series in \eqref{Eq_b3kplus2} and \eqref{Eq_b3kplus4} correspond
to the principal characters of the $\mathrm{A}_2^{(1)}$-highest weight module
$L((K-a)\Lambda_0+a\Lambda_1)$ for $K=3k+2$ and $K=3k+4$, respectively.
Alternatively, they may be recognised as the normalised 
characters of the $\mathcal{W}_3(3,K)$ vertex operator algebra 
of conformal weight $a(a+3)/K-a$.

One of the most streamlined proofs of the Andrews--Gordon--Bressoud
identities \eqref{Eq_AGB} is based on what is known as the Bailey
lattice \cite{AAB87}, which is a generalisation of the well-known
Bailey chain~\cite{Andrews84}.
Our proof of Theorem~\ref{Thm_KRistrue} presented in
Section~\ref{Sec_KR} is based on an $\mathrm{A}_2$-analogue of
a special case of the Bailey lattice which, due to its tree-like
structure, we refer to as the $\mathrm{A}_2$ Bailey tree.
A single branch of the $\mathrm{A}_2$ Bailey tree corresponds to
the $\mathrm{A}_2$ Bailey chain developed in \cite{ASW99} to prove 
the ASW identities \eqref{Eq_KR-ASW}.
Andrews' original proof of the Andrews--Gordon identities
\cite{Andrews74} predates the discoveries of the Bailey chain and
Bailey lattice, and instead is based on recursion relations for the
Rogers--Selberg function $Q_{k,i}(z;q)$ defined by~\cite{RR19,Selberg36}
\begin{equation}\label{Eq_RS}
Q_{k,i}(z;q):=\frac{1}{(zq;q)_{\infty}} 
\sum_{n=0}^{\infty} \big(1-z^iq^{(2n+1)i}\big) 
(-1)^n z^{kn} q^{(2k+1)\binom{n+1}{2}-in} \,
\frac{(zq;q)_n}{(q;q)_n},
\end{equation}
for integers $i,k$ such that $1\leq i\leq k$.
These recursions were solved by Andrews to give the multisum
representation~\cite[Equation (2.5)]{Andrews74}
\begin{equation}\label{Eq_RS-multi}
Q_{k,i}(z;q)=\sum_{\la_1\geq\cdots\geq\la_{k-1}\geq 0}
\frac{z^{\la_1+\dots+\la_{k-1}} 
q^{\la_1^2+\dots+\la_{k-1}+\la_i+\dots+\la_{k-1}}}
{(q;q)_{\la_1-\la_2}\cdots(q;q)_{\la_{k-2}-\la_{k-1}}(q;q)_{\la_{k-1}}}.
\end{equation}
Equating the two expressions for $Q_{k,i}$, specialising $z=1$ and
using the Jacobi-triple product identity yields \eqref{Eq_AGB} with
$(a,k)\mapsto (i-1,k-1)$ and $\tau=1$.
The equality of \eqref{Eq_RS} and \eqref{Eq_RS-multi} may also be
proved by the Bailey lattice, and by lifting this proof to the
$\mathrm{A}_2$-setting we obtain the following identity for the
character of the level-$k$ principal subspace $W_{\la}$ of
$\mathrm{A}_2^{(1)}$ indexed by $\la=(k-a-b)\La_0+a\La_1+b\La_2$
(see Section~\ref{Sec_PS} for details). 
Let $Q_{+}:=\{y=(y_1,y_2,y_3)\in\mathbb{Z}^3:~
y_1+y_2+y_3=0,~y_1\geq 0,~y_1+y_2\geq 0\}$.

\begin{theorem}\label{Thm_Principal-subspace}
For $a,b,k$ integers such that $0\leq a,b\leq k$, let $\nu$ be the 
strict partition $\nu:=(a+b+2,b+1,0)$.
Then
\begin{align*}
&\sum_{\substack{\la_1\geq\cdots\geq\la_k\geq 0 \\[1pt] 
\mu_1\geq\cdots\geq\mu_k\geq 0}}
\Big(1-\frac{x_1}{x_3}\,q^{\la_a+\mu_b-1}\Big)
\prod_{i=1}^k \frac{
\big(\frac{x_1}{x_2}\big)^{\la_i} 
\big(\frac{x_2}{x_3}\big)^{\mu_i} 
q^{\la_i^2-\la_i\mu_i+\mu_i^2-\chi(i\leq a)\la_i-\chi(i\leq b)\mu_i}}
{(q;q)_{\la_i-\la_{i+1}}(q;q)_{\mu_i-\mu_{i+1}}} \\
&\quad\quad = \sum_{y\in Q_{+}}
\frac{\det_{1\leq i,j\leq 3}\big((x_iq^{y_i})^{\nu_i-\nu_j}\big)}
{\prod_{1\leq i<j\leq 3} (x_i/x_j;q)_{\infty}}
\prod_{i=1}^3 \frac{x_i^{(k+2)y_i}
q^{(k+2)\binom{y_i}{2}-\nu_iy_i}(x_i/x_3;q)_{y_i}}{(qx_i/x_1;q)_{y_i}},
\end{align*}
where $q^{\la_0}=q^{\mu_0}=\la_{k+1}=\mu_{k+1}:=0$.
\end{theorem}

Setting $(x_1,x_2,x_3)=(zw,w,1)$ and letting $w$ tend to $0$, the summand
on the left vanishes unless $\mu_1=\dots=\mu_k=0$, resulting in 
$Q_{k+1,a+1}(z/q;q)$ in its multisum representation \eqref{Eq_RS-multi}.
In this same limit the summand on the right vanishes unless 
$(y_1,y_2,y_3)\in Q_{+}$ is of the form $(n,-n,0)$ for $n\in\mathbb{N}_0$.
After some simplifications this yields $Q_{k+1,a+1}(z/q;q)$ as defined
in~\eqref{Eq_RS}.
In contrast to the $\mathrm{A}_1^{(1)}$ case, \eqref{Eq_KR-ASW} does not
follow from Theorem~\ref{Thm_Principal-subspace} by specialisation of the
$x_i$.
For $b=a$ the determinant on the right (which up to normalisation is a
Schur function \cite{Macdonald95}) factorises, resulting in the simpler
\begin{align*}
&\sum_{\substack{\la_1\geq\cdots\geq\la_k\geq 0 \\[1pt] 
\mu_1\geq\cdots\geq\mu_k\geq 0}}
\Big(1-\frac{x_1}{x_3}\,q^{\la_a+\mu_a-1}\Big)
\prod_{i=1}^k \frac{
\big(\frac{x_1}{x_2}\big)^{\la_i} 
\big(\frac{x_2}{x_3}\big)^{\mu_i} 
q^{\la_i^2-\la_i\mu_i+\mu_i^2-\chi(i\leq a)(\la_i+\mu_i)}}
{(q;q)_{\la_i-\la_{i+1}}(q;q)_{\mu_i-\mu_{i+1}}} \\
&\quad = \sum_{y\in Q_{+}} \bigg(\prod_{1\leq i<j\leq 3}
\frac{1-\big(q^{y_i-y_j}x_i/x_j\big)^{a+1}}
{(x_i/x_j;q)_{\infty}} 
\prod_{i=1}^3 \frac{x_i^{(k+2)y_i} 
q^{(k+2)\binom{y_i}{2}+(a+1)iy_i}(x_i/x_3;q)_{y_i}}
{(qx_i/x_1;q)_{y_i}}\bigg).
\end{align*}
For $a=0$ this is \cite[Corollary 7.8]{FFJMM09} by Feigin et al.
The large-$k$ limit of Theorem~\ref{Thm_Principal-subspace} gives
our next result, where $\Par$ denotes the set of integer partitions.

\begin{corollary}\label{Cor_Hua}
For $a,b$ nonnegative integers and $\nu:=(a+b+2,b+1,0)$,
\begin{align*}
&\sum_{\la,\mu\in\Par}
\Big(1-\frac{x_1}{x_3}\,q^{\la_a+\mu_b-1}\Big)
\prod_{i\geq 1} \frac{
\big(\frac{x_1}{x_2}\big)^{\la_i} 
\big(\frac{x_2}{x_3}\big)^{\mu_i} 
q^{\la_i^2-\la_i\mu_i+\mu_i^2-\chi(i\leq a)\la_i-\chi(i\leq b)\mu_i}}
{(q;q)_{\la_i-\la_{i+1}}(q;q)_{\mu_i-\mu_{i+1}}} \\
&\quad\quad = \frac{1}{\prod_{1\leq i<j\leq 3} (x_i/x_j;q)_{\infty}}
\det_{1\leq i,j\leq 3}\big(x_i^{\nu_i-\nu_j}\big),
\end{align*}
where $q^{\la_0}=q^{\mu_0}:=0$.
\end{corollary}

For $a=b$ the right-hand side simplifies to
\[
\prod_{1\leq i<j\leq 3}\frac{1-(x_i/x_j)^{a+1}}{1-x_i/x_j} \,
\frac{1}{(qx_i/x_j;q)_{\infty}}
\]
so that the $a=b=0$ case of Corollary~\ref{Cor_Hua} gives the
$\mathrm{A}_2$ instance of Hua's combinatorial identity for
quivers of arbitrary finite type, see \cite[Theorem 4.9]{Hua00}
and the minor correction pointed out in~\cite{Fulman01}.
The determinant in Theorem~\ref{Thm_Principal-subspace} also simplifies
for $(x_1,x_2,x_3)=(z^2,z,1)$, resulting in 
\begin{align*}
&\sum_{\la,\mu\in\Par} \big(1-z^2\,q^{\la_a+\mu_b-1}\big)
\prod_{i\geq 1} \frac{z^{\la_i+\mu_i} 
q^{\la_i^2-\la_i\mu_i+\mu_i^2-\chi(i\leq a)\la_i-\chi(i\leq b)\mu_i}}
{(q;q)_{\la_i-\la_{i+1}}(q;q)_{\mu_i-\mu_{i+1}}} \\
&\quad\quad = \frac{1}{(zq,zq,z^2q;q)_{\infty}}\:
\frac{(1-z^{a+1})(1-z^{b+1})(1-z^{a+b+2})}{(1-z)(1-z)(1-z^2)},
\end{align*}
where $(a_1,\dots,a_k;q)_{\infty}:=(a_1;q)_{\infty}\cdots(a_k;q)_{\infty}$.
For $z=q$ this proves another conjecture by Kanade and Russell, stated as
Conjecture~{3.1} in \cite{KR23}.

\bigskip

The rest of this paper is organised as follows.
In Section~\ref{Sec_Prelims} we recall some standard material
from the theory of $q$-series, root systems and symmetric functions
that is used throughout the paper.
Then, in Section~\ref{Sec_A1-tree}, we review the classical or
$\mathrm{A}_1$ Bailey chain and a special case of the Bailey lattice
which in this paper will be referred to as the $\mathrm{A}_1$ Bailey tree.
Although all of the material in this section is essentially known,
some results are formulated in a form that is new.
In particular, the Bailey tree will be recast as a one-parameter
deformation of the Bailey chain.
In Section~\ref{Sec_A2-tree} the $\mathrm{A}_2$ Bailey chain of 
\cite{ASW99} is generalised to an $\mathrm{A}_2$ Bailey tree.
The simplest part of this tree consists of a two-parameter 
deformation of the $\mathrm{A}_2$ Bailey chain, analogous to the
one-parameter deformation described in Section~\ref{Sec_A1-tree}.
As it turns out, this two-parameter Bailey tree can only prove the
Kanade--Russell conjecture for $b=0$, and to obtain the full set of
identities we develop an additional and more complicated
four-parameter deformation of the $\mathrm{A}_2$ Bailey chain.
In Section~\ref{Sec_KR} we apply the $\mathrm{A}_2$ Bailey tree
to a suitable root identity to prove Theorem~\ref{Thm_KRistrue}.
As mentioned just above Conjecture~\ref{Con_KR-ASW}, there should
be an ASW identity for each dominant integral weight
$(K-a-b-3)\Lambda_0+a\Lambda_1+b\Lambda_2$ of $\mathrm{A}_2^{(1)}$,
and in Theorem~\ref{Thm_Belowtheline} of Section~\ref{Sec_Belowtheline}
the missing cases for $\tau=0$ are obtained using a key observation
due to Kanade and Russell.
Then, in Section~\ref{Sec_AG}, we prove the $\mathrm{A}_2^{(1)}$-analogues
of the Andrews--Gordon identities, stated in Theorems~\ref{Thm_3k4} and 
\ref{Thm_3k2}.
In Section~\ref{Sec_PS} we give a short introduction to the principal
subspaces of $\mathrm{A}_{r-1}^{(1)}$ in the sense of Feigin and
Stoyanovsky, and then apply the $\mathrm{A}_2$ Bailey tree to prove
Theorem~\ref{Thm_Principal-subspace}.
Finally, in Section~\ref{Sec_Outlook} we discuss the prospects of
an $\mathrm{A}_{r-1}$ Bailey tree and a generalisation of 
\eqref{Eq_KR-ASW} to arbitrary rank $r$.


\section{Preliminaries}\label{Sec_Prelims}

A partition $\la=(\la_1,\la_2,\dots)$ is a sequence of weakly
decreasing integers such that $\abs{\la}:=\la_1+\la_2+\cdots$ is finite.
We will follow the convention to omit the infinite string of zeros in
a partition, writing $(4,3,2,2,1)$ instead of $(4,3,2,2,1,0,\dots)$.
If $\lambda$ is a partition such that $\abs{\la}=n$, we say that $\la$
is a partition of $n$ and write $\la\vdash n$.
The set of all partitions, including the unique partition of $0$, is denoted
by $\Par$.
The length $l(\la)$ of a partition $\la$ is defined as the
number of positive $\la_i$.
A rectangular partition is a partition $\la$ such that 
$\la_1=\dots=\la_r=m$ for some positive integer $m$ and $\la_{r+1}=0$.
We will typically denote such a $\la$ by $(m^r)$.
The partition $\mu$ is said to be contained in the partition $\la$,
denoted $\mu\subseteq\la$ if $\mu_i\leq\la_i$ for all $i\geq 1$.

Many of the identities in this paper involve a sum over the root lattice
$Q$ of $\mathrm{A}_{r-1}$ or a subset thereof, mostly for $r=3$.
It will be convenient to employ the standard embedding of this lattice
in $\mathbb{Z}^r$, and we set
\begin{subequations}
\begin{align}\label{Eq_Q}
Q&:=\{(y_1,y_2,\dots,y_r)\in\mathbb Z^r: y_1+y_2+\dots+y_r=0\}, \\
Q_{+}&:=\{(y_1,y_2,\dots,y_r)\in Q: y_1+\dots+y_i\geq 0 
\text{ for all $1\leq i\leq r$}\}, \\
Q_{++}&:=\{(y_1,y_2,\dots,y_r)\in Q: y_1\geq y_2\geq\cdots\geq y_r\}.
\end{align}
\end{subequations}
For $y\in Q$ we also define $y_{ij}:=y_i-y_j$ for $1\leq i<j\leq r$,
where the reader is warned that for the sake of brevity the two indices
$i$ and $j$ will not be separated by a comma. 
Let $\varepsilon_i$ denote the $i$th standard unit vector in $\mathbb{R}^r$
and $\langle \cdot,\cdot\rangle$ the standard scalar product on 
$\mathbb{R}^r$, so that 
$\langle\varepsilon_i,\varepsilon_j\rangle=\delta_{i,j}$, with $\delta_{i,j}$
the Kronecker delta.
For $i\in I:=\{1,\dots,r-1\}$, let
\[
\alpha_i=\varepsilon_i-\varepsilon_{i+1}
\quad\text{and}\quad
\omega_i=\varepsilon_1+\dots+\varepsilon_i-
\frac{i}{r}(\varepsilon_1+\dots+\varepsilon_r)
\]
be the $i$th simple root and $i$th fundamental weight of $\mathfrak{sl}_r$
respectively, so that $\langle\alpha_i,\omega_j\rangle=\delta_{i,j}$.
Then $Q_{+}$ corresponds to $\sum_{i\in I} \mathbb{N}_0\alpha_i$
and $Q_{++}=Q\cap P_{+}$, where 
$P_{+}:=\sum_{i\in I} \mathbb{N}_0\omega_i$ is the set of dominant
(integral) weights of $\mathfrak{sl}_r$.

In this paper, $q$-series are typically viewed as elements of the formal
power series ring $R[[q]]$ with $R$ an appropriate coefficient ring or
field, such as $\mathbb{Z}$, $\mathbb{Q}(a)$ or $\mathbb{Q}(z,w)$.
A notable exception will be the $q$-series featured in Gustafson's
${_6\psi_6}$ summation \eqref{Eq_Gustafson} for the affine root system
$\mathrm{A}_{r-1}^{(1)}$. 
This require complex $q$ such that $\abs{q}<1$.
Many of our proofs rely on identities for basic hypergeometric
functions~\cite{GR04}.
Using the condensed notation
\[
(a_1,\dots,a_k;q)_n=\prod_{i=1}^k (a_i;q)_n,
\]
for $n\in\mathbb{Z}\cup\{\infty\}$, the
${_r\phi_s}$ basic hypergeometric function is defined as
\begin{equation}\label{Eq_rphis}
\qHyp{r}{s}{a_1,\dots,a_r}{b_1,\dots,b_s}{q,z}:=
\sum_{k=0}^{\infty} \frac{(a_1,\dots,a_r;q)_k}{(q,b_1,\dots,b_s;q)_k}
\Big((-1)^k q^{\binom{k}{2}}\Big)^{s-r+1} z^k.
\end{equation}
This will only ever be used for terminating series, i.e., for series
such that one of the numerator variables $a_i$ is of the form $q^{-n}$
for $n$ a nonnegative integer.
This ensures the summand vanishes unless $k\in\{0,1,\dots,n\}$.
We also adopt the standard one-line notation 
\[
{_r\phi_s}(a_1,\dots,a_r;b_1,\dots,b_s;q,z)
\]
for the series \eqref{Eq_rphis} and abbreviate the very-well-poised
basic hypergeometric function
\[
\qHyp{r}{r-1}{a_1,a_1^{1/2}q,-a_1^{1/2}q,a_4,\dots,a_r}
{a_1^{1/2},-a_1^{1/2},a_1q/a_4,\dots,a_1q/a_r}{q,z}
\]
as ${_rW_{r-1}}(a_1;a_4,\dots,a_r;q,z)$.


\section{The $\mathrm{A}_1$ Bailey lemma}\label{Sec_A1-tree}

To motivate the $\mathrm{A}_2$ Bailey tree presented in the next
section, we first review the classical $\mathrm{A}_1$ case.
Since the aim is to prove $\mathrm{A}_2^{(1)}$ generalisations of the
Andrews--Gordon--Bressoud identities \eqref{Eq_AGB}, we will focus on
that part of the Bailey machinery needed for proving~\eqref{Eq_AGB}.
This allows us to adopt simpler notation than is typically found in
treatments of the Bailey lemma such as in~\cite{Andrews84,Andrews86,W01}.
This notation is also more suited to generalisation to 
$\mathrm{A}_2$ and, ultimately, $\mathrm{A}_{r-1}$, since for higher rank
the use of actual Bailey pairs often is notationally very cumbersome.
The reader familiar with the existing literature should have no 
difficulties translating most of the results presented below in terms 
of Bailey pairs and transformations of such pairs.

Recall that $1/(q;q)_n=0$ for $n$ a negative integer.
The main ingredients in our treatment of the Bailey lemma are the
following three rational functions:
\begin{equation}\label{Eq_Phi-zu}
\Phi_n(z;q):=\frac{1}{(q,zq;q)_n}, \qquad
\Phi_n(u;z;q):=\frac{1-uz-(1-u)zq^n}{(q;q)_n(z;q)_{n+1}} 
\end{equation}
and
\[
\mathcal{K}_{n;r}(z;q):=\frac{z^r q^{r^2}}{(q;q)_{n-r}},
\]
where $n,r\in\mathbb{Z}$.
The reason for separating $u$ and $z$ as well as $n$ and $r$ by semicolons
is that $n,r,u$ and $z$ all become sequences in the higher-rank case.
For later reference we note that
\begin{subequations}
\begin{gather}
\Phi_n(z^{-1};q^{-1})=(zq)^n q^{n^2} \Phi_n(z;q), 
\label{Eq_Phin-inv} \\[1mm]
\Phi_n(1;z;q)=\Phi_n(z;q), \qquad
\Phi_n(z^{-1};z;q)=q^n \Phi_n(z;q)
\end{gather}
\end{subequations}
and
\begin{equation}\label{Eq_Phi-rec}
\Phi_n(u;z;q)=\Phi_n(z/q;q)-\frac{uz}{(z;q)_2}\,\Phi_{n-1}(zq;q).
\end{equation}
From \cite[Equation (2.3.4)]{GR04} it follows that
\begin{equation}\label{Eq_inversion-A1}
\sum_{r=N}^n q^{n-r} 
\Phi_{n-r}\big(zq^{2r};q\big)
\Phi_{r-N}\big(zq^{2r};q^{-1}\big) 
=\delta_{n,N},
\end{equation}
which is Andrews' $\mathrm{A}_1$ matrix inversion
\cite[Lemma 3]{Andrews79} in disguise.

A key role in the Bailey lemma \cite{Bailey48} is played by the
above-mentioned Bailey pairs.
These are pairs of sequences $(\alpha(z;q),\beta(z;q))$ indexed by
nonnegative integers and depending on parameters $z$ and $q$
such that\footnote{It the literature on the Bailey lemma it is customary
to use $a$ instead of $z$, and to refer to a pair satisfying
\eqref{Eq_BP-A1} as a Bailey pair relative to $a$.}
\begin{subequations}
\begin{equation}\label{Eq_BP-A1}
\beta_n(z;q)=\sum_{r=0}^n \frac{\alpha_r(z;q)}{(q;q)_{n-r}(zq;q)_{n+r}},
\end{equation}
or, equivalently, \cite[Lemma 3]{Andrews79}
\begin{equation}\label{Eq_BP-A1-inv}
\alpha_n(z;q)=\sum_{r=0}^n \frac{1-zq^{2n}}{1-z}\,
(-1)^{n-r} q^{\binom{n-r}{2}}\,
\frac{(z;q)_{n+r}}{(q;q)_{n-r}}\,\beta_r(z;q).
\end{equation}
\end{subequations}

In \cite{Andrews84} Andrews discovered that, given a Bailey pair relative
to $z$, there is a simple transformation (already implicit in the work of
Bailey) that turns this pair into a new Bailey pair relative to $z$.
This can be iterated to yield what Andrews termed the Bailey chain:
\begin{equation}\label{Eq_chain}
\big(\alpha(z;q),\beta(z;q)\big)\mapsto \big(\alpha'(z;q),\beta'(z;q)\big)
\mapsto \big(\alpha''(z;q),\beta''(z;q)\big)\mapsto\cdots.
\end{equation}
The essence of (a special case of) this transformation is captured in the
following lemma, which by abuse of terminology we also refer to as the
$\mathrm{A}_1$ Bailey chain.
In particular it should be clear that the result below lends itself to 
iteration thanks to its reproducing nature.

\begin{lemma}[$\mathrm{A}_1$ Bailey chain]\label{Lem_chain}
For $n$ a nonnegative integer,
\begin{equation}\label{Eq_Bailey-chain}
\sum_{r=0}^n \mathcal{K}_{n;r}(z;q) \Phi_r(z;q)=\Phi_n(z;q).
\end{equation}
\end{lemma}

\begin{proof}
In $q$-hypergeometric notation the identity \eqref{Eq_Bailey-chain} is
\begin{equation}\label{Eq_1phi1}
{_1\phi_1}(q^{-n};zq;q;zq^{n+1})=\frac{1}{(zq;q)_n},
\end{equation}
which is the terminating form of~\cite[Equation (II.5)]{GR04}.
\end{proof}

\begin{corollary}\label{Cor_infinite-iteration-A1}
We have
\[
\Phi_n(z;q)=\sum_{\la\in\Par} 
\prod_{i\geq 1} \frac{z^{\la_i} q^{\la_i^2}}{(q;q)_{\la_{i-1}-\la_i}},
\]
where $\la_0:=n$.
\end{corollary}

\begin{proof}
By a $k$-fold application of \eqref{Eq_Bailey-chain},
\[
\Phi_n(z;q)=\sum_{\substack{\la\in\Par \\[1pt] l(\la)\leq k}}
\Phi_{\la_k}(z;q)
\prod_{i=1}^k \mathcal{K}_{\la_{i-1};\la_i}(z;q), 
\]
where $\la_0:=n$.
Letting $k$ tend to infinite yields the claim.
\end{proof}

The Bailey chains \eqref{Eq_chain} or \eqref{Eq_Bailey-chain} alone are 
not enough to prove the full set of Andrews--Gordon--Bressoud identities
\eqref{Eq_AGB}, and in \cite{AAB87} Agarwal, Andrews and Bressoud found
a further transformation for Bailey pairs, this time scaling the parameter
$z$ by a factor $q$:
\[
\big(\alpha(z;q),\beta(z;q)\big)\mapsto 
\big(\alpha'(z/q;q),\beta'(z/q;q)\big).
\]
Combining this with the original transformation allows for more complicated
patterns of iteration which are not linear in nature.
This led Agarwal, Andrews and Bressoud to refer to their discovery as
the Bailey lattice.
Equipped with the Bailey lattice it is a simple exercise to prove
\eqref{Eq_AGB} in full.
The part of the Bailey lattice needed for proving the
Andrews--Gordon--Bressoud identities has the structure of a simple
binary tree, and is captured in the following lemma.

\begin{lemma}[$\mathrm{A}_1$ Bailey tree]\label{Lem_A1-tree}
For $n$ a nonnegative integer,
\begin{subequations}
\begin{align}\label{Eq_A1-tree-a}
\sum_{r=0}^n \mathcal{K}_{n;r}(z;q)\Phi_{r}(1;z;q)&=\Phi_n(1;z;q),
\intertext{and}
\sum_{r=0}^n \mathcal{K}_{n;r}(z/q;q)\Phi_r(u;z;q)&=\Phi_n(uz;z;q).
\label{Eq_A1-tree-b}
\end{align}
\end{subequations}
\end{lemma}

By $\Phi_n(1;z;q)=\Phi_n(z;q)$, the first claim is merely a restatement
of the Bailey chain.
The crucial part of Lemma~\ref{Lem_A1-tree} is that one can first
repeatedly apply \eqref{Eq_A1-tree-a} (or \eqref{Eq_Bailey-chain}) and
then change the nature of the iteration by continuing with
\eqref{Eq_A1-tree-b}, initially with $u=1$, then $u=z$, $u=z^2$ and so
on, changing the linear nature of Lemma~\ref{Lem_chain}, instead
generating the binary tree
\label{page_tree}
\begin{center}
\begin{tikzpicture}[scale=0.35,line width=0.3pt]
\draw[dashed,blue] (0,0)--(6,-6);
\draw[dashed,blue] (-1,-1)--(4,-6);
\draw[dashed,blue] (-2,-2)--(2,-6);
\draw[dashed,blue] (-3,-3)--(0,-6);
\draw[dashed,blue] (-4,-4)--(-2,-6);
\draw[dashed,blue] (-5,-5)--(-4,-6);
\draw[dashed,blue] (-5,-5)--(-6,-6);
\draw[blue] (0,0)--(-5,-5);
\draw[blue] (0,0)--(5,-5);
\draw[blue] (-1,-1)--(3,-5);
\draw[blue] (-2,-2)--(1,-5);
\draw[blue] (-3,-3)--(-1,-5);
\draw[blue] (-4,-4)--(-3,-5);
\draw[blue] (-5,-5);
\draw[blue] (0,0)--(-5,-5);
\foreach \y in {0,...,-5} \foreach \x in {0,...,-\y} 
        {\draw[blue,fill=white] (2*\x+\y,\y) circle (3.5mm);}
\foreach \y in {0,...,-5} \draw (\y,\y) node {$\scriptstyle 0$};
\foreach \y in {-1,...,-5} \draw (\y+2,\y) node {$\scriptstyle 1$};
\foreach \y in {-2,...,-5} \draw (\y+4,\y) node {$\scriptstyle 2$};
\foreach \y in {-3,...,-5} \draw (\y+6,\y) node {$\scriptstyle 3$};
\foreach \y in {-4,...,-5} \draw (\y+8,\y) node {$\scriptstyle 4$};
\foreach \y in {-5,...,-5} \draw (\y+10,\y) node {$\scriptstyle 5$};
\end{tikzpicture}
\end{center}
where the label $i\in\mathbb{N}_0$ represents the rational function 
$\Phi_n(z^i;z;q)$.
Of course, \eqref{Eq_A1-tree-b} in isolation allows for
\[
\Phi_n(u;z;q)\mapsto \Phi_n(uz;z;q)\mapsto 
\Phi_n(uz^2;z;q)\mapsto \Phi_n(uz^3;z;q)\mapsto\cdots,
\]
but if one wishes to combine \eqref{Eq_A1-tree-a} and \eqref{Eq_A1-tree-b}
then this fixes $u=1$.

\begin{proof}[Proof of Lemma~\ref{Lem_A1-tree}]
Since \eqref{Eq_A1-tree-a} is a restatement of \eqref{Eq_Bailey-chain},
only \eqref{Eq_A1-tree-b} requires proof.

By \eqref{Eq_Phi-zu} or \eqref{Eq_Phi-rec} it is clear that both sides 
of \eqref{Eq_A1-tree-b} are polynomials in $u$ of degree one.
Taking the constant term using \eqref{Eq_Phi-rec} yields
\eqref{Eq_Bailey-chain} with $z\mapsto z/q$.
Similarly, extracting the coefficient of $u$ in \eqref{Eq_A1-tree-b}
and dividing both sides by $-z/(z;q)_2$, gives
\[
\sum_{r=1}^n \mathcal{K}_{n;r}(z/q;q) \Phi_{r-1}(zq;q)
=z\Phi_{n-1}(zq;q).
\]
Here we have also used that $\Phi_{-1}=0$ to change the lower bound on
the sum from $0$ to $1$.
Shifting $r\mapsto r+1$ and noting that
\[
\mathcal{K}_{n;r+1}(z/q;q)=z\mathcal{K}_{n-1;r}(zq;q),
\]
results in \eqref{Eq_Bailey-chain} with $(z,n)\mapsto(zq,n-1)$.
\end{proof}

Before we are ready to demonstrate how the Andrews--Gordon--Bressoud
identities \eqref{Eq_AGB} arise from the above results, a slight
reformulation of the previous two lemmas is needed.
For this purpose we define
\begin{align*}
\Phi_{n;y}(z;q)&:=\frac{\Phi_{n-y}(zq^{2y};q)}{(zq;q)_{2y}}
=\frac{1}{(q;q)_{n-y}(zq;q)_{n+y}}, \\[1mm]
\Phi_{n;y}(u;z;q)&:=
\frac{\Phi_{n-y}(u;zq^{2y};q)}{(zq;q)_{2y}}
=\frac{1-uzq^{2y}-(1-u)zq^{n+y}}{(q;q)_{n-y}(zq;q)_{n+y}(1-zq^{2y})},
\end{align*}
where $n,y\in\mathbb{Z}$.
Note that once again $\Phi_{n;y}(1;z;q)=\Phi_{n;y}(z;q)$, and that
$\Phi_{n;y}$ vanishes unless $n\geq y$.
Further note that \eqref{Eq_BP-A1} and \eqref{Eq_BP-A1-inv} 
can be recast in terms of $\Phi_{n;y}(z;q)$ and $\Phi_n(z;q)$ as
\begin{subequations}
\begin{align}\label{Eq_Bailey-pair-2} 
\beta_n(z;q)&=\sum_{r=0}^n \Phi_{n;r}(z;q)\alpha_r(z;q), \\
\label{Eq_Bailey-pair-inv-2} 
\alpha_n(z;q)&=q^{-n} (zq;q)_{2n} \sum_{r=0}^n
q^r \Phi_{n-r}\big(zq^{2n};q^{-1}\big) \beta_{r}(z;q).
\end{align}
\end{subequations}
By replacing $(z,n)\mapsto (zq^{2y},n-y)$ in Lemmas \ref{Lem_chain}
and \ref{Lem_A1-tree} and then shifting the summation index
$r\mapsto r-y$, the following two corollaries 
arise.\footnote{Corollary~\ref{Cor_A1-chain} for $z=1$ and $z=q$ is
equivalent to \cite[(R1) \& (R2)]{Paule85} in that \eqref{Eq_A1-chain} 
for these two values of $z$ corresponds to the coefficient of $a_n$ 
and $b_n$ in equations (R1) and (R2) of \cite{Paule85}.}

\begin{corollary}\label{Cor_A1-chain}
For $n,y\in\mathbb{Z}$,
\begin{equation}\label{Eq_A1-chain}
\sum_{r=y}^n \mathcal{K}_{n;r}(z;q)\Phi_{r;y}(z;q)
=z^y q^{y^2} \Phi_{n;y}(z;q).
\end{equation}
\end{corollary}

\begin{corollary}\label{Cor_A1-tree}
For $n,y\in\mathbb{Z}$,
\begin{subequations}
\begin{align}
\sum_{r=y}^n \mathcal{K}_{n;r}(z;q) \Phi_{r;y}(1;z;q)
&=z^y q^{y^2} \Phi_{n;y}(1;z;q),
\intertext{and}
\sum_{r=y}^n \mathcal{K}_{n;r}(z/q;q) \Phi_{r;y}(u;z;q) 
&=z^y q^{y^2-y} \Phi_{n;y}(uzq^{2y};z;q). 
\label{Eq_right-branch}
\end{align}
\end{subequations}
\end{corollary}

We are now ready to give a short proof of \eqref{Eq_AGB}.

\begin{proof}
Slater's Bailey pairs B(3) and E(3) \cite{Slater51} are equivalent to
the following pair of polynomial identities:\footnote{These two results
can be traced back to Rogers' work on the Rogers--Ramanujan identities.
For example, the left-hand side of \eqref{Eq_B3} is what Rogers denotes
by $q^{-(n+1)^2}\beta_{2n+1}$ on page 316 of~\cite{Rogers17}. 
His equation (5) on the following page then states that 
$q^{-n-1} \beta_{2n+1}/(q;q)_{2n+1}=q^{n(n+1)}/(q;q)_n$.}
\begin{equation}\label{Eq_B3}
\sum_{y=-n-1}^n(-1)^y q^{3\binom{y}{2}+2y} \qbin{2n+1}{n-y}=
\frac{(q;q)_{2n+1}}{(q;q)_n}
\end{equation}
and
\[
\sum_{y=-n-1}^n(-1)^y q^{2\binom{y}{2}+y} \qbin{2n+1}{n-y}=
\frac{(q;q)_{2n+1}}{(q^2;q^2)_n},
\]
where $n$ is a nonnegative integer.
In terms of the rational function $\Phi_{n;y}(z;q)$,
Slater's identities can be written as
\begin{subequations}
\begin{equation}\label{Eq_Slater}
\sum_{y=-n-1}^n(-1)^y q^{(2+\tau)\binom{y+1}{2}-y} \Phi_{n;y}(q;q)=
\frac{1-q}{(q^{2-\tau};q^{2-\tau})_n},
\end{equation}
where $\tau\in\{0,1\}$.
Although this form of the identity is perfectly suitable for the
application of the Bailey tree, we will rewrite it further to more
closely mimic its $\mathrm{A}_2$-analogue, given by \eqref{Eq_seed} 
on page~\pageref{page_seed}. 
To this end, let $t_y$ be the summand of \eqref{Eq_Slater} and rewrite 
the sum as $\sum_y t_y=\sum_y t_{2y}+\sum_y t_{-2y-1}$.
Using that $\Phi_{n;-2y-1}(q;q)=\Phi_{n;2y}(q;q)$ and thus
$t_{-2y-1}=-t_{2y} q^{4y+1}$, this yields
\begin{equation}\label{Eq_Slater2}
\sum_{y\in\mathbb{Z}} q^{(2+\tau)\binom{2y+1}{2}-2y}\,
\frac{1-q^{4y+1}}{1-q}\,\Phi_{n;2y}(q;q)
=\frac{1}{(q^{2-\tau};q^{2-\tau})_n},
\end{equation}
\end{subequations}
where it is noted that the summand vanishes unless
$-\lfloor(n+1)/2\rfloor\leq y\leq \lfloor n/2\rfloor$.
(Since both sides of \eqref{Eq_Slater} and \eqref{Eq_Slater2} 
trivially vanish for negative values of $n$, both forms of the
identity are true for all $n\in\mathbb{Z}$).

In the following we identify the identity \eqref{Eq_Slater2}
with the root of the binary tree shown on page~\pageref{page_tree}.
By a $(k-a)$-fold application of Corollary~\ref{Cor_A1-chain} with $z=q$, 
which corresponds to taking $k-a$ downward steps along the left-most 
branch of the tree,
\begin{align}\label{Eq_ka-steps}
&\sum_{y\in\mathbb{Z}} q^{(2k-2a+2+\tau)\binom{2y+1}{2}-2y}\,
\frac{1-q^{4y+1}}{1-q}\,\Phi_{n;2y}(q;q) \\
& \qquad=\sum_{\la\subseteq (n^{k-a})}
\frac{1}{(q^{2-\tau};q^{2-\tau})_{\la_{\ell}}}
\prod_{i=1}^{k-a} \mathcal{K}_{\la_{i-1};\la_i}(q;q), \notag
\end{align}
where $\la_0:=n$ and $k-a\in\mathbb{N}_0$.
We now replace $\Phi_{n;2y}(q;q)$ by $\Phi_{n;2y}(1;q;q)$
and then take $a$ steps along the tree in the south-east
direction using \eqref{Eq_right-branch} with $z=q$ and
$u=q^{(i-1)(4y+1)}$ in the $i$th step.
Since
\[
\mathcal{K}_{n;r}(z/q;q)=q^{-r} \mathcal{K}_{n;r}(z;q),
\]
this yields
\begin{align}\label{Eq_pre-simp}
&\sum_{y\in\mathbb{Z}}q^{K\binom{2y+1}{2}-2(a+1)y}\,
\frac{1-q^{4y+1}}{1-q}\,\Phi_{n;2y}\big(q^{a(4y+1)};q;q\big) \\
&\qquad=\sum_{\la\subseteq (n^k)}
\frac{1}{(q^{2-\tau};q^{2-\tau})_{\la_k}}
\prod_{i=1}^k q^{-\chi(i\leq a)\la_i}
\mathcal{K}_{\la_{i-1};\la_i}(q;q) \notag \\
&\qquad=\sum_{\la\subseteq (n^k)}
\frac{q^{\la_1^2+\dots+\la_k^2+\la_{a+1}+\dots+\la_k}}
{(q;q)_{n-\la_1}(q;q)_{\la_1-\la_2}\cdots(q;q)_{\la_{k-1}-\la_k}
(q^{2-\tau};q^{2-\tau})_{\la_k}}, \notag
\end{align}
where $a,k$ are integers such that $0\leq a\leq k$, 
and $K:=2k+2+\tau$.
We note that the path along the tree we have taken is 
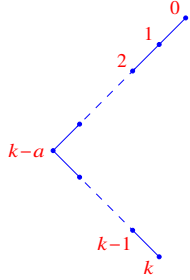
\label{page_path}
\begin{center}
\begin{tikzpicture}[scale=0.35,line width=0.3pt]
\draw[blue] (0,0)--(-2,-2);
\draw[dashed,blue] (-1.9,-1.9)--(-4.1,-4.1);
\draw[blue] (-4,-4)--(-5,-5)--(-4,-6);
\draw[dashed,blue] (-3.8,-6.2)--(-2,-8);
\draw[blue] (-2,-8)--(-1,-9);
\foreach \y in {0,-1,-2,-4,-5} {\filldraw[blue] (\y,\y) circle (0.75mm);}
\foreach \y in {-6,-8,-9} {\filldraw[blue] (-10-\y,\y) circle (0.75mm);}
\draw[red] (-0.4,0.4) node {$\scriptstyle 0$};
\draw[red] (-1.4,-0.6) node {$\scriptstyle 1$};
\draw[red] (-2.4,-1.6) node {$\scriptstyle 2$};
\draw[red] (-6,-5) node {$\scriptstyle k-a$};
\draw[red] (-2.7,-8.5) node {$\scriptstyle k-1$};
\draw[red] (-1.4,-9.4) node {$\scriptstyle k$};
\end{tikzpicture}
\end{center}
where the labels denote the level (or distance to the root) 
of each vertex. 

Although it is not an essential step in the proof and one can proceed
by directly taking the large-$n$ limit in \eqref{Eq_pre-simp}, we observe
that the left-hand side allows for a simplification which only requires
the function $\Phi_{n;y}(q;q)$.
This simplification is achieved by noting that
\begin{align*}
&q^{K\binom{2y+1}{2}-2(a+1)y}\,
\frac{1-q^{4y+1}}{1-q}\,\Phi_{n;2y}\big(q^{a(4y+1)};q;q\big) \\
&\qquad=\sum_{y'\in\{-2y,2y-1\}}(-1)^{y'} q^{K\binom{y'+1}{2}-(a+1)y'}\,
\frac{1-q^{n+y'+1}}{1-q}\,\Phi_{n;y'}(q;q).
\end{align*}
Hence the left-hand side of \eqref{Eq_pre-simp} may also be written as
\[
\sum_{y=-n-1}^n(-1)^y q^{K\binom{y+1}{2}-(a+1)y}\,
\frac{1-q^{n+y+1}}{1-q}\, \Phi_{n;y}(q;q).
\]
Since
\[
\lim_{n\to\infty} \Phi_{n;y}(z;q)=
\frac{1}{(q,zq;q)_{\infty}},
\]
this implies that in the large-$n$ limit
\[
\frac{1}{(q;q)_{\infty}}
\sum_{y\in\mathbb{Z}} (-1)^y q^{K\binom{y+1}{2}-(a+1)y}
=\sum_{\substack{\la\in\Par \\[1pt] l(\la)\leq k}}
\frac{q^{\la_1^2+\dots+\la_k^2+\la_{a+1}+\dots+\la_k}}
{(q;q)_{\la_1-\la_2}\cdots(q;q)_{\la_{k-1}-\la_k}
(q^{2-\tau};q^{2-\tau})_{\la_k}}.
\]
By the Jacobi triple product identity \cite[(II.28)]{GR04} the
left-hand side admits the product form
\[
\frac{(q^K;q^K)_{\infty}}{(q;q)_{\infty}}\,\theta\big(q^{a+1};q^K\big),
\]
resulting in \eqref{Eq_AGB}.
\end{proof}


\section{The $\mathrm{A}_2$ Bailey tree}\label{Sec_A2-tree}

In this section we present an $\mathrm{A}_2$-analogue of the 
$\mathrm{A}_1$ Bailey tree.
This tree is three-dimensional, or parametrisable by three
nonnegative integer variables, with an added layer of complexity
in that the structure of the tree is not actually tree-like in the
strict graph-theoretical sense.
In developing our Bailey tree we once again avoid the use
of Bailey pairs, although in the short Section~\ref{Subsec_A2-Bailey-pairs}
we briefly discuss $\mathrm{A}_2$ Bailey pairs and 
$\mathrm{A}_2$ Bailey pair inversion.

\subsection{A Bailey tree for $\mathrm{A}_2$}\label{Subsec_A2-tree}
The most important definition of this section is the
$\mathrm{A}_2$-analogue of the rational function $\Phi_n(z;q)$, and
following \cite[Definition~4.2]{ASW99} and \cite[Equation (5.1)]{W06},
we let
\begin{equation}\label{Eq_Phi}
\Phi_{n,m}(z,w;q):=\frac{(zwq;q)_{n+m}}{(q,zq,zwq;q)_n(q,wq,zwq;q)_m},
\end{equation}
where $n,m,r,s\in\mathbb{Z}$.
This function was also considered in~\cite{FFJMM09}.
A first hint that \eqref{Eq_Phi} has something to do with the $\mathrm{A}_2$
root system follows from the analogue of \eqref{Eq_Phin-inv}:
\begin{equation}\label{Eq_Phi-recip}
\Phi_{n,m}\big(z^{-1},w^{-1};q^{-1}\big)=(zq)^n(wq)^m q^{n^2-nm+m^2}
\Phi_{n,m}(z,w;q).
\end{equation}
Here $n_1^2-n_1n_2+n_2^2=\frac{1}{2}\sum_{i,j=1}^2 n_i A_{ij} n_j$,
where $(A_{ij})=(\langle \alpha_i,\alpha_j\rangle)$ is the 
$\mathrm{A}_2$ Cartan matrix.
Before we show how the function $\Phi_{n,m}(z,w;q)$ can be used to
generalise all of the results of the previous section, we state the
$\mathrm{A}_2$-analogue of Andrews' matrix
inversion~\eqref{Eq_inversion-A1}.
To the best of our knowledge this result is new.

\begin{proposition}\label{Prop_inversion}
For $n,m,N,M$ integers such that $n\geq N$ and $m\geq M$,
\begin{align}\label{Eq_inversion}
&\sum_{r=N}^n \sum_{s=M}^m q^{n+m-r-s} 
\Phi_{n-r,m-s}\big(zq^{2r-s},wq^{2s-r};q\big)
\Phi_{r-N,s-M}\big(zq^{2r-s},wq^{2s-r};q^{-1}\big) \\[1mm]
&\qquad\qquad =\delta_{n,N}\delta_{m,M}. \notag
\end{align}
\end{proposition}

This inversion relation, which simplifies to \eqref{Eq_inversion-A1}
for $M=m=0$ or $w=0$, will be applied in Section~\ref{Sec_PS} to prove
Theorem~\ref{Thm_Principal-subspace}.

\begin{proof}[Proof of Proposition~\ref{Prop_inversion}]
Replacing
\[
(n,m,z,w)\mapsto \big(n+N,m+M,zq^{M-2N},wq^{N-2M}\big),
\] 
and then shifting the summation indices $(r,s)\mapsto (r+N,s+M)$, 
it follows that \eqref{Eq_inversion} for general $N,M$ is equivalent
to the $N=M=0$ case.
The proof of \eqref{Eq_inversion} for $N=M=0$ requires Gustafson's
multiple ${_6\psi_6}$ summation \cite[Theorem 1.15]{Gustafson87}
for the affine root system $\mathrm{A}_{r-1}^{(1)}$:
\begin{align}\label{Eq_Gustafson}
\sum_{y\in Q} & \prod_{1\leq i<j\leq r} 
\frac{x_iq^{y_i}-x_jq^{y_j}}{x_i-x_j}
\prod_{i,j=1}^r \frac{(a_jx_i/x_j;q)_{y_i}}
{(b_jx_i/x_j;q)_{y_i}} \\
&=\frac{(Bq^{1-r},q/A;q)_{\infty}}{(q,Bq^{1-r}/A;q)_{\infty}}
\prod_{i,j=1}^r \frac{(qx_i/x_j,x_ib_j/a_ix_j;q)_{\infty}}
{(x_ib_j/x_j,x_iq/a_ix_j;q)_{\infty}}, \notag
\end{align}
where $A:=a_1\cdots a_r$, $B:=b_1\cdots b_r$ and
$\max\{\abs{q},\abs{Bq^{1-r}/A}\}<1$.
Assuming $r\geq 3$ and specialising
\[
(a_1,\dots,a_r)=(q^{-n},c_2,\dots,c_{r-1},1),\quad
(b_1,\dots,b_r)=(q,c_2,\dots,c_{r-1},q^{m+1}),
\]
\eqref{Eq_Gustafson} yields
\[
\sum_{y\in Q} \prod_{1\leq i<j\leq r} 
\frac{x_iq^{y_i}-x_jq^{y_j}}{x_i-x_j}
\prod_{i=1}^r \frac{(q^{-n}x_i/x_1,x_i/x_r;q)_{y_i}}
{(qx_i/x_1,q^{m+1}x_i/x_r;q)_{y_i}}=0
\]
for $\abs{q}<1$ and $n+m>r-3$.
The summand on the left vanishes unless $0\leq y_1\leq n$
and $0\leq -y_r\leq m$.
Since $y\in Q$, this implies that for $r=3$ the summand
has finite support, making the condition $\abs{q}<1$ redundant.
Then replacing $(y_1,y_2,y_3)\mapsto (r,s-r,-s)$ and
$(x_1,x_2,x_3)\mapsto(zw,w,1)$, the identity \eqref{Eq_inversion} for
$N=M=0$ and $(n,m)\neq (0,0)$ follows.
Since the $(n,m)=(0,0)$ case trivially holds, we are done.
\end{proof}

Apart from $\Phi_{n,m}(z,w;q)$ we also need the function
\[
\mathcal{K}_{n,m;r,s}(z,w;q):=
\frac{z^r w^s q^{r^2-rs+s^2}}{(q;q)_{n-r}(q;q)_{m-s}}.
\]
Then the $\mathrm{A}_2$ Bailey lemma of \cite[Theorem 4.3]{ASW99} 
is equivalent to the following reproducing identity for
$\Phi_{n,m}(z,w;q)$, see also \cite[Theorem 5.1]{W06}
or~\cite[Corollary 7.9]{FFJMM09}.

\begin{theorem}[$\mathrm{A}_2$ Bailey chain]\label{Thm_A2-Bailey}
For $n,m$ nonnegative integers,
\begin{equation}\label{Eq_Bailey-A2}
\sum_{r=0}^n \sum_{s=0}^m \mathcal{K}_{n,m;r,s}(z,w;q)
\Phi_{r,s}(z,w;q)=\Phi_{n,m}(z,w;q).
\end{equation}
\end{theorem}

Since
\begin{equation}\label{Eq_wnul} 
\Phi_{n,m}(z,0;q)=\frac{\Phi_n(z;q)}{(q;q)_m}
\quad\text{and}\quad
\mathcal{K}_{n,m;r,s}(z,0;q)=\delta_{s,0}\,
\frac{\mathcal{K}_{n;r}(z;q)}{(q;q)_m},
\end{equation}
Theorem~\ref{Thm_A2-Bailey} for $w=0$ simplifies to Lemma~\ref{Lem_chain}.
It also simplifies to this lemma for $m=0$.
We further remark that \eqref{Eq_Bailey-A2} holds for all integers
$n,m$, with both sides vanishing trivially unless $n,m\geq 0$.
The proof of \eqref{Eq_Bailey-A2} presented below replicates the
second part of the proof of~\cite[Theorem 4.3]{ASW99}.
For an alternative approach using Hall--Littlewood polynomials the
reader is referred to~\cite{W06}.

\begin{proof}
Denote the double sum on the left of \eqref{Eq_Bailey-A2}
by $\phi_{n,m}(z,w;q)$.
Then
\[
\phi_{n,m}(z,w;q)=\frac{1}{(q;q)_m}\sum_{r=0}^n 
\frac{z^r q^{r^2}}{(q;q)_{n-r}(q,zq;q)_r}\,
\qHyp{2}{2}{zwq^{r+1},q^{-m}}{wq,zwq}{q,wq^{m-r+1}}.
\]
By a limiting case of \cite[Equation (III.9)]{GR04},
\[
\qHyp{2}{2}{a,q^{-n}}{b,c}{q,\frac{bcq^n}{a}}=\frac{1}{(c;q)_n}\,
\qHyp{2}{1}{b/a,q^{-n}}{b}{q,cq^n}.
\]
Applying this with $(n,a,b,c)\mapsto(m,zwq^{r+1},zwq,wq)$ yields
\begin{align*}
\phi_{n,m}(z,w;q)&=\frac{1}{(q,wq;q)_m}\sum_{r=0}^n 
\frac{z^r q^{r^2}}{(q;q)_{n-r}(q,zq;q)_r}\,
\qHyp{2}{1}{q^{-r},q^{-m}}{zwq}{q,wq^{m+1}} \\
&=\sum_{s=0}^n \sum_{r=s}^n 
\frac{z^r w^s q^{\binom{r-s}{2}+r^2}}
{(q;q)_{n-r}(q;q)_{m-s}(zq;q)_r(q;q)_{r-s}(q,zwq;q)_s}.
\end{align*}
After shifting $r\mapsto r+s$ this gives
\[
\phi_{n,m}(z,w;q)
=\frac{1}{(wq;q)_m}\sum_{s=0}^n \frac{(wz)^s q^{s^2}}
{(q;q)_{n-s}(q;q)_{m-s}(q,zq,zwq;q)_s}\,
\qHyp{1}{1}{q^{-(n-s)}}{zq^{s+1}}{q,zq^{n+1}}.
\]
Finally, by \eqref{Eq_1phi1} with $(z,n)\mapsto(zq^s,n-s)$,
\[
\phi_{n,m}(z,w;q)=\frac{1}{(q,zq;q)_n(q,wq;q)_m}\,
\qHyp{2}{1}{q^{-n},q^{-m}}{zwq}{q,wzq^{n+m+1}}
=\Phi_{n,m}(z,w;q),
\]
where the final equality follows from the $q$-Chu--Vandermonde
summation~\cite[Equation (II.7)]{GR04}.
\end{proof}

Generalising the proof of Corollary~\ref{Cor_infinite-iteration-A1}
to the rank-two setting in the obvious manner gives the following
multisum representation for $\Phi_{n,m}(z,w;q)$, see also
\cite[Corollary 3.4]{W06}.

\begin{corollary}\label{Cor_infinite-iteration-A2}
We have
\[
\Phi_{n,m}(z,w;q)=\sum_{\la,\mu\in\Par} 
\prod_{i\geq 1} \frac{z^{\la_i} w^{\mu_i} 
q^{\la_i^2-\la_i\mu_i+\mu_i^2}}
{(q;q)_{\la_{i-1}-\la_i}(q;q)_{\mu_{i-1}-\mu_i}},
\]
where $\la_0:=n$ and $\mu_0:=m$.
\end{corollary}

Next we will generalise the $\mathrm{A}_1$ Bailey tree of
Lemma~\ref{Lem_A1-tree}.
This requires a suitable $u,v$-generalisation
$\Phi_{n,m}(u,v;z,w;q)$ of $\Phi_{n,m}(z,w;q)$ such that
\begin{subequations}\label{Eq_Phi-uv-special}
\begin{align}\label{Eq_11}
\Phi_{n,m}(1,1;z,w;q)&=\Phi_{n,m}(z,w;q), \\[1mm]
\Phi_{n,m}(z^{-1},w^{-1};z,w;q)&=q^n \Phi_{n,m}(z,w;q)
\label{Eq_11-b}
\end{align}
\end{subequations}
and
\begin{equation}\label{Eq_Phinm_to_Phin}
\Phi_{n,m}(u,v;z,0;q)=\frac{\Phi_n(u;z;q)}{(q;q)_m}.
\end{equation}
We begin by noting that the decomposition \eqref{Eq_Phi-rec}
for $u=1$ follows from the relation $1-z=(1-cz)-z(1-c)$ for $c=q^n$.
This readily generalises to the $4$-term relation
\begin{align*}
&(1-z)(1-w)(1-zw)(1-cdzw) \\
&\qquad =(1-cz)(1-w)(1-czw)(1-dzw) 
-z(1-c)(1-dw)(1-zw)(1-czw) \\
&\qquad\quad +zw^2(1-c)(1-d)(1-z)(1-cz),
\end{align*}
which for $(c,d)\mapsto(q^n,q^m)$ implies
\begin{align}\label{Eq_Phi-three} 
\Phi_{n,m}(z,w;q)&=\Phi_{n,m}(z/q,w;q) 
-\frac{z}{(z;q)_2}\, \Phi_{n-1,m}(zq,w/q;q) \\
&\quad+\frac{zw^2}{(w,zw;q)_2}\,\Phi_{n-1,m-1}(z,wq;q). \notag
\end{align}
Generalising this to include parameters $u$ and $v$, we define
\begin{align}\label{Eq_Phi-three-uv}
\Phi_{n,m}(u,v;z,w;q)&:=\Phi_{n,m}(z/q,w;q) 
-\frac{uz}{(z;q)_2}\, \Phi_{n-1,m}(zq,w/q;q) \\
&\quad+\frac{uvzw^2}{(w,zw;q)_2}\,\Phi_{n-1,m-1}(z,wq;q), \notag
\end{align}
which obviously satisfies \eqref{Eq_11}. 
After clearing denominators, the relation \eqref{Eq_11-b} is a consequence
of the $4$-term relation
\begin{align*}
&c(1-z)(1-w)(1-zw)(1-cdzw) \\
&\qquad =(1-cz)(1-w)(1-czw)(1-dzw)-(1-c)(1-dw)(1-zw)(1-czw) \\
&\qquad\quad +w(1-c)(1-d)(1-z)(1-cz)
\end{align*}
for $(c,d)\mapsto(q^n,q^m)$.
Finally, the relation \eqref{Eq_Phinm_to_Phin} follows from
\eqref{Eq_Phi-rec} and \eqref{Eq_wnul}.
Most importantly, $\Phi_{n,m}(u,v;z,w;q)$ satisfies the following
generalisation of Lemma~\ref{Lem_A1-tree}.

\begin{theorem}[$\mathrm{A}_2$ Bailey tree, part I]
\label{Thm_A2-tree}
For $n,m$ nonnegative integers,
\begin{subequations}
\begin{align}\label{Eq_Bailey-A2-11}
\sum_{r=0}^n \sum_{s=0}^m \mathcal{K}_{n,m;r,s}(z,w;q)
\Phi_{r,s}(1,1;z,w;q)&=\Phi_{n,m}(1,1;z,w;q)
\intertext{and}
\label{Eq_Bailey-A2-uv}
\sum_{r=0}^n \sum_{s=0}^m \mathcal{K}_{n,m;r,s}(z/q,w;q)
\Phi_{r,s}(u,v;z,w;q)&=\Phi_{n,m}(uz,vw;z,w;q).
\end{align}
\end{subequations}
\end{theorem}

By \eqref{Eq_11} the first claim is of course a restatement of 
the $\mathrm{A}_2$ Bailey chain.
We also remark that for $m=0$ or for $w=0$ the theorem simplifies to
Lemma~\ref{Lem_A1-tree}.

The $\mathrm{A}_2$ Bailey tree as stated can only prove 
Theorems~\ref{Thm_KRistrue} and \ref{Thm_Principal-subspace} for $b=0$
(or, by symmetry, $a=0$) and we also need a Bailey-type
transformation for a four-parameter generalisation of $\Phi_{n,m}(z,w;q)$
involving the function $\mathcal{K}_{n,m;r,s}(z/q,w/q;q)$.
This missing part of the $\mathrm{A}_2$ Bailey tree will be discussed
later.

\begin{proof}[Proof of Theorem~\ref{Thm_A2-tree}]
Each side of \eqref{Eq_Bailey-A2-uv} is a polynomial in $u$ and $v$ of the
form $A+Bu+Cuv$.
As in the $\mathrm{A}_1$ case, the constant term of the identity corresponds
to \eqref{Eq_Bailey-A2} with $z$ replaced by $z/q$.
Next, up to an overall factor of $-z/(z;q)_2$, the coefficient of $u$ in
\eqref{Eq_Bailey-A2-uv} is
\[
\sum_{r=1}^n \sum_{s=0}^m \mathcal{K}_{n,m;r,s}(z/q,w;q)
\Phi_{r-1,s}(zq,w/q;q)=z \Phi_{n-1,m}(zq,w/q;q),
\]
where we have used that $\Phi_{r-1,s}$ vanishes for $r=0$.
Shifting the summation index $r\mapsto r+1$ and using that
\[
\mathcal{K}_{n,m;r+1,s}(z/q,w;q)=z\mathcal{K}_{n-1,m;r,s}(zq,w/q;q),
\]
yields \eqref{Eq_Bailey-A2} with $(n,z,w)$ replaced
by $(n-1,zq,w/q)$.
Finally, up to an overall factor of $zw^2/(z,zw;q)_2$, 
the coefficient of $uv$ in \eqref{Eq_Bailey-A2-uv} is given by
\[
\sum_{r=1}^n \sum_{s=1}^m \mathcal{K}_{n,m;r,s}(z/q,w;q)
\Phi_{r-1,s-1}(z,wq;q)=zw\Phi_{n-1,m-1}(z,wq;q).
\]
After shifting $(r,s)\mapsto (r+1,s+1)$ and using that
\[
\mathcal{K}_{n,m;r+1,s+1}(z/q,w;q)=zw\mathcal{K}_{n-1,m-1;r,s}(z,wq;q),
\]
this yields \eqref{Eq_Bailey-A2} with $(n,m,w)\mapsto(n-1,m-1,wq)$.
\end{proof}

To prove Conjecture~\ref{Con_KR-ASW} we need the
$\mathrm{A}_2$-analogues of Corollaries~\ref{Cor_A1-chain}
and~\ref{Cor_A1-tree}.
Unlike the $\mathrm{A}_1$ case, where we used a single integer 
parameter $y$ to parametrise the $\mathrm{A}_1$ root lattice, 
for $\mathrm{A}_2$ we adopt the notation \eqref{Eq_Q} for $r=3$.
That is, for $y=(y_1,y_2,y_3)\in Q$ and $y_{ij}:=y_i-y_j$, we define
\begin{align}\label{Eq_zwprod}
\Phi_{n,m;y}(z,w;q)
&:=\frac{\Phi_{n-y_1,m-y_1-y_2}(zq^{y_{12}},wq^{y_{23}};q)}
{(zq;q)_{y_{12}}(wq;q)_{y_{23}}(zwq;q)_{y_{13}}} \\
&\hphantom{:}=
\frac{(zwq;q)_{n+m}}{(q;q)_{n-y_1}(zq;q)_{n-y_2}(zwq;q)_{n-y_3}
(q;q)_{m+y_3}(wq;q)_{m+y_2}(zwq;q)_{m+y_1}}. \notag
\end{align}
Clearly, $\Phi_{n,m;(0,0,0)}(z,w;q)=\Phi_{n,m}(z,w;q)$
and $\Phi_{n,m;y}(z,w;q)$ vanishes unless $n-y_1\geq 0$ and
$m+y_3=m-y_1-y_2\geq 0$.
Moreover, $\Phi_{n,m;(y_1,-y_1,0)}(z,0;q)=\Phi_{n;y_1}(z;q)/(q;q)_m$.

\begin{corollary}\label{Cor_A2-chain}
For $n,m\in\mathbb{Z}$ and $y=(y_1,y_2,y_3)\in Q$, 
\begin{equation}\label{Eq_A2-chain}
\sum_{r=y_1}^n \sum_{s=y_1+y_2}^m 
\mathcal{K}_{n,m;r,s}(z,w;q) \Phi_{r,s;y}(z,w;q)=
z^{y_1}w^{y_1+y_2}q^{\frac{1}{2}(y_1^2+y_2^2+y_3^2)}
\Phi_{n,m;y}(z,w;q).
\end{equation}
\end{corollary}

\begin{proof}
Replacing  
\begin{equation}\label{Eq_nmzw}
(n,m,z,w)\mapsto (n-y_1,m-y_1-y_2,zq^{y_{12}},wq^{y_{23}})
\end{equation}
in \eqref{Eq_Bailey-A2}, shifting the summation indices
$(r,s)\mapsto (r-y_1,m-y_1-y_2)$ and using 
\begin{equation}\label{Eq_shiftK}
\mathcal{K}_{n-y_1,m-y_1-y_2;r-y_1,s-y_1-y_2}
(zq^{y_{12}},wq^{y_{23}};q)
=z^{-y_1}w^{-y_1-y_2}q^{-\frac{1}{2}(y_1^2+y_2^2+y_3^2)}
\mathcal{K}_{n,m;r,s}(z,w;q)
\end{equation}
as well as definition \eqref{Eq_zwprod}, the claim follows.
\end{proof}

In much the same way we define
\begin{equation}\label{Eq_Phi-nmy-uvzw}
\Phi_{n,m;y}(u,v;z,w;q):=
\frac{\Phi_{n-y_1,m-y_1-y_2}(u,v;zq^{y_{12}},wq^{y_{23}};q)}
{(zq;q)_{y_{12}}(wq;q)_{y_{23}}(zwq;q)_{y_{13}}},
\end{equation}
so that $\Phi_{n,m;(y_1,-y_1,0)}(u,v;z,0;q)=\Phi_{n;y_1}(u;z;q)/(q;q)_m$.
Equation \eqref{Eq_11} implies the simplification
\begin{equation}\label{Eq_Phi-nmy-11zw}
\Phi_{n,m;y}(1,1;z,w;q)=\Phi_{n,m;y}(z,w;q),
\end{equation}
which yields the first of the identities in the next lemma.
The second result follows in an analogous manner as
Corollary~\ref{Cor_A2-chain}, and we omit the proof.

\begin{corollary}\label{Cor_A2-tree}
For $n,m\in\mathbb{Z}$ and $y=(y_1,y_2,y_3)\in Q$,
\begin{subequations}\label{Eq_Bailey-A2-y}
\begin{align}\label{Eq_Bailey-A2-11-Phi}
&\sum_{r=y_1}^n \sum_{s=y_1+y_2}^m
\mathcal{K}_{n,m;r,s}(z,w;q)\Phi_{r,s;y}(1,1;z,w;q) \\
&\qquad\qquad = 
z^{y_1}w^{y_1+y_2}q^{\frac{1}{2}(y_1^2+y_2^2+y_3^2)}
\Phi_{n,m;y}(1,1;z,w;q) \notag
\intertext{and}
&\sum_{r=y_1}^n \sum_{s=y_1+y_2}^m \mathcal{K}_{n,m;r,s}(z/q,w;q)
\Phi_{r,s;y}(u,v;z,w;q) \label{Eq_Bailey-A2-uv-Phi} \\
&\qquad\qquad = 
z^{y_1}w^{y_1+y_2}q^{\frac{1}{2}(y_1^2+y_2^2+y_3^2)-y_1}
\Phi_{n,m;y}(uzq^{y_{12}},vwq^{y_{23}};z,w;q). \notag
\end{align}
\end{subequations}
\end{corollary}

\medskip

As mentioned previously, our $\mathrm{A}_2$ Bailey tree is not yet 
complete.
Conjecture~\ref{Con_KR-ASW} and Theorem~\ref{Thm_Principal-subspace}
contain three integer parameters $a,b$ and $k$.
Theorem~\ref{Thm_A2-tree}, however, is restricted to paths along the
Bailey tree of the form shown on page~\pageref{page_path}.
Since such paths can be described by two parameters, something is
still missing.
The reason for deferring the treatment of the missing part of the
$\mathrm{A}_2$ Bailey tree till now is that it uses most of the
previously-defined functions and is less intuitive than what has
been discussed so far.

For $n,m\in\mathbb{Z}$ and $\rho:=(1,2,3)$, define
\begin{align}\label{Eq_Phiuvcd}
&\Phi_{n,m}(u,v;c,d;z,w;q) \\
&:=\sum_{\sigma\in S_3} 
\sgn(\sigma)\,(uz)^{\sigma_1-1} \Big(\frac{v}{d}\Big)^{\chi(\sigma_3=1)} 
\Big(\frac{c}{u}\Big)^{\chi(\sigma_1=3)} (dw)^{3-\sigma_3} 
\Phi_{n,m;\sigma-\rho}(z/q,w/q;q). \notag
\end{align}
Since the summand contains the factors $(q;q)_{n-\sigma_1+1}$ and
$(q;q)_{m+\sigma_3-3}$ in the denominator, the function
$\Phi_{n,m}(u,v;c,d;z,w;q)$ vanishes unless $n,m\geq 0$. 
If $n=m=0$ then only the term $\sigma=\rho$ contributes to the sum so
that $\Phi_{0,0}(u,v;c,d;z,w;q)=1$.
By replacing $\sigma=(\sigma_1,\sigma_2,\sigma_3)\mapsto 
(4-\sigma_3,4-\sigma_2,4-\sigma_1)$, and using that
$\Phi_{n,m;(y_1,y_2,y_3)}(z,w;q)=\Phi_{m,n;-(y_3,y_2,y_1)}(w,z;q)$
and $\sigma_1+\sigma_2+\sigma_3=6$, it may also be seen that
\begin{equation}\label{Eq_symmetry}
\Phi_{n,m}(u,v;c,d;z,w;q)=\Phi_{m,n}(d,c;v,u;w,z;q).
\end{equation}

Before proving a number of important properties of 
$\Phi_{n,m}(u,v;c,d;z,w;q)$, including a Bailey-type transformation,
we remark that in the $n,m\to\infty$ limit an important special case
of this function is essentially a Schur function~\cite{Macdonald95}.

\begin{lemma}\label{Lem_nmlimit}
Let $z:=x_1/x_2$, $w:=x_2/x_3$ and for $a,b$ nonnegative integers,
let $\nu:=(a+b+2,b+1,0)$.
Then
\begin{equation}\label{Eq_nmlimit}
\lim_{n,m\to\infty}\Phi_{n,m}(z^a,w^a;z^b,w^b;z,w;q)=
\frac{1}{(q,q,z,w,zw/q;q)_{\infty}}\,
\det_{1\leq i,j\leq 3}\big(x_i^{\nu_i-\nu_j}\big)
\end{equation}
\end{lemma}

\begin{proof}
The large-$n,m$ limit of $\Phi_{n,m;\sigma-\rho}(z/q,w/q;q)$
gives the infinite product on the right of \eqref{Eq_nmlimit}.
Moreover, for $(u,v,c,d)=(z^a,w^a;z^b,w^b)$ the sum over $S_3$ in
the definition of \eqref{Eq_Phiuvcd} becomes
\[
\sum_{\sigma\in S_3} 
\sgn(\sigma)\,z^{(a+1)(\sigma_1-1)-(a-b)\chi(\sigma_1=3)} 
w^{(a-b)\chi(\sigma_3=1)-(b+1)(\sigma_3-3)}
=\sum_{\sigma\in S_3} 
\sgn(\sigma)\,\prod_{i=1}^3 x_i^{\nu_i-\nu_{\sigma_i}},
\]
which is the determinant in the numerator.
\end{proof}

Unlike $\Phi_{n,m}(u,v;z,w;q)$, which simplifies to
$\Phi_{n,m}(z,w;q)$ for $u=v=1$, the function
$\Phi_{n,m}(u,v;c,d;z,w;q)$ for $c=d=1$ does
not simplify to $\Phi_{n,m}(u,v;z,w;q)$.
Instead a simple linear combination of
$\Phi_{n,m}(u,v;z,w;q)$ and
$\Phi_{n,m}(u/z,v/w;z,w;q)$ arises.

\begin{lemma}\label{Lem_cd1}
For $n,m\in\mathbb{Z}$,
\begin{equation}\label{Eq_cd1}
\Phi_{n,m}(u,v;1,1;z,w;q) \\
=\frac{\Phi_{n,m}(u,v;z,w;q)
-zwq^{m-1}\Phi_{n,m}(u/z,v/w;z,w;q)}{1-zwq^{-1}}.
\end{equation}
\end{lemma}

By \eqref{Eq_Phi-uv-special}, the $u=v=1$ case of the lemma simplifies to
\begin{equation}\label{Eq_Phi-nm-1111zw}
\Phi_{n,m}(1,1;1,1;z,w;q)
=\frac{1-zwq^{n+m-1}}{1-zwq^{-1}}\, \Phi_{n,m}(z,w;q).
\end{equation}

\begin{proof}
Both sides of \eqref{Eq_cd1} are polynomials in $u$ and $v$.
Equating like coefficients using \eqref{Eq_Phi-three-uv}, the
claim splits into three separate equations.
After normalisation these are
\[
\frac{1-zwq^{m-1}}{1-zwq^{-1}} \,\Phi_{n,m}(z/q,w;q)
=\Phi_{n,m;(0,0,0)}(z/q,w/q;q)-w\Phi_{n,m;(0,1,-1)}(z/q,w/q;q),
\]
\begin{align*}
&\frac{1-wq^{m-1}}{1-zwq^{-1}} \,\Phi_{n-1,m}(zq,w/q;q) \\
&\qquad=(z;q)_2 \Big(\Phi_{n,m;(1,-1,0)}(z/q,w/q;q)-zw
\Phi_{n,m;(2,-1,-1)}(z/q,w/q;q)\Big),
\intertext{and}
&\frac{1-q^{m-1}}{1-zwq^{-1}} \,\Phi_{n-1,m-1}(z,wq;q) \\
&\qquad=(w,zw;q)_2\Big(
\Phi_{n,m;(1,1,-2)}(z/q,w/q;q)
-z\Phi_{n,m;(2,0,-2)}(z/q,w/q;q)\Big),
\end{align*}
corresponding to the coefficients of $u^0v^0$, $u^1v^0$ and $u^1v^1$
respectively.
By the definitions of $\Phi_{n,m}(z,w;q)$ and $\Phi_{n,m;y}(z,w;q)$
given in \eqref{Eq_Phi} and \eqref{Eq_zwprod}, all three equations
are readily verified.
\end{proof}

The missing part of the $\mathrm{A}_2$ Bailey tree can now be
stated as follows.

\begin{theorem}[$\mathrm{A}_2$ Bailey tree, part II]
\label{Thm_A2-tree-2}
For $n,m$ nonnegative integers,
\begin{equation}\label{Eq_tree-2}
\sum_{r=0}^n \sum_{s=0}^m \mathcal{K}_{n,m;r,s}(z/q,w/q;q)
\Phi_{r,s}(u,v;c,d;z,w;q)=\Phi_{n,m}(uz,vw;cz,dw;z,w;q).
\end{equation}
\end{theorem}

Once again consider the tree on page~\pageref{page_tree}.
In view of Lemma~\ref{Lem_cd1}, we can first apply the Bailey
tree of Theorem~\ref{Thm_A2-tree}, starting at the root
and taking $k-a$ south-west steps followed by $a-b$ south-east steps.
This gives the same path along the tree as shown on
page~\pageref{page_path} but with $(k,a)\mapsto (k-b,a-b)$,
so that the final vertex is labelled $k-b$.
Next we can repeat the above but with $a$ replaced by $a-1$,
resulting in the path along the tree shown on
page~\pageref{page_path} with $(k,a)\mapsto (k-b,a-b-1)$,
so that the final vertex is once again labelled $k-b$.
As a third and final step we can then take a linear combination
of the pair of identities represented by the two vertices
labelled $k-b$ and take a further $b$ steps using part II of 
the Bailey tree.
If we think of south-east steps as unit steps in $\mathbb{R}^3$
in the positive $x$-direction, south-west steps as steps in the
positive $y$-direction and the final $b$ steps as steps in the positive
$z$-direction, the above procedure can be represented by the 
three-dimensional diagram

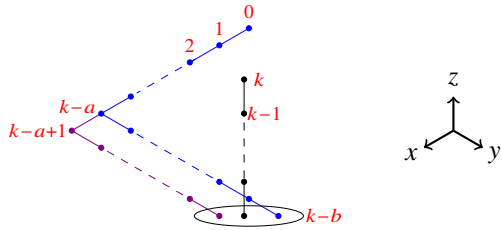
\tikzmath{\x=0.866;}
\tikzmath{\y=0.5;}
\begin{center}
\begin{tikzpicture}[scale=0.45,line width=0.3pt]
\draw[blue] (0,0)--(-2*\x,-2*\y);
\draw[dashed,blue] (-1.8*\x,-1.8*\y)--(-3.8*\x,-3.8*\y);
\draw[blue] (-4*\x,-4*\y)--(-5*\x,-5*\y)--(-4*\x,-6*\y);
\draw[violet] (-5*\x,-5*\y)--(-6*\x,-6*\y)--(-5*\x,-7*\y);
\draw[dashed,blue] (-3.8*\x,-6.2*\y)--(-1.2*\x,-8.8*\y);
\draw[dashed,violet] (-4.8*\x,-7.2*\y)--(-2.2*\x,-9.8*\y);
\draw[blue] (-1*\x,-9*\y)--(1*\x,-11*\y);
\draw[violet] (-2*\x,-10*\y)--(-1*\x,-11*\y);
\foreach \z in {0,-1,-2,-4,-5} {\filldraw[blue] (\z*\x,\z*\y) circle (0.75mm);}
\foreach \z in {-6,-9,-10,-11} {\filldraw[blue] (-10*\x-\z*\x,\z*\y) circle (0.75mm);}
\foreach \z in {-6,-7,-10,-11} {\filldraw[violet] (-12*\x-\z*\x,\z*\y) circle (0.75mm);}
\draw[red] (0,0.5) node {$\scriptstyle 0$};
\draw[red] (-\x,-\y+0.5) node {$\scriptstyle 1$};
\draw[red] (-2*\x,-2*\y+0.5) node {$\scriptstyle 2$};
\draw[red] (-5*\x-0.7,-5*\y+0.2) node {$\scriptstyle k-a$};
\draw[red] (-6*\x-1,-6*\y) node {$\scriptstyle k-a+1$};
\draw[red] (2.2,-11*\y) node {$\scriptstyle k-b$};
\draw[rotate=0] (0,-11*\y) ellipse (1.6 and 0.3);
\draw[black] (-0.14,-11*\y)--(-0.14,-9*\y);
\draw[black] (-0.14,-5*\y)--(-0.14,-3*\y);
\draw[dashed,black] (-0.14,-8.8*\y)--(-0.14,-5.2*\y);
\foreach \z in {-11,-9,-5,-3} {\filldraw[black] (-0.14,\z*\y) circle (0.75mm);}
\draw[red] (0.3,-3*\y) node {$\scriptstyle k$};
\draw[red] (0.45,-5*\y) node {$\scriptstyle k-1$};
\draw[thick,->] (6,-3)--(6,-2);
\draw[thick,->] (6,-3)--(6-\x,-3-\y);
\draw[thick,->] (6,-3)--(6+\x,-3-\y);
\draw (6,-1.5) node {$z$};
\draw (6-1.4*\x,-3-1.4*\y) node {$x$};
\draw (6+1.4*\x,-3-1.4*\y) node {$y$};
\end{tikzpicture}
\end{center}
\label{page_sketch}
where the central black vertex in the encircled region represents the
appropriate linear combination of the violet and blue vertices labelled
$k-b$.

\begin{proof}[Proof of Theorem~\ref{Thm_A2-tree-2}]
Denote the left-hand side of \eqref{Eq_tree-2} by $\phi_{n,m}$.
By \eqref{Eq_Phiuvcd} and an interchange in the order of the
sums over $r,s$ and over $\sigma$,
\begin{align*}
\phi_{n,m}&=\sum_{\sigma\in S_3} \bigg( 
\sgn(\sigma)\,(uz)^{\sigma_1-1} (v/d)^{\chi(\sigma_3=1)} 
(c/u)^{\chi(\sigma_1=3)} (dw)^{3-\sigma_3} \\
& \qquad\qquad \times
\sum_{r=0}^n \sum_{s=0}^m \mathcal{K}_{n,m;r,s}(z/q,w/q;q)
\Phi_{r,s;\sigma-\rho}(z/q,w/q;q) \bigg).
\end{align*}
We now use that
$\Phi_{r,s;\sigma-\rho}(z/q,w/q;q)=0$
unless $r-\sigma_1+1\geq 0$ and $s+\sigma_3-3\geq 0$
to change the lower bounds on the sums over $r$ and $s$ to
$\sigma_1-1$ and $3-\sigma_3$ respectively.
Since Corollary~\ref{Cor_A2-chain} for $y=\sigma-\rho$ simplifies to
\[
\sum_{r=\sigma_1-1}^n \sum_{s=3-\sigma_3}^m
\mathcal{K}_{n,m;r,s}(z,w;q) \Phi_{r,s;\sigma-\rho}(z,w;q) 
=(zq)^{\sigma_1-1}(wq)^{3-\sigma_3}
\Phi_{n,m;\sigma-\rho}(z,w;q),
\]
it follows that
\begin{align*}
\phi_{n,m}&=\sum_{\sigma\in S_3}
\sgn(\sigma)\,(uz^2)^{\sigma_1-1} \Big(\frac{v}{d}\Big)^{\chi(\sigma_3=1)} 
\Big(\frac{c}{u}\Big)^{\chi(\sigma_1=3)} (dw^2)^{3-\sigma_3}
\Phi_{n,m;\sigma-\rho}(z/q,w/q;q) \\
&=\Phi_{n,m}(uz,vw;cz,dw;z,w;q). \qedhere
\end{align*}
\end{proof}

To conclude the section we define $y$-analogue of \eqref{Eq_Phiuvcd}:
\begin{equation}\label{Eq_Phi-nmy-uvcdzw}
\Phi_{n,m;y}(u,v;c,d;z,w;q):=
\frac{\Phi_{n-y_1,m-y_1-y_2}(u,v;c,d;zq^{y_{12}},wq^{y_{23}};q)}
{(z;q)_{y_{12}}(w;q)_{y_{23}}(zw/q;q)_{y_{13}}},
\end{equation}
where $y=(y_1,y_2,y_3)\in Q$.
It follows from Lemma~\ref{Lem_nmlimit} that for
$z:=x_1/x_2$ and $w:=x_2/x_3$, 
\begin{align}\label{Eq_nmlimit-y}
& \lim_{n,m\to\infty}\Phi_{n,m;y}\big((zq^{y_{12}})^a,(wq^{y_{23}})^a;
(zq^{y_{12}})^b,(wq^{y_{23}})^b;z,w;q\big) \\
&\qquad = \frac{1}{(q,q,z,w,zw/q;q)_{\infty}}\,
\det_{1\leq i,j\leq 3}\big((x_iq^{y_i})^{\nu_i-\nu_j}\big), \notag
\end{align}
where $\nu=(a+b+2,b+1,0)$.
Furthermore, noting the minor difference in denominators on the right
of \eqref{Eq_Phi-nmy-uvzw} and \eqref{Eq_Phi-nmy-uvcdzw},
it follows that the special case of Lemma~\ref{Lem_cd1} given in
\eqref{Eq_Phi-nm-1111zw} admits the $y$-generalisation
\begin{equation}\label{Eq_Phi-nmy-1111zw}
\Phi_{n,m;y}(1,1;1,1;z,w;q)=
\frac{1-zwq^{n+m-1}}{1-zwq^{-1}}\,\Delta_y(z,w;q) \Phi_{n,m;y}(z,w;q),
\end{equation}
where
\begin{equation}\label{Eq_Delta}
\Delta_y(z,w;q)
:=\frac{(1-zq^{y_{12}})(1-wq^{y_{23}})(1-zwq^{y_{13}})}
{(1-z)(1-w)(1-zw)}.
\end{equation}
Finally, the $y$-analogues of Lemma~\ref{Lem_cd1}
and Theorem~\ref{Thm_A2-tree-2} follow from
\eqref{Eq_Phi-nmy-uvzw} and \eqref{Eq_shiftK} respectively, 

\begin{corollary}\label{Cor_cd1-y}
For $n,m\in\mathbb{Z}$ and $y=(y_1,y_2,y_3)\in Q$,
\begin{align}\label{Eq_cd1-y}
&\Phi_{n,m;y}(u,v;1,1;z,w;q)=\Delta_y(z,w;q) \\
&\qquad\times
\frac{\Phi_{n,m;y}(u,v;z,w;q)
-zwq^{m+y_1-1}\Phi_{n,m;y}(uq^{-y_{12}}/z,vq^{-y_{23}}/w;z,w;q)}
{1-zwq^{-1}}. \notag
\end{align}
\end{corollary}

\begin{corollary}\label{Cor_A2-tree-2}
For $n,m\in\mathbb{Z}$ and $y=(y_1,y_2,y_3)\in Q$,
\begin{align*}
&\sum_{r=y_1}^n \sum_{s=y_1+y_2}^m \mathcal{K}_{n,m;r,s}(z/q,w/q;q)
\Phi_{r,s;y}(u,v;c,d;z,w;q) \\
&\qquad = 
z^{y_1}w^{y_1+y_2}q^{\frac{1}{2}(y_1^2+y_2^2+y_3^2)-2y_1-y_2}
\Phi_{n,m;y}(uzq^{y_{12}},vwq^{y_{23}};czq^{y_{12}},dwq^{y_{23}};z,w;q).
\end{align*}
\end{corollary}

\subsection{$\mathrm{A}_2$ Bailey pairs}\label{Subsec_A2-Bailey-pairs}
In this section we briefly discuss the notion of $\mathrm{A}_2$
Bailey pairs.
We will not, however, translate all of the various $\mathrm{A}_2$
Bailey transformations of Section~\ref{Subsec_A2-tree} in terms
of such pairs.
Throughout, \eqref{Eq_Q} is used for $r=3$.

Let
\begin{align*}
\alpha(z,w;q)&=\big(\alpha_y(z,w;q)\big)_{y\in Q_{+}}, \\
\beta(z,w;q)&=\big(\beta_{n,m}(z,w;q)\big)_{n,m\in\mathbb{N}_0}
\end{align*}
be a pair of sequences such that
\begin{equation}\label{Eq_BP-A2}
\beta_{n,m}(z,w;q)=\sum_{y\in Q_{+}}\Phi_{n,m;y}(z,w;q)\alpha_y(z,w;q).
\end{equation}
Then we say that $(\alpha(z,w;q),\beta(z,w;q))$ is an
$\mathrm{A}_2$ Bailey pair relative to $z,w$.
Note that in the above definition only those $y\in Q_{+}$
contribute to the sum on the right for which
$y_1\leq n$ and $y_1+y_2\leq m$.
Definition \eqref{Eq_BP-A2} is not the same as the one adopted
in \cite{ASW99}, where $Q_{++}$ was used instead of $Q_{+}$.
Further define
\begin{align*}
\Psi_{y;r,s}(z,w;q)&:=q^{r+s-y_{13}} 
(zq;q)_{y_{12}}(wq;q)_{y_{23}}(zwq;q)_{y_{13}} \\[1mm]
&\qquad\qquad\times
\Phi_{y_1-r,y_1+y_2-s}\big(zq^{y_{12}},wq^{y_{23}};q^{-1}\big),
\end{align*}
for $r,s\in\mathbb{N}_0$ and $y\in Q_{+}$.
Recalling \eqref{Eq_zwprod}, the $\mathrm{A}_2$ inversion relation 
\eqref{Eq_inversion} may then be written as
\[
\sum_{y\in Q_{+}} \Phi_{n,m;y}(z,w;q)\Psi_{y;N,M}(z,w;q)
=\delta_{n,N}\delta_{m,M},
\]
for $n,m,N,M\in\mathbb{N}_0$.
Since $\Phi_{n,m;y}(z,w;q)$ vanishes unless $n\geq y_1$ and $m\geq y_1+y_2$
and $\Psi_{y;N,M}(z,w;q)$ vanishes unless $y_1\geq N$ and $y_1+y_2\geq M$,
the summand on the left is only supported on $y\in Q_{+}$ such that
$N\leq y_1\leq n$ and $M\leq y_1+y_2\leq m$.
Similarly, it follows that for $y,Y\in Q_{+}$,
\[
\sum_{r,s\in\mathbb{N}_0}
\Psi_{y;r,s}(z,w;q)\Phi_{r,s;Y}(z,w;q)=\delta_{y,Y},
\]
with summand supported on $Y_1\leq r\leq y_1$ and 
$Y_1+Y_2\leq s\leq y_1+y_2$.
The relation \eqref{Eq_BP-A2} is thus invertible, so that
\begin{equation}\label{Eq_BP-A2-inv}
\alpha_y(z,w;q)=
\sum_{r,s\in\mathbb{N}_0} \Psi_{y;r,s}(z,w;q)\beta_{r,s}(z,w;q).
\end{equation}
This in turn gives rise to what may be called the 
$\mathrm{A}_2$ unit Bailey pair:
\begin{equation}\label{Eq_UBP}
\alpha_y(z,w;q)=\Psi_{y;0,0}(z,w;q) \quad\text{and}\quad 
\beta_{n,m}(z,w;q)=\delta_{n,0}\delta_{m,0}.
\end{equation}
If $\alpha_n(z;q):=\alpha_{(n,-n,0)}(z,0;q)$ and
$\beta_n(z;q):=\beta_{n,0}(z,0;q)$, then \eqref{Eq_BP-A2} for $m=w=0$
and \eqref{Eq_BP-A2-inv} for $y=(n,-n,0)$ and $w=0$ correspond to 
\eqref{Eq_BP-A1} and \eqref{Eq_BP-A1-inv} respectively.

For a number of $\mathrm{A}_2$ Bailey pairs, such as the
unit Bailey pair \eqref{Eq_UBP} or pairs that follow from the
unit Bailey pair by application of \eqref{Eq_A2-chain},
the definition \eqref{Eq_BP-A2} is perfectly useable.
However, the explicit form of many $\mathrm{A}_2$ Bailey pairs
is rather unwieldy, making the definition not particularly
practical.
A good example is the Bailey pair implied by the identity 
\eqref{Eq_seed} of the next section.
This identity corresponds to the root of the tree of identities
on which our proof of the Kanade--Russell conjecture is based.
It is quite artificial, and not at all enlightening, to write the
left-hand side of \eqref{Eq_seed} as a sum over $Q_{+}$ ---
which is necessary in order to read off $\alpha_y$ ---
instead of $Q$.


\section{Proof of the Kanade--Russell conjecture}\label{Sec_KR}

Before we can apply the $\mathrm{A}_2$ Bailey tree to prove
Conjecture~\ref{Con_KR-ASW} we need a suitable identity playing the
role of root in the Bailey tree.
This root identity is given by the $\mathrm{A}_2$-analogue
of~\eqref{Eq_Slater2}.
Before stating the actual result, we note that for $n,m\in\mathbb{Z}$
and $y=(y_1,y_2,y_3)\in Q$,
\begin{equation}
\Phi_{n,m;y}(q,q;q):=\lim_{z,w\to 1} \Phi_{n,m;y}(zq,wq;q)=
\frac{(q;q)_1^2(q;q)_2}{(q;q)_{n+m+2}^2}
\prod_{i=1}^3 \qbin{n+m+2}{n-y_i+i-1},
\end{equation}
which vanishes unless $i-m-3\leq y_i\leq n+i-1$ for all $1\leq i\leq 3$.
The reason $\Phi_{n,m;y}(q,q;q)$ is defined as a limit
is that the term $(zwq^3;q)_{n+m}$ in the numerator of
$\Phi_{n,m;y}(zq,wq;q)$ has a simple pole at $zw=1$ if $n+m+2<0$
(for $n+m+2\geq 0$ the function $\Phi_{n,m;y}(z,w;q)$ is regular at
$z=w=1$).
This pole has zero residue and the above expression on the right arises.
Moreover, it follows from the above inequalities for the $y_i$ that the
only instances where $\Phi_{n,m;y}(q,q;q)$ is nonvanishing for
$\min\{n,m\}<0$ correspond to $y=(-1,0,1)$ and $\min\{n,m\}=-1$.
This in particular implies that if $t$ is an integer greater than $1$,
then $\Phi_{n,m;ty}(q,q;q)$ vanishes if $(n,m)\not\in\mathbb{N}_0^2$.

Recall definition \eqref{Eq_Delta} of $\Delta_y$.

\begin{proposition}
Let $n,m\in\mathbb{N}_0$, $\tau\in\{-1,0,1\}$ and
\[
g_{n,m;\tau}(q):=
\frac{q^{\tau(\tau-1)nm}}{(q,q^2;q)_{n+m}}\qbin{n+m}{n}_p,
\]
where $p=q$ if $\tau^2=1$ and $p=q^3$ if $\tau=0$.
Then \label{page_seed}
\begin{equation}\label{Eq_seed}
\sum_{y\in Q} \Phi_{n,m;3y}(q,q;q) \Delta_{3y}(q,q;q)
\prod_{i=1}^3 q^{3(3+\tau)\binom{y_i}{2}-\tau iy_i}=g_{n,m;\tau}(q).
\end{equation}
\end{proposition}

The above definition of $g_{n,m;\tau}(q)$ is the same as \eqref{Eq_g}
of the introduction.

\begin{proof}
Recall that $y_{ij}:=y_i-y_j$.
The identity \eqref{Eq_seed} for $\tau=1$ is a bounded form of the
$\mathrm{A}_2$-analogue of Euler's penta\-gonal number theorem, stated
in \cite[page 692]{ASW99} in the form
\begin{align}\label{Eq_page692}
&\sum_{y\in Q} \prod_{1\leq i<j\leq 3} (1-q^{3y_{ij}+j-i})
\prod_{i=1}^3 q^{12\binom{y_i}{2}-iy_i} \qbin{n+m+2}{n-3y_i+i-1} \\
&\qquad=(1-q^{n+m+1})(1-q^{n+m+2})^2\,\qbin{n+m}{n}, \notag
\end{align}
for $n,m\in\mathbb{N}_0$.
The proof of \eqref{Eq_page692} given in \cite{ASW99} is very involved.
First an identity for what are known as supernomial coefficients is
established (the $\ell=0$ case of \cite[Equation (5.3)]{ASW99}).
This is then transformed using an $\mathrm{A}_2$ Bailey lemma
for supernomial coefficients, resulting in \cite[Equation (5.15)]{ASW99}.
Using the determinant evaluation \eqref{Eq_Kratt_b} below, this finally
yields~\eqref{Eq_page692}.
In the appendix we present a much simpler proof which implies that
\eqref{Eq_page692} arises by taking the constant term with respect to $z$ in
the $r=3$ instance of the identity
\begin{align}\label{Eq_CT}
&\sum_{y_1,\dots,y_r\in\mathbb{Z}} 
\prod_{1\leq i<j\leq r} (1-q^{ry_{ij}+j-i})
\prod_{i=1}^r (-1)^{ry_i} z^{y_i} q^{\binom{r+1}{2}y_i^2-iy_i} 
\qbin{n+m+r-1}{n-ry_i+i-1} \\
&\qquad=\bigg(\,\prod_{i=1}^{r-1}(1-q^{n+m+i})^i\bigg)
\sum_{k=-m}^n (-1)^k z^k q^{(r+1)\binom{k}{2}}\qbin{n+m}{n-k}.
\notag
\end{align}
This last result is a polynomial analogue of the classical theta
function identity
\[
\det_{1\leq i,j\leq r} \bigg( q^{i(i-j)} 
\theta\Big(z\big({-}q^{-i}\big)^{r+1} 
q^{rj+\binom{r+1}{2}};q^{r(r+1)}\Big)\bigg) 
=\frac{(q^{r+1};q^{r+1})_{\infty}(q;q)_{\infty}^{r-1}}
{(q^{r(r+1)};q^{r(r+1)})_{\infty}^r}\, \theta\big(z;q^{r+1}\big).
\]

Replacing $q\mapsto 1/q$ in \eqref{Eq_seed} for $\tau=-1$, and
then using $\qbin{n+m}{n}_{1/q}=q^{-nm} \qbin{n+m}{n}$ as well as
\[
\Phi_{n,m;y}\big(z^{-1},w^{-1};q^{-1}\big)
=z^{n+2y_1} w^{m-2y_3} 
q^{n^2-nm+m^2+n+m+\sum_{i=1}^3 y_i^2}\,\Phi_{n,m;y}(z,w;q)
\]
where $y\in Q$, gives \eqref{Eq_seed} for $\tau=-1$.

Finally, according to \cite[Equation (6.18)]{GK97},
\[
\sum_{y\in rQ} \det_{1\leq i,j\leq r} \bigg(
q^{\binom{y_i}{2}+(j-i)(j+y_i)}\qbin{n+m}{n-y_i+i-j}\bigg) 
=\qbin{n+m}{n}_{q^r}.
\]
By \cite[page 189]{Krattenthaler90} 
\begin{align}\label{Eq_Kratt_b}
&\det_{1\leq i,j\leq r} \bigg(q^{(j-i)(j+i+b_i)}
\qbin{n+m}{n-b_i-j}\bigg) \\ 
&\qquad =\prod_{1\leq i<j\leq r}(1-q^{b_i-b_j})
\prod_{i=1}^r \frac{1}{(q^{n+m+i};q)_{r-i}}
\qbin{n+m+r-1}{n-b_i-1} \notag
\end{align}
for $n,m,b_1,\dots,b_r\in\mathbb{Z}$, this can be recast as
\begin{equation}\label{Eq_KG}
\sum_{y\in rQ} \prod_{1\leq i<j\leq r}(1-q^{y_{ij}+j-i})
\prod_{i=1}^r q^{\binom{y_i}{2}} \qbin{n+m+r-1}{n-y_i+i-1} 
=\qbin{n+m}{n}_{q^r} \prod_{i=1}^{r-1} (1-q^{n+m+i})^i.
\end{equation}
For $r=3$ this is \eqref{Eq_seed} with $\tau=0$.
\end{proof}

\begin{remark}
Although \eqref{Eq_seed} is the natural $\mathrm{A}_2$-analogue of
\eqref{Eq_Slater2}, there is a notable difference between the $\tau=1$
instances of these identities.
From \eqref{Eq_BP-A1-inv} we may infer what is known as the 
$\mathrm{A}_1$ unit Bailey pair:
\[
\alpha_n(z;q)=\frac{1-zq^{2n}}{1-z}\,
(-1)^n q^{\binom{n}{2}} \quad\text{and}\quad \beta_n(z;q)=\delta_{n,0}.
\]
By \eqref{Eq_Bailey-pair-2} this yields\footnote{Alternatively, this
follows after specialising $N=0$ in \eqref{Eq_inversion-A1}.}
\[
\sum_{r=0}^n \frac{1-zq^{2r}}{1-z}\,
(-1)^r q^{\binom{r}{2}} \Phi_{n;r}(z;q)=\delta_{n,0}.
\]
Applying Corollary~\eqref{Cor_A1-chain} then gives
\[
\sum_{r=0}^n \frac{1-zq^{2r}}{1-z}\,
(-1)^r z^r q^{3\binom{r}{2}+r} \Phi_{n;r}(z;q)=\frac{1}{(q;q)_n},
\]
which for $z=q$ is the same as \eqref{Eq_Slater} (and hence 
\eqref{Eq_Slater2}) for $\tau=1$.
The $\tau=1$ case of \eqref{Eq_seed}, however, does \emph{not} follow
from the $\mathrm{A}_2$ unit Bailey pair \eqref{Eq_UBP}.
Indeed, the once-iterated $\mathrm{A}_2$ unit Bailey pair gives the
$k=a$ case of \eqref{Eq_step1}, which has $1/((q;q)_n(q;q)_m)$ as
right-hand side, not $g_{n,m;1}(q)$.
Instead, \eqref{Eq_seed} for $\tau=1$ follows from the $\mathrm{A}_2$
unit Bailey pair for supernomial coefficients, see~\cite{ASW99,W99}.
\end{remark}

Equipped with the identity \eqref{Eq_seed} we can prove 
Conjecture~\ref{Con_KR-ASW}.
The essence of the proof is encoded in the diagram on 
page~\pageref{page_sketch},
where the root identity corresponds to the vertex labelled $0$ and the
Kanade--Russell conjecture \eqref{Eq_KR-ASW} corresponds to the vertex
labelled $k$.

\begin{proof}[Proof of Conjecture~\ref{Con_KR-ASW}]
In view of the discussion regarding $\Phi_{n,m}(q,q;q)$ at the beginning
of this section, if in Corollary~\ref{Cor_A2-chain} we restrict $n,m$ to
nonnegative integers and replace $y\mapsto ty$ for $y\in Q$ where
$t$ is an integer greater than $1$, then the resulting transformation
can be written as
\begin{equation}\label{Eq_A2-chain-zwisq}
\sum_{r=0}^n \sum_{s=0}^m 
\mathcal{K}_{n,m;r,s}(q,q;q) \Phi_{r,s;ty}(q,q;q)=
\Phi_{n,m;ty}(q,q;q)
\prod_{i=1}^3 q^{t^2\binom{y_i}{2}-tiy_i}.
\end{equation}
Indeed, the summand on the left vanishes unless
$r\geq\max\{0,ty_1\}$ and $s\geq\max\{0,ty_1+ty_2\}$ so that
\eqref{Eq_A2-chain-zwisq} is consistent with \eqref{Eq_A2-chain}.
(The transformation \eqref{Eq_A2-chain-zwisq} fails
for $n,m\in\mathbb{N}_0$ and $t=1$ when $y=(-1,0,1)$, requiring 
a lower bound of $-1$ in the summations over $r$ and $s$ instead
of $0$.)

Now let $a,k$ be integers such that $a\leq k$.
(Initially only $k-a$ is required to be a nonnegative integer, but
there is no loss of generality in assuming integrality of $a$ and $k$
from the outset.)
Then, by a $(k-a)$-fold application of \eqref{Eq_A2-chain-zwisq} with
$t=3$, the root identity \eqref{Eq_seed} transforms into 
\begin{align}\label{Eq_not-compact}
&\sum_{\substack{\la\subseteq (n^{k-a}) \\ \mu\subseteq (m^{k-a})}}
g_{\la_{k-a},\mu_{k-a};\tau}(q) 
\prod_{i=1}^{k-a} 
\mathcal{K}_{\la_{i-1},\mu_{i-1};\la_i,\mu_i}(q,q;q) \\[-2mm]
&\qquad=\sum_{y\in Q} \Phi_{n,m;3y}(q,q;q)\Delta_{3y}(q,q;q)
\prod_{i=1}^3 q^{3(K-3a)\binom{y_i}{2}-(K-3a-3)iy_i}, \notag
\end{align}
where $\la_0:=n$, $\mu_0:=m$, $n,m\in\mathbb{N}_0$ and $K:=3k+3+\tau$.
This is the identity represented by the vertex labelled $k-a$ in the
diagram on page~\pageref{page_sketch}.
For later reference we note that by \eqref{Eq_Phi-nmy-1111zw} the above
may also be stated as
\begin{align}\label{Eq_compact}
& \frac{1-q^{n+m+1}}{1-q}
\sum_{\substack{\la\subseteq (n^{k-a})\\\mu\subseteq (m^{k-a})}}
g_{\la_{k-a},\mu_{k-a};\tau}(q)  
\prod_{i=1}^{k-a} \mathcal{K}_{\la_{i-1},\mu_{i-1};\la_i,\mu_i}(q,q;q) \\[-2mm]
&\qquad=\sum_{y\in Q} \Phi_{n,m;3y}(1,1;1,1;q,q;q) 
\prod_{i=1}^3 q^{3(K-3a)\binom{y_i}{2}-(K-3a-3)iy_i}. \notag
\end{align}

In the remainder of the proof we will use the shorthand
\[
Z_t:=q^{ty_{12}+1} \quad \text{and} \quad W_t:=q^{ty_{23}+1},
\]
where $t$ is an integer greater than $1$.
We then make the simultaneous substitutions
\[
(u,v,z,w,y)\mapsto \big(Z_t^{\ell-1},W_t^{\ell-1},q,q,ty\big)
\]
in \eqref{Eq_Bailey-A2-uv-Phi} for $n,m\in\mathbb{N}_0$.
By
\begin{equation}\label{Eq_zw-scaled}
\mathcal{K}_{n,m;r,s}(az,bw;q)=a^rb^s\mathcal{K}_{n,m;r,s}(z,w;q),
\end{equation}
for $z=w=q$ and $(a,b)=(1/q,1)$, this yields
\begin{align}\label{Eq_latticeqq}
&\sum_{r=0}^n \sum_{s=0}^m q^{-r} \mathcal{K}_{n,m;r,s}(q,q;q)
\Phi_{r,s;ty}
\big(Z_t^{\ell-1},W_t^{\ell-1};q,q;q\big) \\
&\qquad = q^{-ty_3+t^2\sum_{i=1}^3 \binom{y_i}{2}}
\Phi_{n,m;ty}\big(Z_t^{\ell},W_t^{\ell};q,q;q\big).  \notag
\end{align}
Here we have once again used that for $t\geq 2$ the lower bounds
on the sums over $r$ and $s$ may be simplified to $0$.
Using \eqref{Eq_Phi-nmy-11zw} to replace $\Phi_{n,m;3y}(q,q;q)$
by $\Phi_{n,m;3y}(1,1;q,q;q)$ in the summand on the right of
\eqref{Eq_not-compact}, and then applying \eqref{Eq_latticeqq}
with $t=3$ a total of $a-b$ times, first with $\ell=1$, then $\ell=2$
and so on, we obtain
\begin{align}\label{Eq_k-a_a-b}
&\sum_{\substack{\la\subseteq (n^{k-b})\\\mu\subseteq (m^{k-b})}}
g_{\la_{k-b},\mu_{k-b};\tau}(q)
\prod_{i=1}^{k-b} q^{-\chi(i\leq a-b)\la_i}
\mathcal{K}_{\la_{i-1},\mu_{i-1};\la_i,\mu_i}(q,q;q) \\[-2mm]
&\:=\sum_{y\in Q} \Phi_{n,m;3y}\big(Z_3^{a-b},W_3^{a-b};q,q;q\big)
\Delta_{3y}(q,q;q)
\prod_{i=1}^3 q^{3(K-3b)\binom{y_i}{2}-Kiy_i-3\nu_iy_i}, \notag
\end{align}
for integers $b\leq a\leq k$ and $n,m\in\mathbb{N}_0$,
where $\nu:=(a+b+2,b+1,0)$.
To express the summand on the right in terms of the partition $\nu$ we 
have used that $(a-b)y_3-a\sum_{i=1}^3 iy_i = \sum_{i=1}^3 (\nu_i+i)y_i$
for $y\in Q$.
The identity \eqref{Eq_k-a_a-b} is represented by the right-most 
vertex labelled $k-b$ in the diagram on page~\pageref{page_sketch}, and
by abuse of notation will be denoted in the following as $I_a$.
Similarly, the identity corresponding to the
left-most vertex labelled $k-b$ in the diagram on page~\pageref{page_sketch}
will be denoted by $I_{a-1}$, since is follows from $I_a$ by the substitution
$a\mapsto a-1$.
It then follows from Corollary~\ref{Cor_cd1-y} with
\[
(u,v,z,w,y)\mapsto \big(Z_3^{a-b},W_3^{a-b},q,q,3y\big)
\]
that $(I_a-q^{m+1} I_{a-1})/(1-q)$ is given by
\begin{align}
\label{Eq_IaIam1}
&\sum_{\substack{\la\subseteq (n^{k-b})\\\mu\subseteq (m^{k-b})}}
g_{\la_{k-b},\mu_{k-b};\tau}(q) \, \frac{1-q^{m+\la_{a-b}+1}}{1-q}
\prod_{i=1}^{k-b} q^{-\chi(i\leq a-b)\la_i}
\mathcal{K}_{\la_{i-1},\mu_{i-1};\la_i,\mu_i}(q,q;q) \\
&\qquad=\sum_{y\in Q}\Phi_{n,m;3y}\big(Z_3^{a-b},W_3^{a-b};1,1;q,q;q\big)
\prod_{i=1}^3 q^{3(K-3b)\binom{y_i}{2}-Kiy_i-3\nu_i y_i}.  \notag
\end{align}
Since this is a linear combination of $I_a$ and $I_{a-1}$, we should
now restrict the parameters to $b<a\leq k$.
However, since $\la_0:=n$, the identity \eqref{Eq_IaIam1} for $b=a$
simplifies to \eqref{Eq_compact}.
Hence \eqref{Eq_IaIam1}, which corresponds to the central
vertex in the encircled region of the diagram on page~\pageref{page_sketch},
holds for all $b\leq a\leq k$.

In our third and final application of the $\mathrm{A}_2$ Bailey tree,
we carry out the substitutions
\[
(u,v,c,d,z,w,y)\mapsto 
\big(uZ_t^{\ell-1},vW_t^{\ell-1},Z_t^{\ell-1},W_t^{\ell-1},q,q,ty\big)
\]
in Corollary~\ref{Cor_A2-tree-2}.
By \eqref{Eq_zw-scaled} for $z=w=q$ and $a=b=1/q$, this gives
\begin{align*}
&\sum_{r=0}^n  \sum_{s=0}^m  q^{-r-s} \mathcal{K}_{n,m;r,s}(q,q;q) 
\Phi_{r,s;ty}\big(uZ_t^{\ell-1},vW_t^{\ell-1};Z_t^{\ell-1},W_t^{\ell-1};
q,q;q\big) \\[1mm]
&\qquad=q^{t^2\sum_{i=1}^3 \binom{y_i}{2}}
\Phi_{n,m;ty}\big(uZ_t^{\ell},vW_t^{\ell};Z_t^{\ell},W_t^{\ell};q,q;q\big),
\end{align*}
for $n,m\in\mathbb{N}_0$ and $t\geq 2$.
This transformation is applied to \eqref{Eq_IaIam1} 
a total of $b$ times, with $t,u,v$ fixed as
\[
(t,u,v)=\big(3,Z_3^{a-b},W_3^{a-b}\big),
\]
and $\ell=1$ in the first application, $\ell=2$ in the second
application and so on.
As a result,
\begin{align}\label{Eq_KR-finite}
&\sum_{\substack{\la\subseteq (n^k) \\ \mu\subseteq (m^k)}}
g_{\la_k,\mu_k;\tau}(q)\,\frac{1-q^{\la_a+\mu_b+1}}{1-q}
\prod_{i=1}^k q^{-\chi(i\leq a)\la_i-\chi(i\leq b)\mu_i}
\mathcal{K}_{\la_{i-1},\mu_{i-1};\la_i,\mu_i}(q,q;q) \\[-2mm]
&\qquad=\sum_{y\in Q} 
\Phi_{n,m;3y}\big(Z_3^a,W_3^a;Z_3^b,W_3^b;q,q;q\big) 
\prod_{i=1}^3 q^{3K\binom{y_i}{2}-Kiy_i-3\nu_i y_i}, \notag
\end{align}
which is a rational function analogue of \eqref{Eq_KR-ASW},
and corresponds to the vertex labelled $k$ in the diagram on
page~\pageref{page_sketch}.
Although it suffices to prove \eqref{Eq_KR-ASW} for $0\leq b\leq a\leq k$,
we note that the $a,b$-symmetry that is manifest in \eqref{Eq_KR-ASW}
is also satisfied by \eqref{Eq_KR-finite} thanks to \eqref{Eq_symmetry}.
Hence \eqref{Eq_KR-finite} holds for all $0\leq a,b,\leq k$.
Specifically, from \eqref{Eq_symmetry} the $a,b$-symmetry follows 
by making the simultaneous substitutions $(a,b,n,m)\mapsto (b,a,m,n)$
(so that $\nu=(a+b+2,b+1,0)\mapsto (a+b+2,a+1,0)$) and by then changing
the summation indices $(\la,\mu)\mapsto(\mu,\la)$ on the left and
$(y_1,y_2,y_3)\mapsto (-y_3,-y_2,-y_1)$ on the right.

It remains to be shown that \eqref{Eq_KR-ASW} simplifies to the
Kanade--Russell conjecture in the large-$n,m$ limit.
By \eqref{Eq_nmlimit-y} with $(y,x_i)\mapsto (ty,q^{-i})$ (so that
$(z,w)\mapsto(q,q)$), 
\[
\lim_{n,m\to\infty}\Phi_{n,m;ty}\big(Z_t^a,W_t^a;Z_t^b,W_t^b;q,q;q\big) 
=\frac{1}{(q;q)_{\infty}^5}\,
\det_{1\leq i,j\leq 3}\Big(q^{(ty_i-i)(\nu_i-\nu_j)}\Big).
\]
The limit of \eqref{Eq_KR-finite} is thus given by
\begin{align}\label{Eq_S}
&\sum_{\substack{\la,\mu\in\Par \\[1pt] l(\la),\la(\mu)\leq k}} 
\frac{1-q^{\la_a+\mu_b+1}}{1-q}\;
\frac{\prod_{i=1}^k q^{\la_i^2-\la_i\mu_i+\mu_i^2
+\chi(i>a)\la_i+\chi(i>b)\mu_i}}
{\prod_{i=1}^{k-1} (q;q)_{\la_i-\la_{i+1}}(q;q)_{\mu_i-\mu_{i+1}}}\,
g_{\la_k,\mu_k;\tau}(q)\\
&\qquad=\frac{1}{(q;q)_{\infty}^3} \sum_{y\in Q} 
\det_{1\leq i,j\leq 3} 
\Big(q^{3K\binom{y_i}{2}-Kiy_i-(3y_i+j-i)\nu_j}\Big), \notag
\end{align}
where $\nu:=(a+b+2,b+1,0)$, as before.
The remaining task of writing the right-hand side in product-form can
easily be carried out for arbitrary rank, and in the following we consider
\[
A_{\nu;k}(q):=\sum_{y\in Q}\det_{1\leq i,j\leq r}
\Big(q^{rK\binom{y_i}{2}-Kiy_i-(ry_i+j-i)\nu_j}\Big),
\]
for $\nu=(\nu_1,\dots,\nu_r)$.
First we write $A_{\nu;k}(q)$ as a constant term and then
appeal to multilinearity.
Thus
\begin{align*}
A_{\nu;k}(q) &= [z^0]\sum_{y\in\mathbb{Z}^r} 
\det_{1\leq i,j\leq r} \Big(z^{y_i} q^{rk\binom{y_i}{2}-kiy_i
+(ry_i+j-i)(j-\nu_j)}\Big) \\
&= [z^0] \det_{1\leq i,j\leq r} 
\bigg(\sum_{y\in\mathbb{Z}}z^y 
q^{rk\binom{y}{2}-kiy-(ry+j-i)\nu_j}\bigg).
\end{align*}
Interchanging rows and columns (i.e., replacing $(i,j)\mapsto (j,i)$),
negating $y$ and using the fact that we are taking the constant
term with respect to $z$, this leads to
\begin{align*}
A_{\nu;k}(q) & = [z^0] \det_{1\leq i,j\leq r} 
\bigg(\sum_{y\in\mathbb{Z}}
z^y q^{rk\binom{y}{2}+kiy+ry\nu_i+(j-i)(ky+\nu_i)}\bigg) \\
&=\sum_{y\in Q} \det_{1\leq i,j\leq r}\Big(q^{(j-i)(ky_i+\nu_i)}\Big) 
\prod_{i=1}^r q^{rk\binom{y_i}{2}+kiy_i+r\nu_iy_i}.
\end{align*}
Applying the Vandermonde determinant
\[
\det_{1\leq i,j\leq r} \big(x_i^{j-i}\big)=
\prod_{1\leq i<j\leq r}(1-x_i/x_j)
\]
this gives
\[
A_{\nu;k}(q)=\sum_{y\in Q}
\prod_{i=1}^r q^{rk\binom{y_i}{2}+kiy_i+r\nu_iy_i}
\prod_{1\leq i<j\leq r}\big(1-q^{ky_{ij}+\nu_i-\nu_j}\big).
\]
By the $\mathrm{A}_{r-1}^{(1)}$ Macdonald identity~\cite{Macdonald72}
\begin{equation}\label{Eq_Macdonald}
\sum_{y\in Q} \prod_{i=1}^r x_i^{ry_i} q^{r\binom{y_i}{2}+iy_i}
\prod_{1\leq i<j\leq r} \big(1-q^{y_{ij}}x_i/x_j\big)
=(q;q)_{\infty}^{r-1} \prod_{1\leq i<j\leq r} \theta(x_i/x_j;q)
\end{equation}
with $(q,x_i)\mapsto (q^k,q^{\nu_i})$, this results in the product
form
\[
A_{\nu;k}(q)=(q^k;q^k)_{\infty}^{r-1}
\prod_{1\leq i<j\leq r}\theta\big(q^{\nu_i-\nu_j};q^k\big).
\]
Taking $r=3$, $\nu=(a+b+2,b+1,0)$ and $k=K$, yields
\[
\frac{(q^K;q^K)_{\infty}^2}{(q;q)_{\infty}^3}
\prod_{1\leq i<j\leq 3}\theta\big(q^{a+1},q^{b+1},q^{a+b+2};q^K\big)
\]
for the right-hand side of \eqref{Eq_S}.
\end{proof}

\section{Below-the-line identities}\label{Sec_Belowtheline}

As in Conjecture~\ref{Con_KR-ASW}, fix the modulus $K$ as $K=3k+\tau+3$
for $k$ a nonnegative integer and $\tau\in\{-1,0,1\}$.
In the introduction immediately preceding the conjecture,
we remarked that there should be an ASW-type identity for all nonnegative
integers $a,b$ such that $a+b\leq K-3$, with product side given
by\footnote{For $\tau=-1$ this rules out $(a,b)=(k,k)$, which as
discussed in the introduction gives the same product as $(a,b)=(k,k-1)$
albeit a slightly different multisum according to \eqref{Eq_KR-ASW}.}
\[
\frac{(q^K;q^K)_{\infty}^2}{(q;q)_{\infty}^3}\,
\theta\big(q^{a+1},q^{b+1},q^{a+b+2};q^K\big),
\]
Without loss of generality assuming that
\begin{equation}\label{Eq_ab-order}
0\leq b\leq a\leq K-a-b-3,
\end{equation}
this corresponds to
\[
\binom{k+2}{2}-\delta_{\tau,-1}+\Big\lfloor \frac{(k+\tau)^2}{4}\Big\rfloor
\]
distinct ASW-type identities.
Hence in the Kanade--Russell conjecture roughly one third of all cases
is missing, counted by the above floor function.
In their paper, Kanade and Russell adopt a certain diagrammatic arrangement
for the triples $(K-a-b-3,a,b)$ with fixed $K$, leading them to refer to
the missing identities as the `below-the-line' cases.
Equivalently, this corresponds to \eqref{Eq_ab-order} with $a>k$
(and thus $b\leq k+\tau-2$).
If $k=1$ this forces $\tau=1$, in which case there is the
single below-the-line solution: $(a,b)=(2,0)$.
By solving the modulus-$7$ Corteel--Welsh equations \cite{CW19},
Kanade and Russell found the missing multisum, resulting in
\[
\sum_{\la_1,\mu_1=0}^{\infty} \,
\frac{1-q^{2\la_1-\mu_1}}{1-q}\,
\frac{q^{\la_1^2-\la_1\mu_1+\mu_1^2-\la_1+\mu_1}}
{(q;q)_{\la_1}(q;q)_{\mu_1}(q^2;q)_{\la_1+\mu_1}}
=\frac{(q^7;q^7)_{\infty}}{(q;q)_{\infty}^3}\,
\theta\big(q,q^3,q^3;q^7\big).
\]
In general, however, no explicit such multisum-forms for below-the-line
values of $a$ and $b$ are known.
The exception is $\tau=0$, in which case Kanade and Russell observed that
if
\[
\Theta_{a,b;k}(q):=\theta\big(q^{a+1},q^{b+1},q^{a+b+2};q^{3k+3}\big),
\]
then Weierstrass' three-term relation \cite[page 61]{GR04} implies,
\[
\Theta_{a,b;k}(q)=\Theta_{2k-a,a+b-k;k}-q^{b+1}\Theta_{2k-a-b-1,a-k-1;k}(q).
\]
Importantly, for $\tau=0$ and fixed $k\geq 2$, the below-the-line 
values of $(a,b)$ satisfy $k<a\leq \lfloor 3k/2\rfloor$ and
$0\leq b\leq 3k-2a$.
Assuming such $a,b$ and defining $(a',b'):=(2k-a,a+b-k)$ and
$(a'',b''):=(2k-a-b-1,a-k-1)$, it follows that
$0<b'\leq a'\leq \lceil k/2 \rceil$ and $0\leq b''\leq a'\leq k-2$.
This implies the following theorem covering all of the
below-the-line cases.
For integers $a,b,k$ such that $0\leq a,b\leq k$, let
\[
\mathcal{F}_{a,b;k}(q):=
\sum_{\substack{\la_1\geq\cdots\geq\la_k\geq 0 \\[1pt] 
\mu_1\geq\cdots\geq\mu_k\geq 0}}\, 
\frac{1-q^{\la_a+\mu_b+1}}{1-q}\,
\frac{q^{\sum_{i=1}^k(\la_i^2-\la_i\mu_i+\mu_i^2)+
\sum_{i=a+1}^k\la_i+\sum_{i=b+1}^k\mu_i}}
{\prod_{i=1}^{k-1} (q;q)_{\la_i-\la_{i+1}}(q;q)_{\mu_i-\mu_{i+1}}} \,
g_{\la_k,\mu_k;0}(q),
\]
where $q^{\la_0}=q^{\mu_0}:=0$.

\begin{theorem}\label{Thm_Belowtheline}
Let $a,b,k$ be integers such that $2\leq k<a\leq \lfloor 3k/2\rfloor$ 
and $0\leq b\leq 3k-2a$.
Then
\[
\mathcal{F}_{2k-a,a+b-k;k}(q)-
q^{b+1} \mathcal{F}_{2k-a-b-1,a-k-1;k}(q)
=\frac{(q^K;q^K)_{\infty}^2}{(q;q)_{\infty}^3}\,
\theta\big(q^{a+1},q^{b+1},q^{a+b+2};q^K\big),
\]
where $K:=3k+3$.
\end{theorem}

This was first stated in \cite{KR23} as a conditional result,
depending on the validity of Conjecture~\ref{Con_KR-ASW}. 
By Theorem~\ref{Thm_KRistrue} the result is now unconditional.
It remains an open problem to express the left-hand side in manifestly
positive form.

\section{Character identities for the $\mathcal{W}_3(3,K)$
vertex operator algebra}\label{Sec_AG}

As explained in full detail in \cite[Section 4]{W23}, for $\tau\neq 0$
(so that $3\nmid K$) the $q$-series in \eqref{Eq_KR-ASW} multiplied
by $q^{h-c/24}(q;q)_{\infty}$ are characters $\chi_{a,b}^K(q)$ of the
$\mathcal{W}_3(3,K)$ vertex operator algebra \cite{FL88,Zamolodchikov85}
of central charge 
\begin{equation}\label{Eq_c}
c=-\frac{2(K-4)(4K-9)}{K}
\end{equation}
and conformal weight
\[
h_{a,b}=\frac{a^2+ab+b^2-(K-3)(a+b)}{K}.
\]
That is,
\[
\chi_{a,b}^K(q)=q^{h_{a,b}-c/24}\,
\frac{(q^K;q^K)_{\infty}^2}{(q;q)_{\infty}^2}\,
\theta\big(q^{a+1},q^{b+1},q^{a+b+2};q^K\big),
\]
where $a,b,K$ are nonnegative integers such that $K\geq 5$, $3\nmid K$ and
$a+b\leq K-3$.
To obtain a multisum expression for these characters without an overall
factor $(q;q)_{\infty}$, we need to carry out a suitable rewriting of
the multisum in \eqref{Eq_KR-ASW}.
This is possible by means of the next lemma, which is a limiting case
of~\cite[Lemma 7.2]{W23}.

\begin{lemma}\label{Lem_F-alt}
For $k$ a positive integer, $m$ a nonnegative integer and
$u=(u_1,\dots,u_{k+1})\in\mathbb{Z}^{k+1}$ define
\[
\mathcal{F}_u(q):=\sum_{\mu_1\geq\cdots\geq\mu_k\geq 0}
\frac{q^{\sum_{i=1}^k \mu_i(\mu_i+u_i)}}{(q)_{\mu_k+u_{k+1}} 
\prod_{i=1}^k (q;q)_{\mu_i-\mu_{i+1}}},
\]
where $\mu_{k+1}:=0$.
If
\[
u_1\leq u_2\leq\cdots\leq u_{k+1},
\]
then
\begin{equation}\label{Eq_F-alt}
\mathcal{F}_u(q)=\frac{1}{(q;q)_{\infty}}
\sum_{\mu_1,\dots,\mu_k\geq 0}
q^{\sum_{i=1}^k \mu_i(\mu_i+u_i)}
\prod_{i=1}^k \qbin{\mu_{i+1}+u_{i+1}-u_i}{\mu_i},
\end{equation}
where, again, $\mu_{k+1}:=0$.
\end{lemma}

The left-hand side of \eqref{Eq_KR-ASW} for $\tau\neq 0$ may be
expressed in terms of $\mathcal{F}_u$ as
\[
\sum_{\la_1\geq\cdots\geq\la_k\geq 0}
\, \frac{q^{\sum_{i=1}^k\la_i^2+\sum_{i=a+1}^k\la_i}}
{\prod_{i=1}^k (q;q)_{\la_i-\la_{i+1}}}\,\mathcal{F}_u(q) 
-\chi(ab>0)\!\sum_{\la_1\geq\cdots\geq\la_k\geq 0} 
\,\frac{q^{1+\sum_{i=1}^k\la_i^2+\sum_{i=a}^k\la_i}}
{\prod_{i=1}^k (q;q)_{\la_i-\la_{i+1}}}\,\mathcal{F}_v(q),
\]
where $\la_{k+1}:=0$, 
\[
u_i=\begin{cases}
\chi(i>b)-\la_i & \text{for $1\leq i<k$}, \\
\chi(k>b)-\tau\la_k & \text{for $i=k$}, \\
1+\la_k & \text{for $i=k+1$}, \end{cases}
\quad\text{and}\quad
v_i=\begin{cases}
\chi(i\geq b)-\sigma_i\la_i & \text{for $1\leq i<k$}, \\
1-\tau\la_k & \text{for $i=k$}, \\
1+\la_k & \text{for $i=k+1$}.
\end{cases}
\]
Since for $\la_1\geq\la_2\geq\cdots\geq\la_k$ the inequalities
$u_i\leq u_{i+1}$ and $v_i\leq v_{i+1}$ hold for all $1\leq i\leq k$,
we may use the alternative expressions for $\mathcal{F}_u(q)$ and
$\mathcal{F}_v(q)$ provided by \eqref{Eq_F-alt}.
First, for $\tau=1$, this yields our next theorem,
where $\tilde{\chi}_{a,b}^K(q):=q^{c/24-h_{a,b}} \chi_{a,b}^K(q)$.

\begin{theorem}[$\mathrm{A}_2^{(1)}$ Andrews--Gordon identities, I]
\label{Thm_3k4}
Let $K=3k+4$ for $k\geq 1$.
Then
\begin{align*}
&\tilde{\chi}_{a,b}^K(q) \\
&=\sum_{\substack{\la_1,\dots,\la_k\geq 0\\[1pt]
\mu_1,\dots,\mu_k\geq 0}}
\, \frac{q^{\sum_{i=a+1}^k\la_i+\sum_{i=b+1}^k\mu_i}}
{(q;q)_{\la_1}} 
\prod_{i=1}^k q^{\la_i^2-\la_i\mu_i+\mu_i^2}
\qbin{\la_i}{\la_{i+1}}
\qbin{\la_i-\la_{i+1}+\mu_{i+1}+\delta_{b,i}}{\mu_i} \\[1mm]
&\quad-
\sum_{\substack{\la_1,\dots,\la_k\geq 0\\[1pt]
\mu_1,\dots,\mu_k\geq 0}}
\, \frac{q^{1+\sum_{i=a}^k\la_i+\sum_{i=b}^k\mu_i}}
{(q;q)_{\la_1}} 
\prod_{i=1}^k q^{\la_i^2-\la_i\mu_i+\mu_i^2}
\qbin{\la_i}{\la_{i+1}}
\qbin{\la_i-\la_{i+1}+\mu_{i+1}+\delta_{b-1,i}}{\mu_i}
\end{align*}
for all $0\leq a,b\leq k$, and
\[
\tilde{\chi}_{k,k}^K(q)
=\sum_{\substack{\la_1,\dots,\la_k\geq 0\\[1pt]
\mu_1,\dots,\mu_k\geq 0}}
\, \frac{1}{(q;q)_{\la_1}} 
\prod_{i=1}^k q^{\la_i^2-\la_i\mu_i+\mu_i^2}
\qbin{\la_i}{\la_{i+1}}
\qbin{\la_i-\la_{i+1}+\mu_{i+1}}{\mu_i},
\]
where $q^{\la_0}=q^{\mu_0}=\la_{k+1}:=0$ and $\mu_{k+1}:=\la_k$.
\end{theorem}

The second, simpler expression for $\tilde{\chi}_{k,k}^K(q)$
follows by either noting that for $a=b=k$, the left-hand side
of \eqref{Eq_KR-ASW} for $\tau\neq 0$ may alternatively be
recognised as
\[
\sum_{\la_1\geq\cdots\geq\la_k\geq 0}
\, \frac{\mathcal{F}_w(q)}{\prod_{i=1}^k (q;q)_{\la_i-\la_{i+1}}},
\]
where $\la_{k+1}:=0$ and
\[
w_i=\begin{cases}
-\la_i & \text{for $1\leq i<k$}, \\
-\tau\la_k & \text{for $i=k$}, \\
\la_k & \text{for $k+1$},
\end{cases}
\]
or by substituting $a=b=k$ in the expression for 
$\tilde{\chi}_{a,b}^K(q)$, replacing $\mu_k\mapsto\mu_k-1$ in the
second multisum and then combining the two multisums using the
standard recursion for the $q$-binomial coefficient.
The $b=0$ case of Theorem~\ref{Thm_3k4} proves \cite[Conjecture~2.8]{W23} 
and the $a=0$ case proves Equation (2.7) of that same paper.
Since $\tilde{\chi}_{a,b}^K(q)=\tilde{\chi}_{b,a}^K(q)$ while the
right-hand side of the first character formula does not have 
$a,b$-symmetry, there are two distinct expressions for each 
$\mathcal{W}_3(3,K)$ character $\chi_{a,b}^K(q)$ such that $a\neq b$.
The reason for viewing the above as analogues of the Andrews--Gordon
identities \eqref{Eq_AGB} is that in much the same way the latter are 
known to be identities for characters of the Virasoro algebra 
$\Vir(2,K)=\mathcal{W}_2(2,K)$.

For $\tau=-1$ we obtain the following companion to the previous theorem.

\begin{theorem}[$\mathrm{A}_2^{(1)}$ Andrews--Gordon identities, II]
\label{Thm_3k2}
Let $K=3k+2$ for $k\geq 1$ and $0\leq a\leq k$, $0\leq b<k$.
Then
\begin{align*}
&\tilde{\chi}_{a,b}^K(q) \\
&\;=\sum_{\substack{\la_1,\dots,\la_k\geq 0\\[1pt]
\mu_1,\dots,\mu_{k-1}\geq 0}}\!\!
\, \frac{q^{\la_k^2+\sum_{i=a+1}^k\la_i+\sum_{i=b+1}^{k-1}\mu_i}}
{(q;q)_{\la_1}} 
\prod_{i=1}^{k-1} q^{\la_i^2-\la_i\mu_i+\mu_i^2}
\qbin{\la_i}{\la_{i+1}}
\qbin{\la_i-\la_{i+1}+\mu_{i+1}+\delta_{b,i}}{\mu_i} \\[1mm]
&\;\quad- \sum_{\substack{\la_1,\dots,\la_k\geq 0\\[1pt]
\mu_1,\dots,\mu_{k-1}\geq 0}}\!\!
\, \frac{q^{1+\la_k^2+\sum_{i=a}^k\la_i+\sum_{i=b}^{k-1}\mu_i}}
{(q;q)_{\la_1}} 
\prod_{i=1}^{k-1} q^{\la_i^2-\la_i\mu_i+\mu_i^2}
\qbin{\la_i}{\la_{i+1}}
\qbin{\la_i-\la_{i+1}+\mu_{i+1}+\delta_{b-1,i}}{\mu_i}
\end{align*}
for $0\leq a\leq k$, $0\leq b<k$, and
\[
\tilde{\chi}_{k,k}^K(q)
=\sum_{\substack{\la_1,\dots,\la_k\geq 0\\[1pt]
\mu_1,\dots,\mu_{k-1}\geq 0}}\!\!
\, \frac{q^{\la_k^2}}{(q;q)_{\la_1}} 
\prod_{i=1}^{k-1} q^{\la_i^2-\la_i\mu_i+\mu_i^2}
\qbin{\la_i}{\la_{i+1}}\qbin{\la_i-\la_{i+1}+\mu_{i+1}}{\mu_i},
\]
where $q^{\la_0}=q^{\mu_0}=\la_{k+1}:=0$ and $\mu_k:=2\la_k$.
\end{theorem}

This time the $b=0$ case proves \cite[Conjecture~2.1]{W23} 
and the $a=0$ case proves Equation (2.2) of~\cite{W23}.

For a number of special values of $k$, alternative multisum expressions
to those of Theorems~\ref{Thm_3k4} and \ref{Thm_3k2} are known.
In \cite{CDU22}, Corteel, Dousse and Uncu solved the Corteel--Welsh system
of equations for the two-variable generating function of three-row cylindric
partitions with profile $(5-a-b,a,b)$, resulting in quadruple-sum
expressions for the characters $\tilde{\chi}_{a,b}^8(q)$.
For example (see \cite[Theorem 1.6]{CDU22}), 
\[
\tilde{\chi}_{2,1}^8(q)=
\sum_{n_1,n_2,n_3,n_4=0}^{\infty}
\frac{q^{n_1^2+n_2^2+n_3^2+n_4^2-n_1n_2+n_2n_4}}
{(q;q)_{n_1}}\,\qbin{n_1}{n_2}\qbin{n_1}{n_4}\qbin{n_2}{n_3}.
\]
In \cite[Theorems 2.3 \& 2.4]{FFW08}, Feigin, Foda and Welsh obtained
an Andrews--Gordon-type theorem for a linear combination of characters
of $\Vir(3,3k+2)$ of central charge $c=-3k(6k-5)/(3k+2)$.
For $k=4$ this yields $c=-114/7$, which coincides with the central
charge of $\mathcal{W}_3(3,7)$. 
In this case, four of the six linear combinations considered in
\cite{FFW08} correspond to actual $\mathcal{W}_3(3,7)$ characters.
Three are also covered in Theorem~\ref{Thm_3k4} while the fourth
is below-the-line in the sense of Kanade and Russell.
For example, the character expression for $\tilde{\chi}_{1,1}^7(q)$
arising from $\Vir(3,14)$ is~\cite[Equation (20c)]{FFW08}
\[
\tilde{\chi}_{1,1}^7(q)=\sum_{n_1,n_2,n_3,n_3=0}^{\infty}
\frac{q^{n_1^2+n_2^2+n_3^2+n_4^2+(n_1+n_2+n_3)n_4}}
{(q;q)_{n_1}(q;q)_{n_4}} \qbin{n_1}{n_2} \qbin{n_2}{n_3}.
\]
After the substitutions
\[
(n_1,n_2,n_3,n_4)\mapsto (n_1+n_3+n_4,n_3+n_4,n_4,n_2)
\]
this takes the form
\begin{equation}\label{Eq_FFW}
\tilde{\chi}_{1,1}^7(q)=\sum_{n_1,n_2,n_3,n_3=0}^{\infty}
\frac{q^{\sum_{i,j=1}^4 n_i A_{ij}n_j}}
{(q;q)_{n_1}(q;q)_{n_2}(q;q)_{n_3}(q;q)_{n_4}},
\end{equation}
where
\[
(A_{ij})=\frac{1}{2}
\begin{pmatrix}
2 & 1 & 2 & 2 \\
1 & 2 & 2 & 3 \\
2 & 2 & 4 & 4 \\
2 & 3 & 4 & 6 
\end{pmatrix}.
\]

At a workshop on cylindric partitions held at RICAM in 2022,
Shunsuke Tsuchioka raised the question if all the
$\mathrm{A}_2^{(1)}$ Andrews--Gordon identities
admit alternative sum-sides of the form \eqref{Eq_FFW}.
Such expressions would be closer to the 
$\mathrm{A}_1^{(1)}$ Andrews--Bressoud--Gordon identities, where the
variable change $n_i\mapsto n_i+\dots+n_k$ for all $1\leq i\leq k$
leads to the multisum
\[
\sum_{n_1,\dots,n_k\geq 0}
\frac{q^{\sum_{i,j=1}^k n_i A_{ij} n_j+\sum_{i=1}^k (A_{ki}-A_{ai})n_i}}
{(q;q)_{n_1}\dots(q;q)_{n_{k-1}}(q^{2-\tau};q^{2-\tau})_{n_k}},
\]
where $(A_{ij})_{i,j=1}^k=(\min\{i,j\})_{i,j=1}^k$ is the
Cartan-type matrix of the tadpole graph on $k$ vertices.
As further evidence that such a rewriting might exist for all moduli, 
he made a conjecture for modulus $8$, complementing his own proven 
modulus-$6$ identities \cite{Tsuchioka22}, such as
\begin{align*}
&\sum_{n_1^{(1)},n_2^{(1)},n_1^{(2)},n_2^{(2)}=0}^{\infty} 
\frac{q^{\sum_{i,j,a,b=1}^2 A_{ia,jb} n_i^{(a)}n_j^{(b)}}}
{\prod_{i,a=1}^2(q;q)_{n_i^{(a)}}} \\
&\qquad\qquad=\sum_{n,m,k,l=0}^{\infty} \frac{q^{n^2+3kn+3k^2}}
{(q;q)_{n}(q^3;q^3)_k}\qbin{n}{m}\qbin{k}{l}_{q^3}
=(-q;q)_{\infty}^2(q^2,q^4;q^6)_{\infty},
\end{align*}
where $A=\frac{1}{2}B\otimes C$ (i.e.,
$A_{ia,jb}=\frac{1}{2}B_{ij}C_{ab}$) with matrices $B$ and $C$
given by $B=\begin{psmallmatrix}2&3\\3&6\end{psmallmatrix}$ and 
$C=\begin{psmallmatrix}1&1\\1&1\end{psmallmatrix}$.
From the structure of the summands in Theorems~\ref{Thm_3k4} and
\ref{Thm_3k2} it follows relatively straightforwardly that a rewriting
of the form \eqref{Eq_FFW} can be carried out for the moduli $7$ and $8$.
For larger moduli, however, this simple method fails due to the form of
the summands.
By iterating the Durfee rectangle identity
\cite[Equation (3.3.10)]{Andrews98}
\begin{equation}\label{Eq_Durfee}
\qbin{n+m}{n+a} = 
\sum_{k=0}^n q^{k(k+a)} \qbin{n}{k}\qbin{m}{k+a} 
\end{equation}
for $n,m\in\mathbb{N}_0$ and $a\in\mathbb{Z}$, it follows that the
$q$-binomial coefficient admits the telescopic expansion
\begin{equation}\label{Eq_telescope}
\qbin{k_0+m}{k_0+a} = 
\sum_{k_0\geq k_1\geq k_2\geq\cdots\geq k_r\geq 0} 
\qbin{k_0+m-\sum_{i=0}^{r-1}k_i}{k_r+a}
\prod_{i=1}^r
q^{k_i(k_i+a)}\qbin{k_{i-1}}{k_i},
\end{equation}
for arbitrary nonnegative integer $r$ and integers $a,k_0,m$ such that
$k_0,m\geq 0$ and, if $a=-k_0$, then $m\geq (r-1)k_0$.
If we take $r=2$ and once more apply \eqref{Eq_Durfee} with
$(n,m,a)\mapsto (k_0-k_1,m-k_0,k_1+k_2+a-k_0)$, this implies
\begin{align*}
&\frac{1}{(q;q)_{m-k_0}(q;q)_{k_0}} \qbin{k_0+m}{k_0+a}:= \\
&\sum_{k_1,k_2,k_3} 
\frac{q^{\sum_{i=1}^3 k_i(k_i+a)+(k_1+k_2-k_0)k_3}}
{(q;q)_{k_1-k_2}(q;q)_{k_2}(q;q)_{k_3}
(q;q)_{k_0-k_1-k_3}(q;q)_{a+k_1+k_2+k_3-k_0}(q;q)_{m-a-k_1-k_2-k_3}},
\end{align*}
for all integers $a,k_0,m$ such that $0\leq k_0\leq m$.
Since
\[
\tilde{\chi}_{1,1}^7(q)=
\sum_{\la_1,\mu_1}
\frac{q^{\la_1^2-\la_1\mu_1+\mu_1^2}}
{(q;q)_{\la_1}}\qbin{2\la_1}{\mu_1}
\]
and
\[
\tilde{\chi}_{2,2}^8(q)=
\sum_{\la_1,\la_2,\mu_1}
\frac{q^{\la_1^2-\la_1\mu_1+\mu_1^2+\la_2^2}}
{(q;q)_{\la_1-\la_2}(q;q)_{\la_2}}\qbin{\la_1+\la_2}{\mu_1},
\]
we can use the above expansion with $(m,k_0,a)$ given by
$(\la_1,\la_1,\la_1-\mu_1)$ and $(\la_1,\la_2,\la_1-\mu_1)$
respectively. 
In the first case this fixes $k_3$ as $k_3=\mu_1-k_1-k_2$.
Finally, making the substitutions
\[
(\la_1,\mu_1,k_1,k_2)\mapsto (n_1+n_2+n_3+n_4,n_2+n_3+2n_4, 
n_3+n_4,n_4) 
\]
and
\begin{align*}
&(\la_1,\mu_1,\la_2,k_1,k_2,k_3) \\
&\quad \mapsto
(n_1+n_2+n_3+n_4+n_5+n_6,n_2+n_4+n_5+2n_6,n_3+n_4+n_5+n_6,
n_5+n_6,n_6,n_4)
\end{align*}
yields, respectively, \eqref{Eq_FFW} and
\begin{equation}\label{Eq_ST-conjecture}
\tilde{\chi}_{2,2}^8(q)=
\sum_{n_1,\dots,n_6=0}^{\infty}
\frac{q^{\sum_{i,j=1}^6 n_i A_{ij} n_i}}
{(q;q)_{n_1}\cdots(q;q)_{n_6}},
\end{equation}
for
\[
(A_{ij})=\frac{1}{2}\begin{pmatrix}
2 & 1 & 2 & 2 & 2 & 2 \\
1 & 2 & 1 & 2 & 2 & 3 \\
2 & 1 & 4 & 3 & 4 & 4 \\
2 & 2 & 3 & 4 & 4 & 5 \\
2 & 2 & 4 & 4 & 6 & 6 \\
2 & 3 & 4 & 5 & 6 & 8
\end{pmatrix}.
\]
This last result is exactly one of formulas for
$\tilde{\chi}_{a,b}^8$ conjectured by Tsuchioka~\cite{Tsuchioka22b}.


\section{Character formulas for principal subspaces of
$\mathrm{A}_2^{(1)}$}\label{Sec_PS}

Let $\mathfrak{g}=\mathfrak{sl}_r=\mathrm{A}_{r-1}$ and
$\hat{\mathfrak{g}}=\widehat{\mathfrak{sl}}_r=\mathrm{A}_{r-1}^{(1)}$ 
its untwisted affinisation, i.e.,
\[ 
\hat{\mathfrak{g}}\cong \mathfrak{g}\otimes \mathbb{C}[t,t^{-1}]
\oplus \mathbb{C}c\oplus\mathbb{C}d,
\]
where $c$ is the canonical central element and $d$ a derivation, 
acting on the loop algebra $\mathfrak{g}\otimes \mathbb{C}[t,t^{-1}]$ as
$t\frac{\textup{d}}{\textup{d} t}$, see \cite[Chapter 7]{Kac90} for details.
Fix $I:=\{0,1,\dots,r-1\}$ and let $\CSAa$ be the Cartan subalgebra of 
$\hat{\mathfrak{g}}$ 
with basis $\{\alpha_0^{\vee},\dots,\alpha_{r-1}^{\vee},d\}$,
where the $\alpha_i^{\vee}$ ($i\in I$) are the simple coroots
(so that $c=\sum_{i\in I} \alpha_i^{\vee}$).
Let $A=(a_{ij})_{i,j=0}^{r-1}$ be the (generalised) Cartan matrix
of $\hat{\mathfrak{g}}$, 
and fix the non-degenerate symmetric bilinear form $\bil{\cdot}{\cdot}$
on $\CSAa$ by setting $\bil{\alpha_i^{\vee}}{\alpha_j^{\vee}}=a_{ij}$,
$\bil{d}{d}=0$, $\bil{\alpha_0^{\vee}}{d}=1$ and
$\bil{\alpha_i^{\vee}}{d}=0$ otherwise. 
Further let $\dCSAa$ be the dual of the Cartan subalgebra with basis
$\{\alpha_0,\dots,\alpha_{r-1},\Lambda_0\}$, 
where the $\alpha_i$ ($i\in I$) are the simple roots and $\La_0$ is
the $0$th fundamental weight.
Denote the standard pairing between the Cartan subalgebra and its dual by
$\pairing{\cdot}{\cdot}$, so that $\pairing{\alpha_i}{\alpha_j^{\vee}}=
\bil{\alpha_i^{\vee}}{\alpha_j^{\vee}}=a_{ij}$ and
$\pairing{\La_0}{a_i^{\vee}}=0$.
The additional fundamental weights $\La_1,\dots,\La_{r-1}\in\dCSAa$ are
fixed as $\pairing{\La_i}{\alpha_j^{\vee}}=\delta_{ij}$ for all $i,j\in I$
and $\pairing{\La_i}{d}=0$ for all $i\in I$.
The level of $\la\in\dCSAa$ is defined by $\lev(\la):=\pairing{\la}{c}$.
Hence $\lev(\La_i)=1$ for all $i\in I$ and if
$\delta:=\sum_{i\in I} \alpha_i$ is the null root, then
$\lev(\delta)=\sum_{i,j\in I} a_{ij}=0$.
Finally, let
\[
P:=\big\{\la\in\mathfrak{h}^{\ast}: 
\pairing{\la}{\alpha_i^{\vee}}\in\mathbb{Z} \text{ for all } i\in I\big\}
\]
be the weight lattice of $\hat{\mathfrak{g}}$, and
$P_{+}\subset P$ and $P_{+}^{\ell}\subset P_{+}$ the set of dominant 
integral weights and level-$\ell$ dominant integral weights respectively:
\begin{align*}
P_{+}&=\big\{\la\in\mathfrak{h}^{\ast}: 
\pairing{\uplambda}{\alpha_i^{\vee}}\in\mathbb{N}_0
\text{ for all } i\in I\big\}=
\mathbb{N}_0\La_0+\dots+\mathbb{N}_0\La_{r-1}+\mathbb{C}\delta, \\
P_{+}^{\ell}&=\big\{\la\in P_{+}: \lev(\uplambda)=\ell \big\}.
\end{align*}

A much studied class of representations of $\mathrm{A}_{r-1}^{(1)}$
are the standard or integrable highest weight modules.
There is a unique such module, $L_{\uplambda}$, for each 
$\uplambda\in P_{+} \mod \mathbb{C}\delta$.
If $v_{\uplambda}$ denotes the highest weight vector of $L_{\uplambda}$,
then $\CSAa$ acts diagonally on $v_{\la}$ and
$c v_{\uplambda}=\lev(\uplambda) v_{\uplambda}$.
The principal subspace $W_{\uplambda}\subset L_{\uplambda}$ is defined as
\cite{AKS06,FS93,SF94}\footnote{There are two related but distinct definitions
used in the literature, and here we follow the less standard~\cite{AKS06}.
In the original paper \cite{SF94}, 
$U\big(\mathfrak{n}_{+}\otimes\mathbb{C}[t,t^{-1}]\big) v_{\uplambda}$ is 
used instead.}
\[
W_{\la}:=U\big(\mathfrak{n}_{-}\otimes\mathbb{C}[t,t^{-1}]\big) v_{\uplambda}
=U\big(\mathfrak{n}_{-}\otimes\mathbb{C}[t^{-1}]\big) v_{\uplambda},
\]
where $\mathfrak{n}_{-}\oplus\CSA\oplus\mathfrak{n}_{+}$
is the triangular or Cartan decomposition of $\mathfrak{g}$
and $U(\cdot)$ denotes the universal enveloping algebra.
Let $f_1,\dots,f_{r-1}\in\mathfrak{g}$ denote the standard generators
of $\mathfrak{n}_{-}$.
Then the character of the principal subspace $W_{\uplambda}$ is defined as
\[
\ch W_{\uplambda}:=\sum_{n,d_1,\dots,d_{r-1}\geq 0}
\dim \big(W_{\uplambda;n;d_1,\dots,d_{r-1}}\big) 
\eup^{\uplambda-\delta n-\sum_{i=1}^{r-1}d_i \alpha_i},
\]
where $W_{\uplambda;n;d_1,\dots,d_{r-1}}\subset W_{\uplambda}$ is the
subspace generated by those elements in
$U(\mathfrak{n}_{-}\otimes\mathbb{C}[t^{-1}])$ of degree $d_i$ in $f_i$
and degree $n$ in $t^{-1}$.
For convenience we in the following use the normalised character
\[
\ch W'_{\uplambda}:=\eup^{-\uplambda} \ch W_{\uplambda}.
\]
Ardonne, Kedem and Stone \cite[Equation (6.9)]{AKS06}\footnote{For
$\uplambda=(k-a)\La_0+a\La_i$, $i\in I$, the dependence on the generalised
Kostka polynomials trivialises and the result is essentially due to 
Georgiev \cite{Georgiev96}, with the caveat that he used the definition
of principal subspace from \cite{SF94}.}
found an explicit expression for $\ch W_{\uplambda}$ in terms of
generalised Kostka polynomials~\cite{KS02,SW99}.
Restricting considerations to $r=3$, and assuming the parametrisation
\begin{equation}\label{Eq_lakab}
\uplambda=(k-a-b)\La_0+a\La_1+b\La_2\in P_{+}^k,
\end{equation}
the Ardonne, Kedem and Stone character formula simplifies to
\cite[Equations (6.9), (6.15) \& (6.16)]{AKS06}
\begin{equation}\label{Eq_fermionic}
\ch W'_{\uplambda}:=
\sum_{\substack{\la,\mu\in\Par \\[1pt] l(\la),l(\mu)\leq k}}
\bigg(\big(1-zwq^{\la_a+\mu_b-1}\big) 
\prod_{i=1}^k \frac{z^{\la_i} w^{\mu_i} 
q^{\la_i^2-\la_i\mu_i+\mu_i^2-\chi(i\leq a)\la_i-\chi(i\leq b)\mu_i}}
{(q;q)_{\la_i-\la_{i+1}}(q;q)_{\mu_i-\mu_{i+1}}}\bigg),
\end{equation}
where $q^{\la_0}=q^{\mu_0}:=0$ and $q:=\eup^{-\delta}$,
$z:=\eup^{-\alpha_1}$, $w:=\eup^{-\alpha_2}$.
The restrictions $l(\la),l(\mu)\leq k$ in the sum imply that 
$\la_{k+1}=\mu_{k+1}=0$.
By mild abuse of notation we in the remainder of this section
use $\ch W'_{\uplambda}$ to mean the right-hand
side of \eqref{Eq_fermionic} for all $0\leq a,b\leq k$, despite the
fact that for $a+b>k$ the weight $\uplambda$ is not dominant. 

In the vacuum case, corresponding to $a=b=0$, Feigin et al.\
\cite[Corollary 7.8]{FFJMM09} obtained an alternative `bosonic'
expression for $W_{\uplambda}$.
This is the $a=b=0$ case of our next theorem.

\begin{theorem}\label{Thm_bosonic}
For $a,b,k$ integers such that $0\leq a,b\leq k$, let the weight
$\uplambda$ and partition $\nu$ be given by \eqref{Eq_lakab} and 
$\nu=(a+b+2,b+1,0)$ respectively.
Then
\begin{align}\label{Eq_bosonic}
\ch W'_{\uplambda}&=\prod_{1\leq i<j\leq 3}\frac{1}{(x_i/x_j;q)_{\infty}} \\
& \quad \times \sum_{y\in Q_{+}}
\det_{1\leq i,j\leq 3}\big((x_iq^{y_i})^{\nu_i-\nu_j}\big)
\prod_{i=1}^3 
\frac{x_i^{(k+2)y_i} q^{(k+2)\binom{y_i}{2}-\nu_iy_i}
(x_i/x_3;q)_{y_i}}{(qx_i/x_1;q)_{y_i}},\notag
\end{align}
where $x_1/x_2:=\eup^{-\alpha_1}$ and $x_2/x_3:=\eup^{-\alpha_2}$.
\end{theorem}

By \eqref{Eq_fermionic} this is Theorem~\ref{Thm_Principal-subspace} of the
introduction.

\begin{proof}[Proof of Theorem~\ref{Thm_bosonic}]
The main steps of the proof are the same as in the proof of the
Kanade--Russell conjecture in Section~\ref{Sec_KR}.
Key difference is the root identity to which the $\mathrm{A}_2$
Bailey tree is applied, which essentially is the $\mathrm{A}_2$
unit Bailey pair \eqref{Eq_UBP}.
Also, since the right-hand side of \eqref{Eq_bosonic} does not admit a
product form, this time round there is no need for the
$\mathrm{A}_2^{(1)}$ Macdonald identity in the final stages of the proof.

For $y=(y_1,y_2,y_3)\in Q$, let
\begin{equation}\label{Eq_Psi}
\Psi_y(z,w;q):=q^{-y_{13}} (zq;q)_{y_{12}}(wq;q)_{y_{23}}(zwq;q)_{y_{13}}
\Phi_{y_1,y_1+y_2}\big(zq^{y_{12}},wq^{y_{23}};q^{-1}\big).
\end{equation}
Point of departure for our proof is \eqref{Eq_inversion} for $N=M=0$.
Identifying $(r,s)=(y_1,y_1+y_2)$ and using \eqref{Eq_zwprod},
this may also be written as
\begin{equation}\label{Eq_unit-BP}
\delta_{n,0}\delta_{m,0}=
\sum_{y\in Q_{+}}\Phi_{n,m;y}(z,w;q)\Psi_y(z,w;q),
\end{equation}
where $n,m\in\mathbb{N}_0$.
Since $\Phi_{n,m;y}$ vanishes unless $y_1\leq n$ and $y_1+y_2\leq m$,
the sum over $y$ in \eqref{Eq_unit-BP} has finite support.

As in the proof in Section~\ref{Sec_KR}, let $a,k$ be integers such that
$a\leq k$. 
Then, by a $(k-a+1)$-fold application of \eqref{Eq_A2-chain} starting with
the root identity \eqref{Eq_unit-BP}, we obtain
\begin{align}\label{Eq_step1}
&\sum_{\substack{\la\subseteq (n^{k-a}) \\[1pt] \mu\subseteq (m^{k-a})}}
\prod_{i=1}^{k-a+1} \mathcal{K}_{\la_{i-1},\mu_{i-1};\la_i,\mu_i}(z,w;q) 
\\ &\qquad
=\sum_{y\in Q_{+}}
\big(z^{y_1}w^{y_1+y_2}q^{\frac{1}{2}(y_1^2+y_2^2+y_3^2)}\big)^{k-a+1}
\Phi_{n,m;y}(z,w;q)\Psi_y(z,w;q), \notag
\end{align}
where $\la_0:=n$, $\mu_0:=m$.
Next we use \eqref{Eq_Phi-nmy-11zw} to replace $\Phi_{n,m;y}(z,w;q)$
in the summand on the right by $\Phi_{n,m;y}(1,1,z,w;q)$ and define
\[
Z:=zq^{y_{12}} \quad\text{and}\quad W:=wq^{y_{23}}.
\]
Then, by an $(a-b)$-fold application of \eqref{Eq_Bailey-A2-uv-Phi}
where $(u,v)=(Z^{i-1},W^{i-1})$ in the $i$th step, as well as the use
of \eqref{Eq_zw-scaled} for $(a,b)=(1/q,1)$, we find
\begin{align*}
&\sum_{\substack{\la\subseteq(n^{k-b}) \\[1pt] \mu\subseteq(m^{k-b})}}
\prod_{i=1}^{k-b+1} q^{-\chi(i\leq a-b)\la_i} 
\mathcal{K}_{\la_{i-1},\mu_{i-1};\la_i,\mu_i}(z,w;q) \\
&=\sum_{y\in Q_{+}} 
\big(z^{y_1}w^{y_1+y_2}q^{\frac{1}{2}(y_1^2+y_2^2+y_3^2)}\big)^{k-b+1}
q^{-(a-b)y_1} \Phi_{n,m;y}\big(Z^{a-b},W^{a-b};z,w;q\big) \Psi_y(z,w;q),
\end{align*}
for integers $a,b,k$ such that $b\leq a\leq k$.
Again denoting this by $I_a$, it follows from Corollary~\ref{Cor_cd1-y}
that $(I_a-zwq^{m-1} I_{a-1})/(1-zwq^{-1})$ is given by
\begin{align}\label{Eq_step3}
&\sum_{\substack{\la\subseteq(n^{k-b}) \\[1pt] \mu\subseteq(m^{k-b})}}
\frac{1-zwq^{m+\la_{a-b}-1}}{1-zwq^{-1}}
\prod_{i=1}^{k-b+1} q^{-\chi(i\leq a-b)\la_i} 
\mathcal{K}_{\la_{i-1},\mu_{i-1};\la_i,\mu_i}(z,w;q) \\
&\quad=\sum_{y\in Q_{+}} \bigg( 
\big(z^{y_1}w^{y_1+y_2}q^{\frac{1}{2}(y_1^2+y_2^2+y_3^2)}\big)^{k-b+1}
q^{-(a-b)y_1} \notag \\
& \qquad\qquad\quad\times
\Phi_{n,m;y}\big(Z^{a-b},W^{a-b};1,1;z,w;q\big)\,
\frac{\Psi_y(z,w;q)}{\Delta_y(z,w;q)} \bigg).  \notag
\end{align}
Once again this holds for $b\leq a\leq k$ instead of the more restricted
range $b<a\leq k$ since \eqref{Eq_step3} for $b=a$ simplifies to
\eqref{Eq_step1} by $\la_0:=n$ and~\eqref{Eq_Phi-nmy-1111zw}.
The final iterative step in our proof is a $b$-fold 
application of Corollary~\ref{Cor_A2-tree-2}, where
\[ 
(u,v,c,d)=(Z^{a-b+i-1},W^{a-b+i-1},Z^{i-1},W^{i-1})
\]
in the $i$th step.
By \eqref{Eq_zw-scaled} for $a=b=1/q$ this yields
\begin{align}\label{Eq_before-symm}
&\sum_{\substack{\la\subseteq(n^k) \\[1pt] \mu\subseteq(m^k)}}
\frac{1-zwq^{\la_a+\mu_b-1}}{1-zwq^{-1}}\,
\prod_{i=1}^{k+1} q^{-\chi(i\leq a)\la_i-\chi(i\leq b)\mu_i} 
\mathcal{K}_{\la_{i-1},\mu_{i-1};\la_i,\mu_i}(z,w;q) \\
&\quad=\sum_{y\in Q_{+}} \bigg( 
\big(z^{y_1}w^{y_1+y_2}q^{\frac{1}{2}(y_1^2+y_2^2+y_3^2)}\big)^{k+1}
q^{-\sum_{i=1}^3 (\nu_i+i) y_i} \notag \\
& \qquad\qquad\quad\times
\Phi_{n,m;y}\big(Z^a,W^a;Z^b,W^b;z,w;q\big)\,
\frac{\Psi_y(z,w;q)}{\Delta_y(z,w;q)}\bigg), \notag
\end{align}
where we have used that $-ay_1-b(y_1+y_2)=-\sum_{i=1}^3(\nu_i+i)y_i$
for $\nu:=(a+b+2,b+1,0)$.
As for the analogous result \eqref{Eq_KR-finite} in the proof of 
the Kanade--Russell conjecture, this holds for all $0\leq a,b\leq k$.
Specifically, making the simultaneous substitutions
$(z,w,a,b,n,m)\mapsto (w,z,b,a,m,n)$, changing the summation 
indices $(y_1,y_2,y_3)\mapsto (-y_3,-y_2,-y_1)$ on the right and
$(\la,\mu)\mapsto(\mu,\la)$ on the left, it follows from
\eqref{Eq_before-symm} that the both sides are invariant under
the interchange of $a$ and $b$.

Taking the large-$n,m$ limit using \eqref{Eq_nmlimit-y}, using
definitions \eqref{Eq_Delta} and \eqref{Eq_Psi},
and eliminating $z$ and $w$ from the right-hand side in favour of
$x_1,x_2,x_3$, we obtain
\begin{align*}
&\sum_{\substack{\la,\mu\in\Par \\[1pt] l(\la),l(\mu)\leq k}}
\big(1-zwq^{\la_a+\mu_b-1}\big)
\prod_{i=1}^k \frac{z^{\la_i} w^{\mu_i}
q^{\la_i^2-\la_i\mu_i+\mu_i^2-\chi(i\leq a)\la_i-\chi(i\leq b)\mu_i}}
{(q;q)_{\la_i-\la_{i+1}}(q;q)_{\mu_i-\mu_{i+1}}} \\ 
&\:=\prod_{1\leq i<j\leq 3} \frac{1}{(x_i/x_j;q)_{\infty}}
\sum_{y\in Q_{+}} \bigg(
\prod_{1\leq i<j\leq 3} (x_i/x_j;q)_{y_{ij}}
\prod_{i=1}^3 x_i^{(k+1)y_i} q^{(k+1)\binom{y_i}{2}-\nu_iy_i} \\
&\qquad\qquad\qquad\qquad\qquad \times
\det_{1\leq i,j\leq 3}\big((x_iq^{y_i})^{\nu_i-\nu_j}\big)
\Phi_{y_1,y_1+y_2}\big(x_1q^{y_{12}}/x_2,x_2q^{y_{23}}/x_3;q^{-1}\big)
\bigg).
\end{align*}
Since, by $(a/q;q^{-1})_n=(aq^{-n};q)_n$,
\[
\Phi_{y_1,y_1+y_2}\big(x_1q^{y_{12}}/x_2,x_2q^{y_{23}}/x_3;q^{-1}\big) 
=\prod_{1\leq i<j\leq 3} \frac{1}{(x_i/x_j;q)_{y_{ij}}}
\prod_{i=1}^3 \frac{x_i^{y_i} q^{\binom{y_i}{2}} (x_i/x_3;q)_{y_i}}
{(qx_i/x_1;q)_{y_i}},
\]
this gives \eqref{Eq_bosonic}.
\end{proof}

As mentioned in the introduction, the $\mathrm{A}_1$-analogue of
Theorem~\ref{Thm_bosonic} was first proved by Andrews, who showed
that the right hand sides of \eqref{Eq_RS} and \eqref{Eq_RS-multi}
both satisfy
\begin{equation}\label{Eq_RS-eqn}
Q_{k,i}(z;q)-Q_{k,i-1}(z;q)=(zq)^{i-1} Q_{k,k-i+1}(zq;q)
\end{equation}
for $1\leq i\leq k$, where $Q_{k,0}:=0$.
Since both expressions satisfy the same initial conditions 
$Q_{k,i}(0;q)=Q_{k,i}(z;0)=1$, this proves the equality of
\eqref{Eq_RS} and \eqref{Eq_RS-multi}.
The equation \eqref{Eq_RS-eqn} may also be derived purely algebraically
using the theory of intertwining operators for vertex operator
algebras, see \cite{CLM06}.
For general $\mathrm{A}_{r-1}^{(1)}$ this approach has only been
completed fully for level-$1$ modules, see \cite[Theorem 5.3]{CLM10}.
Restricting to $r=3$, this yields 
\begin{subequations}\label{Eq_rec}
\begin{align}
\ch W'_{\La_0}(z,w;q)-\ch W'_{\La_0}(zq,w;q)
&=zq\,\ch W'_{\La_0}(zq^2,wq^{-1};q), \\
\ch W'_{\La_0}(z,w;q)-\ch W'_{\La_0}(q,wq;q)
&=wq\,\ch W'_{\La_0}(zq^{-1},wq^2;q),
\end{align}
\end{subequations}
where the exponents of $q$ in the argument of $\ch W'_{\La_0}$ on
the right are the Cartan integers of $\mathfrak{sl}_3$.
Together with
\[
\ch W'_{k\La_0}(z,w;q)=\ch W'_{k\La_1}(zq,w;q)=\ch W'_{k\La_2}(z,wq;q)
\]
for arbitrary level $k$ and
$\ch W'_{\La_0}(0,0;q)=\ch W'_{\La_0}(z,w;0)=1$, this uniquely determines
the characters $\ch W'_{\La_i}$ for $0\leq i\leq 2$.
It is routine to show that the right-hand side of \eqref{Eq_fermionic} for 
$k=1$ and $a=b=0$ satisfies \eqref{Eq_rec}.
The same cannot be said for the bosonic representation
\begin{align*}
&\ch W'_{\La_0}(z,w;q)=\frac{1}{(zq,wq,zwq;q)_{\infty}} \\
& \quad \times 
\sum_{r,s=0}^{\infty}\bigg(
(-1)^{r+s} z^{2r} w^{2s} q^{2r^2+2s^2-2rs+\binom{r}{2}+\binom{s}{2}} 
\frac{(1-z q^{2r-s})(1-wq^{2s-r})(1-z wq^{r+s})}
{(1-z)(1-w)(1-zw)}\\
&\qquad\qquad\;\times
\frac{(zw;q)_r(zw;q)_s(z;q)_{r-s}(w;q)_{s-r}}{(q;q)_r(q;q)_s} \bigg),
\end{align*}
for which showing \eqref{Eq_rec} holds requires a lengthy computation.
It would be very interesting to extend the approach using functional
equations to $\ch W'_{\uplambda}(z,w;q)$ for weights of arbitrary level.


\section{Outlook}\label{Sec_Outlook}

An important open question is how to generalise Theorems~\ref{Thm_KRistrue}
and \ref{Thm_Principal-subspace}
to $\mathrm{A}_{r-1}^{(1)}$ for all $r$.\footnote{The 
vacuum case $a=b=0$ of Theorem~\ref{Thm_bosonic} was generalised
to all $r$ in~\cite[Theorem 3.1]{FJM11} without the use of the
Bailey machinery.}
As far as the $\mathrm{A}_{r-1}$-analogue of the Bailey chains of 
Lemma~\ref{Lem_chain} and Theorem~\ref{Thm_A2-Bailey} is 
concerned, things are relatively straightforward.
Let $\boldsymbol{n}=(n_1,\dots,n_{r-1})$, 
$\boldsymbol{m}=(m_1,\dots,m_{r-1})$
be integer sequences and $\boldsymbol{z}=(z_1,\dots,z_{r-1})$ a 
sequence of indeterminates.
In \cite{W06} the definition of the rational function 
$\Phi_{n,m}(z,w;q)$ was extended to $\mathrm{A}_{r-1}$ as:
\begin{equation}\label{Eq_Phir}
\Phi_{\boldsymbol{n}}(\boldsymbol{z};q):=
\sum_{\la^{(1)},\dots,\la^{(r-1)}\in\Par}
\prod_{i=1}^{r-1} \prod_{l\geq 1} \frac{z_i^{\la^{(i)}_l}
q^{\frac{1}{2} \sum_{j=1}^{r-1} A_{ij}\la^{(i)}_l\la^{(j)}_l}}
{(q;q)_{\la^{(i)}_{l-1}-\la^{(i)}_l}},
\end{equation}
where $\la^{(i)}_0:=n_i$ and where $(A_{ij})_{1\leq i,j\leq r-1}$
is the Cartan matrix of $\mathrm{A}_{r-1}$.
For an arbitrary sequence $\boldsymbol{a}=(a_1,\dots,a_{r-1})$,
let $\bar{\boldsymbol{a}}:=(a_{r-1},\dots,a_1)$.
Replacing $\la^{(i)}$ by $\la^{(r-i)}$ in \eqref{Eq_Phir} it follows that
\begin{equation}\label{Eq_Z2-symmetry}
\Phi_{\boldsymbol{n}}(\boldsymbol{z};q)=
\Phi_{\bar{\boldsymbol{n}}}(\bar{\boldsymbol{z}};q).
\end{equation}
Another immediate consequence of the definition \eqref{Eq_Phir} is
the $\mathrm{A}_{r-1}$ Bailey chain
\begin{equation}\label{Eq_Ar-Bailey-chain}
\sum_{m_1=0}^{n_1}\dots\sum_{m_{r-1}=0}^{n_{r-1}}
\mathcal{K}_{\boldsymbol{n},\boldsymbol{m}}(\boldsymbol{z};q)
\Phi_{\boldsymbol{m}}(\boldsymbol{z};q)=
\Phi_{\boldsymbol{n}}(\boldsymbol{z};q),
\end{equation}
with $\mathcal{K}_{\boldsymbol{n},\boldsymbol{m}}$ given by
\[
\mathcal{K}_{\boldsymbol{n},\boldsymbol{m}}(\boldsymbol{z};q):=
\prod_{i=1}^{r-1} \frac{z_i^{m_i} 
q^{\frac{1}{2}\sum_{j=1}^{r-1} A_{ij} m_i m_j}}{(q;q)_{n_i-m_i}}. 
\]
Moreover, by Hua's identity \cite[Theorem 4.9]{Hua00} for
$\mathrm{A}_{r-1}$,
\begin{equation}\label{Eq_Hua-limit}
\lim_{n_1,\dots,n_{r-1}\to\infty} 
\Phi_{\boldsymbol{n}}(\boldsymbol{z};q)=
\frac{1}{(q;q)_{\infty}^{r-1}} \prod_{1\leq i<j\leq r}
\frac{1}{(z_i\cdots z_{j-1}q;q)_{\infty}}.
\end{equation}

The alternative expressions for $\Phi_n(z;q)$ and $\Phi_{n,m}(z,w;q)$ as
given in \eqref{Eq_Phi-zu} and \eqref{Eq_Phi} follow from
Corollaries~\ref{Cor_infinite-iteration-A1} and
\ref{Cor_infinite-iteration-A2}, or from \cite[Proposition 2.2]{FJM11}
which is based on the decomposition in the Gelfand--Zetlin basis of 
the Whittaker vectors for the quantum group $U_v(\mathfrak{gl}_r)$ over
$\mathbb{C}(v)$.
This more generally implies that 
\begin{align}\label{Eq_FJM}
\Phi_{\boldsymbol{n}}(\boldsymbol{z};q)&=
\sum \prod_{k\geq 1} \Bigg(
\prod_{i=1}^{r-1} 
\frac{(-1)^{\la^{(i)}_{k+1}} q^{\binom{\la^{(i)}_{k+1}}{2}}}
{(q;q)_{\la^{(i)}_k-\la^{(i)}_{k+1}}} 
\prod_{1\leq i<j\leq r} \bigg( z_{j-1}^{\la^{(i)}_{k+j-i}} 
q^{-(\la^{(i)}_{k+j-i}-\la^{(i)}_{k+j-i+1})\la^{(j)}_k} \\
& \qquad\qquad\quad\times
\frac{1-z_i\cdots z_{j-1}q^{\la^{(i)}_{k+j-i}-\la^{(j)}_k}}
{1-z_i\cdots z_{j-1}}\,
\frac{(z_i\cdots z_{j-1};q)_{\la^{(i)}_{k+j-i+1}-\la^{(j)}_k}}
{(z_i\cdots z_{j-1}q;q)_{\la^{(i)}_{k+j-i-1}-\la^{(j)}_k}}\bigg)\Bigg),
\notag
\end{align}
where the sum is over partitions $\la^{(1)},\dots,\la^{(r)}$ such that
$l(\la^{(i)})\leq r-i$ for $1\leq i\leq r$ (so that $\la^{(r)}=0$)
and $\la^{(i)}_1+\la^{(i-1)}_2+\dots+\la^{(1)}_i=n_i$ for $1\leq i\leq r-1$.
For $r=2$ this yields \eqref{Eq_Phi-zu} and for $r=3$ it gives
\begin{align}\label{Eq_rdrie}
\Phi_{n,m}(z,w;q)&=\frac{1}{(q,zq^{1-m};q)_n(q,wq;q)_m}\\[1mm]
& \quad \times
{_6W_5}\big(zq^{-m};q^{-m}/w,q^{-n},q^{-m};q,zwq^{n+m+1}\big).
\notag
\end{align}
By Jackson's ${_6W_5}$ summation \cite[Equation (II.20)]{GR04}
this simplifies to~\eqref{Eq_Phi}.
The expression \eqref{Eq_FJM} obscures the symmetry \eqref{Eq_Z2-symmetry},
although it can be simplified relatively easily to a $\binom{r-2}{2}$-fold
multisum that is symmetric.
For example, for $r=4$ two of the three summations can be carried out to
give an expression as a balanced ${_4\phi_3}$ basic hypergeometric series:
\begin{align*}
\Phi_{\boldsymbol{n}}(\boldsymbol{z};q)&=\frac{(z_1z_2q;q)_{n_1+n_2}
(z_2z_3q;q)_{n_2+n_3}}
{(q,z_1q,z_1z_2q;q)_{n_1} (q,z_2q,z_1z_2q,z_2z_3q;q)_{n_2} 
(q,z_3q,z_2z_3q;q)_{n_3}} \\[2mm]
&\quad \times
{_4\phi_3}\biggl[\genfrac{}{}{0pt}{}{q^{-n_2}/z_2,q^{-n_1},q^{-n_2},q^{-n_3}}
{q^{-n_1-n_2}/z_1z_2,q^{-n_2-n_3}/z_2z_3,z_1z_2z_3q};q,q\biggr].
\end{align*}
Regardless of how $\Phi_{\boldsymbol{n}}(\boldsymbol{z};q)$ is
expressed, it is an open problem to lift the $\mathrm{A}_{r-1}$
Bailey chain \eqref{Eq_Ar-Bailey-chain} to an $\mathrm{A}_{r-1}$
Bailey tree.
It follows from the work of Ardonne, Kedem and Stone
(see \cite[Equation (6.16)]{AKS06}) that the $1$ and $-q^{-1}$ in
$1-zwq^{\la_a+\mu_b-1}$ in formula \eqref{Eq_fermionic} --- this
factor can be traced back to the structure of the numerator of
\eqref{Eq_cd1} --- should be interpreted as entries of the inverse
of the matrix of generalised Kostka polynomials \cite{KS02,SW99} for
$\mathfrak{sl}_3$.
This suggests that the as-yet-to-be-discovered $\mathrm{A}_{r-1}$ Bailey 
tree involves the generalised Kostka polynomials for $\mathfrak{sl}_r$.
Another open problem is to find the $\mathrm{A}_{r-1}$-analogue of
the \eqref{Eq_Slater2} and~\eqref{Eq_seed}.
For $y=(y_1,\dots,y_r)\in Q$, let
\[
\Phi_{\boldsymbol{n};y}(\boldsymbol{z};q):=
\frac{\Phi_{\boldsymbol{m}}(\boldsymbol{w};q)}
{\prod_{1\leq i<j\leq r}(z_i\cdots z_{j-1}q;q)_{y_{ij}}},
\]
where $m_i:=n_i-y_1-\dots-y_i$ and $w_i:=z_iq^{y_{i,i+1}}$
for $1\leq i\leq r-1$.
The problem then is to find a manifestly positive representation for
the rational function $g_{\boldsymbol{n};\tau}(q)$ defined by
\[
g_{\boldsymbol{n};\tau}(q):=
\sum_{y\in Q} \Phi_{\boldsymbol{n};ry}
(\underbrace{q,\dots,q}_{r-1 \text{ times}};q)
\prod_{1\leq i<j\leq r} \frac{1-q^{ry_{ij}+j-i}}{1-q^{j-i}}
\prod_{i=1}^r q^{r(r+\tau)\binom{y_i}{2}-\tau iy_i},
\]
where $\boldsymbol{n}\in\mathbb{N}_0^{r-1}$ and 
$\tau\in\{2-r,\dots,0,1\}$.
For general $r$ this is a very hard problem since 
\begin{align*}
&\Phi_{n_1,\dots,n_{i-1},0,n_{i+1},\dots,n_{r-1};ry}
(\underbrace{q,\dots,q}_{r-1 \text{ times}};q) \\
&\quad =
\Phi_{n_1,\dots,n_{i-1}}(\underbrace{q,\dots,q}_{i-1 \text{ times}};q)\,
\Phi_{n_{i+1},\dots,n_{r-1}}(\!\underbrace{q,\dots,q}_{r-i-1 \text{ times}}\!;q)
\prod_{i=j}^r \delta_{y_j,0},
\end{align*}
which implies that
\begin{align*}
&g_{n_1,\dots,n_{i-1},0,n_{i+1},\dots,n_{r-1};\tau}(q) \\[1mm]
&\quad =
\Phi_{n_1,\dots,n_{i-1}}(\underbrace{q,\dots,q}_{i-1 \text{ times}};q)\,
\Phi_{n_{i+1},\dots,n_{r-1}}(\underbrace{q,\dots,q}_{r-i-1 \text{ times}};q).
\end{align*}
For example, setting $m=0$ in \eqref{Eq_g} gives
$g_{n,0;\tau}(q)=1/(q,q^2;q)_n=\Phi_n(q;q)$.
Some properties of $g_{\boldsymbol{n};\tau}(q)$ are
easily deduced for general $r$.
From \eqref{Eq_Hua-limit} followed by \eqref{Eq_Macdonald} it
immediately follows that
\begin{equation}\label{Eq_glimit}
\lim_{n_1,\dots,n_{r-1}\to\infty} g_{\boldsymbol{n};\tau}(q)
=\begin{cases}
\displaystyle \prod_{1\leq i<j\leq r} \frac{1}{(q^{j-i};q)_{\infty}}
& \text{if $\tau=1$}, \\[5mm]
\displaystyle \frac{(q;q)_{\infty}}{(q^r;q^r)_{\infty}}
\prod_{1\leq i<j\leq r} \frac{1}{(q^{j-i};q)_{\infty}}
& \text{if $\tau=0$}, \\[3mm]
0 & \text{if $\tau\in\{2-r,\dots,-1\}$}. 
\end{cases}
\end{equation}
We can do slightly better for special values of $\tau$.
First we note that by \cite[Equation (6.3)]{W06}\footnote{This result
is stated in \cite{W06} without proof.} it follows that for $r\geq 3$
\begin{align*}
\lim_{n_2,\dots,n_{r-2}\to\infty}
\Phi_{\boldsymbol{n};y}(\boldsymbol{z};q)&=
\frac{1}{(q;q)_{\infty}^{r-3}}\prod_{2\leq i<j\leq r-1}
\frac{1}{(z_i\cdots z_{j-1}q;q)_{\infty}} \\
&\quad\times
\frac{(z_1\cdots z_{r-1}q;q)_{n_1+n_{r-1}}}
{\prod_{i=1}^r (z_1\cdots z_{i-1}q;q)_{n_1-y_i}
(z_i\cdots z_{r-1}q;q)_{n_{r-1}+y_i}},
\end{align*}
generalising \eqref{Eq_Hua-limit}.
Hence, for such $r$,
\begin{align*}
g_{n,m;\tau}^{(r)}(q)&:=
\lim_{n_2,\dots,n_{r-2}\to\infty}
g_{(n,n_2,\dots,n_{r-2},m);\tau}(q) \\
&\hphantom{:}=
(q;q)_{\infty}
\prod_{1\leq i<j\leq r-1} \frac{1}{(q^{j-i};q)_{\infty}}
\prod_{i=1}^{r-1} \frac{1}{(q^{r-i};q)_{n+m+i}} \\
&\quad \times
\sum_{y\in Q}
\prod_{1\leq i<j\leq r} (1-q^{ry_{ij}+j-i})
\prod_{i=1}^r q^{r(r+\tau)\binom{y_i}{2}-\tau iy_i}
\qbin{n+m+r-1}{n-ry_i+i-1}.
\end{align*}
By \eqref{Eq_KG}, \eqref{Eq_tau-een} and $q\mapsto 1/q$ duality
this may be expressed in closed form for $\tau\in\{-1,0,1\}$ as
\begin{align*}
g_{n,m;\tau}^{(r)}(q)&=
q^{\binom{\tau}{2}(r-1)nm}
\qbin{n+m}{n}_p 
(q;q)_{\infty} 
\prod_{1\leq i<j\leq r-1} \frac{1}{(q^{j-i};q)_{\infty}}
\prod_{i=1}^{r-1} \frac{1}{(q^{r-i};q)_{n+m}},
\end{align*}
where $p=q$ if $\tau\in\{-1,1\}$ and $p=q^r$ if $\tau=0$.
For $r=3$ this is \eqref{Eq_seed}, and in the
limit of large $n$ and $m$ this gives \eqref{Eq_glimit}
for $\tau\in\{-1,0,1\}$.

\appendix

\section{New proof of \eqref{Eq_page692}}

We begin with the following $q$-Pfaff--Saalsch\"utz summation
for the root system $\mathrm{A}_{r-1}$:
\begin{align}\label{Eq_PS}
\sum_{y\in\mathbb{N}_0^r} &
\bigg( (b,q^{-N};q)_{\abs{y}}
\prod_{1\leq i<j\leq r} \frac{x_iq^{y_i}-x_jq^{y_j}}{x_i-x_j}
\prod_{i,j=1}^r \frac{(a_jx_i/x_j;q)_{y_i}}{(qx_i/x_j;q)_{y_i}} \\
& \times
\prod_{i=1}^r \frac{(bq^{1-N}/cx_i;q)_{\abs{y}-y_i} q^{y_i}}
{(a_ibq^{1-N}/cx_i;q)_{\abs{y}}(cx_i;q)_{y_i}}\bigg)
=\prod_{i=1}^r \frac{(cx_i/a_i,cx_i/b;q)_N}{(cx_i,cx_i/a_ib;q)_N}, 
\notag
\end{align}
where $N$ is a nonnegative integer and $\abs{y}:=y_1+\dots+y_r$.
It should be noted that the summand vanishes unless $\abs{y}\leq N$
so that only finitely many terms contribute to the sum.
The result \eqref{Eq_PS} was first obtained in the appendix of a
preliminary version of Leininger and Milne's paper~\cite{LM99}; 
an appendix that was dropped in the published version.
Subsequently \eqref{Eq_PS} was rederived and published by Bhatnagar
and Schlosser, see \cite[Remark~5.11]{BS98}.

To obtain \eqref{Eq_page692}, we replace $q\mapsto q^r$ in \eqref{Eq_PS}
and then specialise $x_i=q^{r-i}bz/c$ and $a_i=q^{-n}$ for $n$ a nonnegative
integer. 
Using $\prod_{i=1}^r (aq^{r-i};q)_k=(a;q)_{rk}$, this gives
\[
\prod_{i,j=1}^r \frac{(a_jx_i/x_j;q)_{y_i}}{(qx_i/x_j;q)_{y_i}}
\mapsto \prod_{i=1}^r \frac{(q^{-n-i+1};q)_{ry_i}}{(q^{r-i+1};q)_{ry_i}},
\]
so that the resulting summand vanishes unless $0\leq ry_i\leq n+i-1$.
Since this is independent of $N$, $q^{-rN}$ may be replaced by the
indeterminate $d$, resulting in
\begin{align*}
\sum_{y\in\mathbb{N}_0^r}&
\frac{(b,d;q^r)_{\abs{y}}}{(dq^{1-n}/z;q)_{r\abs{y}}}
\prod_{1\leq i<j\leq r} \frac{1-q^{ry_{ij}+j-i}}{1-q^{j-i}} 
\prod_{i=1}^r \frac{(q^{-n-i+1};q)_{ry_i}
(dq^i/z;q^r)_{\abs{y}-y_i} q^{riy_i}}
{(q^{r-i+1};q)_{ry_i}(bzq^{r-i};q^r)_{y_i}} \\ 
&=\frac{(z,bz/d;q)_n}{(bz,z/d;q)_n},
\end{align*}
where the reader is reminded that $y_{ij}:=y_i-y_j$.
Indeed, after multiplying the above identity by $d^n (z/d;q)_n$
and carrying out some standard simplifications of the $q$-shifted
factorials involving $d$, it follows that both sides are polynomials
in $d$ of degree $n$.
Since the difference between the right- and left-hand side is zero
for $d=q^{-rN}$ where $N$ is an arbitrary nonnegative integer, this
difference is zero for all $d$.
Next, if we set $b=0$, let $d$ tend to infinity and carry out some
elementary manipulations, we find
\begin{align}\label{Eq_before-rotation}
\sum_{y\in\mathbb{N}_0^r}& 
\big((-1)^rz\big)^{\abs{y}} q^{-r\binom{\abs{y}}{2}}
\prod_{1\leq i<j\leq r} (1-q^{ry_{ij}+j-i}) 
\prod_{i=1}^r q^{\binom{r+1}{2}y_i^2-iy_i} 
\qbin{n+r-1}{n-ry_i+i-1} \\ 
&=(z;q)_n \prod_{i=1}^{r-1} (1-q^{n+i})^i. \notag
\end{align}
We now consider the sum over the $y_i$ for fixed $\abs{y}=m$ and
carry out what in \cite{GK97} is referred to as the rotation trick.
That is, if $u,v$ are the unique integers such that $m=ur+v$ for
$0\leq v<r$, $u\geq 0$, then we shift and rotate the summation indices
$y_1,\dots,y_r$ as
\[
y_i\mapsto 
\begin{cases} 
y_{i+v}+u & \text{for $1\leq i\leq r-v$}, \\
y_{i+v-r}+u+1 & \text{for $r-v<i\leq r$}.
\end{cases}
\]
This substitution leads to the following alternative expression for the
left-hand side of \eqref{Eq_before-rotation}:
\[
\sum_{m=0}^n \sum_{y\in Q}
(-z)^m q^{\binom{m}{2}} \prod_{1\leq i<j\leq r} (1-q^{ry_{ij}+j-i})
\prod_{i=1}^r q^{\binom{r+1}{2}y_i^2-iy_i} \qbin{n+r-1}{n-m-ry_i+i-1}.
\]
Equating coefficients of $z^m$ with the right-hand side of 
\eqref{Eq_before-rotation} using the $q$-binomial theorem
\begin{equation}\label{Eq_qbt}
(z;q)_n=\sum_{m=0}^n (-z)^m q^{\binom{m}{2}}\qbin{n}{m},
\end{equation}
this implies
\begin{equation}\label{Eq_tau-een}
\sum_{y\in Q}
\prod_{1\leq i<j\leq r} (1-q^{ry_{ij}+j-i})
\prod_{i=1}^r q^{\binom{r+1}{2}y_i^2-iy_i}
\qbin{n+r-1}{n-m-ry_i+i-1}
=\qbin{n}{m} \prod_{i=1}^{r-1} (1-q^{n+i})^i.
\end{equation}
Finally, replacing $n$ by $n+m$ and specialising $r=3$ yields
\eqref{Eq_page692}.

\begin{ack}
I would like to thank Shashank Kanade for many helpful discussions,
Michael Schlosser for sending me to the preprint version of the paper
\cite{LM99} by Leininger and Milne which contains the key identity
\eqref{Eq_PS}, and Shunsuke Tsuchioka for sharing his
conjecture~\eqref{Eq_ST-conjecture}.
Thanks are also due to the anonymous referee for their careful
reading of the paper and for their helpful suggestions to improve
the presentation.
\end{ack}

\begin{funding}
Work supported by the Australian Research Council Discovery Project
DP200102316.
\end{funding}

\end{document}